\documentclass[oneside,11pt,a4paper]{amsart}

\usepackage{amsmath,yhmath}
\usepackage{amssymb}
\usepackage{amsfonts}
\usepackage{bbm} 
\usepackage{mathtools}
\usepackage{tikz}
\usepackage{tikz-cd}
\usepackage[inline]{enumitem}
\usepackage[mathscr]{eucal} 
\allowdisplaybreaks
\usepackage{stmaryrd} 
\usepackage[all]{xy}

\numberwithin{equation}{section}
\usepackage{caption}
\usepackage[margin=1in,marginparwidth=0.8in, marginparsep=0.1in]{geometry}
\usepackage{epigraph}

\usepackage{tikz, tikz-cd,pgfplots}
\usetikzlibrary{decorations.markings, intersections, arrows.meta, positioning}
\tikzset{ Ap/.style={decoration={markings, mark=between positions 0 and 1 step 3mm with{#1}}, postaction={decorate}} }

\pgfplotsset{compat=1.18}

\usepackage{aliascnt}
\usepackage[colorlinks, linkcolor=black,citecolor=blue, urlcolor=cyan, unicode, bookmarksnumbered]{hyperref}

\newcommand*{\eeqref}[2][Equation~]{%
  \hyperref[{#2}]{#1(\ref*{#2})}%
}

\newtheoremstyle{note}
{\topsep}
{\topsep}
{}
{0pt}
{\bfseries}
{.}
{0.5em }
{}
\theoremstyle{plain}

\newtheorem{Thm}{Theorem}[section]

\newtheorem*{Thm*}{Theorem}

\newaliascnt{prop}{Thm}
\newtheorem{Prop}[prop]{Proposition}
\aliascntresetthe{prop}

\newaliascnt{lemma}{Thm}
\newtheorem{Lemma}[lemma]{Lemma}
\aliascntresetthe{lemma}

\newaliascnt{coro}{Thm}
\newtheorem{Coro}[coro]{Corollary}
\aliascntresetthe{coro}

\newaliascnt{conjecture}{Thm}
\newtheorem{Conjecture}[conjecture]{Conjecture}
\aliascntresetthe{conjecture}

\newtheorem{Lemma*}{Lemma}
\newtheorem*{Prop*}{Proposition}
\newtheorem*{Coro*}{Corollary}

\theoremstyle{definition}

\newaliascnt{def}{Thm}
\newtheorem{Def}[def]{Definition}
\aliascntresetthe{def}

\newtheorem*{Def*}{Definition}

\newaliascnt{Eg}{Thm}
\newtheorem{eg}[Eg]{Example}
\aliascntresetthe{Eg}

\newtheorem*{eg*}{Example}

\theoremstyle{remark}

\newaliascnt{rmk}{Thm}
\newtheorem{RMK}[rmk]{Remark}
\aliascntresetthe{rmk}

\newtheorem*{RMK*}{Remark}

\theoremstyle{plain}

\newtheorem{iThm}{Theorem}[section]

\newtheorem*{iThm*}{Theorem}

\newaliascnt{iprop}{iThm}

\aliascntresetthe{iprop}

\newaliascnt{iCoro}{iThm}

\aliascntresetthe{iCoro}

\newaliascnt{iconjecture}{iThm}

\aliascntresetthe{iconjecture}

\theoremstyle{definition}

\newaliascnt{idef}{iThm}

\aliascntresetthe{idef}

\newtheorem*{iDef*}{Definition}

\newaliascnt{iEg}{iThm}

\aliascntresetthe{iEg}

\newtheorem*{ieg*}{Example}

\theoremstyle{remark}

\newaliascnt{irmk}{iThm}

\aliascntresetthe{irmk}

\newtheorem*{iRMK*}{Remark}

\newcommand{\arbR} { \overrightarrow{\mathbb{R}}}
\newcommand{\arbRpt} {{ \overrightarrow{\mathbb{R}}_*}}

\newcommand{\bA} { {\mathbb{A}}}
\newcommand{\bC} { {\mathbb{C}}}
\newcommand{\bD} { {\mathbb{D}}}

\newcommand{\bP} { {\mathbb{P}}}
\newcommand{\bQ} { {\mathbb{Q}}}
\newcommand{\bT} { {\mathbb{T}}}
\newcommand{\bZ} { {\mathbb{Z}}}
\newcommand{\bR} { {\mathbb{R}}}

\newcommand{\bS} { {\mathbb{S}}}
\newcommand{\bN} { {\mathbb{N}}}

\newcommand{\bfk} { {\mathbf{k}}}
\newcommand{\cA} { {\mathcal{A}}}

\newcommand{\cC} { {\mathcal{C}}}
\newcommand{\cD} { {\mathcal{D}}}
\newcommand{\cE} { {\mathcal{E}}}

\newcommand{\cY} { {\mathcal{Y}}}

\newcommand{\sT}{ \mathscr{T} }
\newcommand{\supp}{{\textnormal{supp}}}

\newcommand{\HOM}{{ \textnormal{Hom}  }}

\newcommand{\pt}{{\textnormal{pt}}}

\newcommand{\nT}{{\textnormal{T}}}

\newcommand{\id}{{\operatorname{id}}}
\newcommand{\Mod}{{\operatorname{Mod}}}
\newcommand{\Moda}{{\operatorname{aMod}}}
\newcommand{\Perf}{\operatorname{Perf}}
\newcommand{\QCoh}{{\operatorname{QCoh}}}
\newcommand{\QCoha}{{\operatorname{aQCoh}}}
\newcommand{\Spec}{{\operatorname{Spec}}}
\newcommand{\Sp}{\operatorname{Sp}}

\newcommand{\Sh}{{\operatorname{Sh}}}
\newcommand{\coSh}{{\operatorname{coSh}}}
\newcommand{\PSh}{{\operatorname{PSh}}}
\newcommand{\Loc}{{\operatorname{Loc}}}
\newcommand{\Fun}{{\operatorname{Fun}}}
\newcommand{\PrSt}{{\operatorname{Pr_{st}^L}}}
\newcommand{\PrStR}{{\operatorname{Pr_{st}^R}}}
\newcommand{\Cat}{ {{\operatorname{Cat}}}}

\newcommand{\CatPerf}{{\operatorname{Cat^{\operatorname{Perf}}}}}
\newcommand{\Catdual}{{\operatorname{Cat^{\operatorname{dual}}}}}
\newcommand{\PrStCG}{{\operatorname{Pr_{st,\omega}^{L}}}}
\newcommand{\Compass}{\operatorname{CompAss}}
\newcommand{\Stab}{{\rm Stab}}
\newcommand{\cont}{{\operatorname{cont}}}
\newcommand{\dual}{{\operatorname{dual}}}

\newcommand{\CAlg}{\operatorname{CAlg}}

\newcommand{\op}{{\operatorname{op}}}
\renewcommand{\SS}{{\operatorname{SS}}}
\newcommand{\coSS}{{\operatorname{coSS}}}

\newcommand{\fib}{{\operatorname{fib}}}
\newcommand{\cofib}[1]{{\operatorname{cofib}\left( #1 \right)}}
\newcommand{\Ind}{{\operatorname{Ind}}}

\newcommand{\Fuk}{{\operatorname{Fuk}}}
\newcommand{\Nov}{{\operatorname{Nov}}}
\newcommand{\Rep}{{\operatorname{Rep}}}
\newcommand{\Pers}{{\operatorname{Pers}}}
\newcommand{\Rees}{{\operatorname{Rees}}}
\renewcommand{\mathring}{{\operatorname{Int}\,}}
\newcommand{\dint}{{\operatorname{int}}}

\title[Almost mathematics, Persistence modules, and Tamarkin category]{Almost mathematics, Persistence modules, \\and Tamarkin category}
\author{Tatsuki Kuwagaki}
\author{Bingyu Zhang}

\begin{document}
\begin{abstract}
We give a precise unification of three theories that are widely used by symplectic geometers: (Almost) modules over the Novikov ring, Persistence modules, and the Tamarkin category. Our method provides new input in this direction, especially in relation to Vaintrob's Novikov/log-perfectoid mirror symmetry for Novikov toric schemes. The results of this paper can also be treated as a study of persistent homology from a higher algebra point of view.

As applications, we establish a version of homological mirror symmetry over the Novikov ring for toric varieties and propose a conjecture for homological mirror symmetry over the Novikov ring for log Calabi-Yau varieties.
\end{abstract}

\maketitle

\epigraph{Advanced Package Tool, or APT, is a free-software user interface that works with core libraries to handle the installation and removal of software on Debian and Debian-based Linux distributions.}{Wikipedia}

\tableofcontents

\section{Introduction}

There are three parallel algebraic theories widely used in symplectic geometry: modules over the Novikov ring, persistence modules, and the Tamarkin category. Each of them packages $\mathbb R$-filtered complexes in a different language. 

\begin{enumerate}[fullwidth]
    \item \textbf{Modules over the Novikov ring}, which is an indispensable language in Morse--Floer theory to encode symplectic energy. Moreover, a general Floer theory is only defined over the (completed) Novikov ring.
    \item \textbf{Persistence modules}, which were first invented in the study of mathematical data science, now many important ideas are imported to symplectic geometry.
    \item \textbf{The Tamarkin category}, the key object used to study non-conic, quantitative symplectic geometry by using microlocal sheaf theory. The Tamarkin category is also essential in the formulation of irregular Riemann-Hilbert correspondence of D'Agnolo--Kashiwara. 
\end{enumerate}
A precise understanding of the relationships between them in neat/sophisticated language is required for recent developments of symplectic geometry: As indicated in the above, many important ideas (such as interleaving distances, barcodes, etc.) are imported from the study of persistence modules to symplectic geometry by using the translations between 1 and 2. The Fukaya--Sheaf comparison, or the sheaf-theoretic study of symplectic topology, is largely motivated by the translation between (1) and (3), and then the ideas of (2) are also imported to (3) via the translations between them. 

However, the unified, precise, and clear understanding of the relationships between them is currently lacking in the literature. \textbf{The primary goal of this paper is to give a precise and natural comparison among these theories, including their higher-dimensional analogues, and an application to mirror symmetry.}

The key ingredient for these relations is  almost mathematics, originally invented by Faltings \cite{p_adic_Hodge_Faltings} and developed further by Gabber and Ramero \cite{Gabber_Ramero}. The role of almost mathematics for this correspondence was first observed by Vaintrob \cite{Vaintrob-logCCC} and independently by the first-named author \cite{Kuwagaki_almost_equivalence}, and its role for symplectic geometry was first observed by Fukaya \cite{fukaya_distance}. Therefore, we call the natural identification the (1-dimensional) APT (short for Almost-Persistence-Tamarkin) correspondence. 

In fact, Vaintrob claims a more general version of the APT correspondence, which he calls Novikov/log-perfectoid mirror symmetry in \cite{Vaintrob-logCCC}, and we call it the global APT correspondence. We will explain in Subsection \ref{subsection: intro 1d APT} why the correspondence gives the expected comparison between the objects mentioned above.

\medskip

Usual homological mirror symmetry results are stated as equivalences between two categories: one coherent-sheaves-type category and one Fukaya-type category. For toric varieties, a version of homological mirror symmetry, also known as the Coherent-Constructible correspondence (CCC), which has been investigated by many authors including (but not limited to) \cite{FLTZ11,FLTZ12,FLTZ14,KuwagakiCCC,Zhou_2019,ToricMSrevisited-Shende} after the pioneering work of Bondal \cite{bondal2006derived}, states that for a nice enough toric variety $X$, there exists a conic Lagrangian $\Lambda\subset T^*T^n$ and a categorical equivalence 
\[\operatorname{Coh}(X)\simeq \Sh_{\Lambda}(T^n),\]
where the latter is a category of constructible sheaves whose singularities are controlled by the conic Lagrangian set $\Lambda$ in terms of microlocal theory of sheaves. The latter is a Fukaya-type category in light of \cite{GPS3}: $\Sh_{\Lambda}(T^n)\simeq \mathscr{W}(T^*T^n,\Lambda)$, where the latter is the wrapped Fukaya category of $T^*T^n$ with the stop $\Lambda$. 

An interesting question is what the mirror of $\Sh(T^n)$, the category of sheaves on $T^n$ without any constraint, is. As $\Sh(T^n)$ is equivalent to $\Sh^{\bZ^n}(\bR^n)$, the category of $\mathbb Z^n$-equivariant sheaves on $\bR^n$, we may also ask what the mirror of $\Sh(\bR^n)$ is, and then the mirror for $\Sh(T^n)$ follows by taking the homotopy fixed point. 

Notice that since $\Sh(\bR^n)$ has many non-constructible objects, one should not expect that there is a scheme/stack $Y$ of ``finite type'' such that $\QCoh(Y)$ is equivalent to $\Sh(\bR^n)$. Nevertheless, there is a categorical obstruction: normally $\Sh(\bR^n)$ does not have enough compact objects, while $\QCoh(Y)$ has plenty. To address the issue, we shall consider an infinite type stack $Y$ and a certain quotient of $\QCoh(Y)$, say, the category of almost quasi-coherent sheaves $\QCoha(Y)$ (strictly speaking, more data are needed to be specified in a moment), which has been systematically studied in almost mathematics.

The global APT correspondence then states that the category of all sheaves on $\bR^n$ is equivalent to almost quasi-coherent sheaves on $Y$
\[ \QCoha(Y)\simeq \Sh(\bR^n).\]

\subsection{The global APT correspondence: statement}
In this subsection, we explain the precise statement for the global APT correspondence.

Let $\Sigma$ be a fan, which is a set of proper closed (possibly non-rational) polyhedral cones in a real vector space $N_\bR$ that is closed under taking faces of cones and taking intersections of cones (see \autoref{def: fan}). In the introduction, we take $\bfk$ to be a commutative (discrete) ring\footnote{In the main part of the article, we will consider commutative ring spectra in most of the results}.

\begin{enumerate}[fullwidth]
  \item[\textbf{Symplectic geometry}:] Let $M_\bR$ be the dual vector space of $N_\bR$. We identify $T^*M_\bR=M_\bR\times N_\bR$. Notice that we should regard the fan $\Sigma$ as a subset of cotangent fibers $N_\bR=T_v^*M_\bR$ for $v\in M_\bR$, and then the support of the fan $|\Sigma|=\cup_{\sigma \in \Sigma} \sigma \subset N_\bR$ is a subset of cotangent fibers.
    
Consider the category $\Sh(M_\bR)$ of sheaves on the real vector space $M_\bR$. In \cite{KS90}, Kashiwara and Schapira associated a conic closed subset $\SS(F)\subset T^*M_\bR$ to a sheaf $F\in \Sh(M_\bR)$, which is called the microsupport of $F$. It measures how sections of $F$ propagate along certain (co)directions. 

\medskip

The symplectic side category we are considering is \textit{the category of sheaves microsupported in the fan support $|\Sigma|$}, i.e., the following full subcategory of $\Sh(M_\bR)$,
    \[\Sh_{M_\bR\times |\Sigma|}(M_\bR)\coloneqq\{F\in \Sh(M_\bR) \colon \SS(F)\subset M_\bR\times |\Sigma|\subset T^*M_\bR\}.\]

    These categories are deeply related to Fukaya categories, which we will explain later (for example, \autoref{Stupid sheaf-fukaya correspondence} and \autoref{Fukaya-SheafcomparsionLG}).

    \item[\textbf{Algebraic geometry}:] 
    
    On the algebraic side, the geometric object is a so-called Novikov toric scheme $X_\Sigma^\Nov$ associated to the fan $\Sigma$, which is a non-Noetherian analogue of the usual definition of a toric variety associated with $\Sigma$. 
    
    The category we are considering is the category of almost quasi-coherent sheaves on the Novikov toric scheme, which relies on a closed subscheme $\partial_\Sigma \rightarrow X_\Sigma^\Nov$ that is defined by an idempotent ideal sheaf. 
    
    We start from the definition of Novikov toric schemes and the closed subscheme. To begin with, we recall the definition of the usual toric varieties: we glue the spectra of the monoid algebra of lattice points of $\sigma^\vee \subset M_\bR$ (in a fixed lattice in $M_\bR$) for each cone $\sigma\in \Sigma$. 
    
    To define  $X_\Sigma^\Nov$, the difference here is that, rather than considering integer points, we consider the cone $\sigma^\vee \subset M_\bR$ itself as a monoid, and consider its monoid algebra ${\bfk}[\sigma^\vee]$. Moreover, we consider the idempotent ideal ${\bfk}[\mathring{\sigma^\vee}]$ of ${\bfk}[\sigma^\vee]$ given by  interior points of the cone. We set
    \[\Lambda_{\sigma^\vee}\coloneqq{\bfk}[\sigma^\vee],\quad I_{\sigma^\vee}\coloneqq{\bfk}[\mathring{\sigma^\vee}].\]
    
     Passing to spectra, we set
    \[X_{\sigma}^\Nov\coloneqq \Spec(\Lambda_{\sigma^\vee}), \quad \partial_\sigma\coloneqq V(I_{\sigma^\vee}) \rightarrow X_{\sigma}^\Nov.\]

    Then we glue them along the fan $\Sigma$ to obtain a scheme $X_\Sigma^\Nov$ and a closed subscheme $\partial_\Sigma \rightarrow X_\Sigma^\Nov$. Notice that the ``Novikov torus'' $\bT^\Nov=\Spec(\bfk[M_\bR])$ acts on $X_\Sigma^\Nov$ since  all the rings $\Lambda_{\sigma^\vee}$ are $M_\bR$($\cong \bR^n$)-graded. We refer to \autoref{subsection: def Novikov toric scheme} for more details on the definition.

\medskip

Now, \emph{the category of $\bT^\Nov$-equivariant almost quasi-coherent sheaves} 
\[\QCoha_{\bT^\Nov}(X_\Sigma^\Nov,\partial_\Sigma):=\QCoh_{\bT^\Nov}(X_\Sigma^\Nov)/\QCoh_{\bT^\Nov}^{\partial_\Sigma}(X_\Sigma^\Nov)\] could be described as the Verdier quotient of the category of $\bT^\Nov$-equivariant quasi-coherent sheaves $\QCoh_{\bT^\Nov}(X_\Sigma^\Nov)$ by the full subcategory $\QCoh_{\bT^\Nov}^{\partial_\Sigma}(X_\Sigma^\Nov)$ consisting of $\bT^\Nov$-equivariant quasi-coherent sheaves that are schematically supported on ${\partial_\Sigma}$. 

On the affine piece $X_{\sigma}^\Nov$, the category $\QCoha_{\bT^\Nov}(X_\sigma^\Nov,\partial_\sigma)$ is tautologically identified with a category $\Moda _{\bR^n-gr} (\Lambda_{\sigma^\vee},I_{{\sigma^\vee}})$ defined in the same way, i.e., the category of $\bR^n$-graded modules over $\Lambda_{\sigma^\vee}$ quotiented by the subcategory of modules annihilated by $I_{\sigma^\vee}$. 

We refer to \autoref{subsection: def almost qcoh} for more discussion on the definition; notice that the definition there is slightly different from the one here, but equivalent.

\end{enumerate}

\medskip

In \cite{Vaintrob-logCCC}, Vaintrob states the following categorical equivalences and some of their variants regarding different equivariance conditions on sheaves, which we refer to \autoref{theorem: Novikov mirror symmetry} and \autoref{theorem: root stack-Novikov mirror symmetry-rational} for their precise versions in this article:
\begin{iThm}[Global APT correspondence]\label{itheorem: Global apt 1}For a fan $\Sigma$, we have an equivalence
\[\QCoha_{\bT^\Nov}(X_\Sigma^\Nov,\partial_\Sigma) \simeq \Sh_{M_\bR\times |\Sigma|}(M_\bR).\]   
\end{iThm}

To prove Theorem \ref{itheorem: Global apt 1}, a form of Theorem \ref{iThm: sheaves with fan support} in a different formulation is implicit in Vaintrob's approach, but its proof requires the microlocal cut-off argument developed here.

\begin{iThm}\label{iThm: sheaves with fan support}For a fan $\Sigma$, we have an equivalence
\[ \Sh_{M_\bR\times |\Sigma|}(M_\bR) \simeq \varinjlim_{\sigma\in \Sigma } \Sh_{M_\bR\times \sigma}(M_\bR),\] 
where the colimit is taken in $\PrSt$, the category of presentable stable categories and left adjoint functors.
\end{iThm}

We can deduce Theorem \ref{itheorem: Global apt 1} from Theorem \ref{iThm: sheaves with fan support} plus the following two independent observations\footnote{In fact, by the two observations, we also see that Theorem \ref{iThm: sheaves with fan support} can be deduced from Theorem \ref{itheorem: Global apt 1} as well, i.e., they are equivalent.}, (a) and (b):
\begin{enumerate}[fullwidth,label=(\alph*)]
    \item First, in \autoref{subsection: def Novikov toric scheme}, we can describe $\QCoha$ by Zariski descent: we have a limit formula in $\Cat_\infty$, say
\[\QCoha_{\bT^\Nov}(X_\Sigma^\Nov,\partial_\Sigma)\simeq  \varprojlim_{\sigma\in \Sigma}\QCoha_{\bT^\Nov}(X_\sigma^\Nov,\partial_\sigma)=\varprojlim_{\sigma\in \Sigma}\Moda_{M_\bR-gr} (\Lambda_{\sigma^\vee},I_{{\sigma^\vee}}).\] 

Notice that all functors in the limit diagram are right adjoints and all categories are presentable. Therefore, after passing to their left adjoints, this is a colimit in $\PrSt$.

\item 
Second, using the microlocal cut-off lemma for cones, which was originally proven by Kashiwara and Schapira, we can describe the affine piece $\QCoha_{\bT^\Nov}(X_\sigma^\Nov,\partial_\sigma)$: For a cone $\sigma\in \Sigma$, we have an equivalence
\[\Sh_{M_\bR\times \sigma} (M_\bR) \simeq \Moda_{{M_\bR-gr}} (\Lambda_{\sigma^\vee},I_{{\sigma^\vee}})=\QCoha_{\bT^\Nov}(X_\sigma^\Nov,\partial_\sigma);\]
and for $\sigma \subset \gamma$ in $\Sigma$, under the identification, the inclusion functors on the sheaf side (induced by inclusions of microsupport bounds) are left adjoints to the corresponding restriction functors on $\QCoha$-categories. We refer to \autoref{prop: APT HD} for more details of the local result.
\end{enumerate}

However, in \cite{Vaintrob-logCCC}, the proof of the corresponding version of Theorem \ref{iThm: sheaves with fan support} for non-complete fans $\Sigma$ (we say a fan is complete if $|\Sigma|= \bR^n$) is missing; moreover, for complete fans, the argument does not successfully establish the required fully faithfulness in Theorem \ref{iThm: sheaves with fan support}.

Building on Vaintrob's pioneering work, the main technical contribution of this article is a complete proof of Theorem \ref{iThm: sheaves with fan support}. Combined with the algebraic descent statement above, this gives a proof of the global APT correspondence, Theorem \ref{itheorem: Global apt 1}. We also refine the correspondences: the results are formulated in the language of
$\infty$-categories, and the monoidality of the equivalences is established.

Our proof of Theorem \ref{iThm: sheaves with fan support} contains some ingredients that could be of independent interest: for example, we prove a cut-off result \autoref{coro: cut-off mult cones} for non-polyhedral cones, which is similar to \cite[Proposition 3.1.10]{guillermou2019sheaves}.

\begin{RMK}The main input of Theorem \ref{iThm: sheaves with fan support} for complete fans is a limit formula for the $\star$-convolution unit, \autoref{lemma: star convolution limit}, which also appears in \cite[Theorem 3.7]{FLTZ11}, \cite[Lemma 10.1]{KuwagakiCCC} or \cite[Proposition 4.5.4]{bai2025toricmirrorsymmetryhomotopy} (they require the smoothness or the projectivity for other purposes, which is irrelevant for us). For incomplete fans, we implement the effect of microsupport for the limit formula using a non-characteristic deformation argument.
\end{RMK}

\begin{RMK}
Here, we explain how a version of Theorem \ref{iThm: sheaves with fan support} has implicitly appeared in \cite{Vaintrob-logCCC}, and where our proof supplies the missing input.

In \cite[Section 5.2]{Vaintrob-logCCC}, Vaintrob introduced a site, which he called the Dolan topology, for each fan $\Sigma$. On the other hand, for each closed cone $\gamma\subset \bR^n$, Kashiwara and Schapira introduced a topology on $\bR^n$ which is denoted by $\bR^n_\gamma$ and will be reviewed in \autoref{Section: microsupport}.

By definition of the Dolan topology and by our proof of \autoref{prop: APT HD}, we can see that the Dolan topology coincides with the topology generated by certain bases of $\bR^n_{-\sigma^\vee}$ for all $\sigma\in \Sigma$. Therefore, by standard topos theory results, the category of sheaves on the Dolan topology is nothing but $\varprojlim_{\sigma}\Sh(\bR^n_{-\sigma^\vee})$, which can be identified with $\varprojlim_{\sigma\in \Sigma } \Sh_{\bR^n\times \sigma}(\bR^n)$ using the microlocal cut-off lemma. Passing to adjoints, we obtain a colimit in $\PrSt$.

Therefore, \cite[Section 6]{Vaintrob-logCCC} can be viewed as an attempt to prove a version of Theorem \ref{iThm: sheaves with fan support} using the Dolan topology. The suggested argument for the essential surjectivity has the correct form, but a crucial computation is missing compared with our argument in \autoref{subsection: cut-off for fans}.  On the other
hand, the suggested compact-generation argument for fully faithfulness cannot be applied to the categories involved, and therefore does not establish the fully faithfulness needed for Theorem \ref{iThm: sheaves with fan support}. In this paper, we establish this point using the microlocal cut-off lemma for cones.

\end{RMK}

\subsection{Sample results and applications}

\subsubsection{1-dimensional APT correspondence}\label{subsection: intro 1d APT}

We specialize to the fan $\Sigma=\{\{0\},[0,\infty)\}$ in $\bR$. Here, we explain why Theorem \ref{itheorem: Global apt 1} recovers the APT correspondence as suggested by the title of the article. We will recap the results and state some variants that are useful in quantitative symplectic geometry.

\medskip

We start by unwrapping the data for categories.

\begin{itemize}[fullwidth]
    \item \textit{Novikov ring and its almost modules.}

On the algebraic side, we only have two rings
\[\Lambda_{\geq }\coloneqq \Lambda_{[0,\infty)}={\bfk}[\bR_{\geq 0}],\quad \Lambda_{\bR }={\bfk}[\bR].\]

In the symplectic literature, the ring $\Lambda_{\geq }$ is called the \textit{polynomial Novikov ring}. The idempotent ideal of $\Lambda_{\geq }$ is $I_\geq\coloneqq I_{[0,\infty)}\coloneqq {\bfk}[\bR_{> 0}]$. Here, we set $(\mathbb{A}^1)^\Nov=X_{[0,\infty)}^\Nov=\Spec(\Lambda_\geq)$ for the cone $[0,\infty)$.

We mention that even when $\bfk$ is a field, the ring $\Lambda_{\geq }$ is not a principal ideal domain since $I_\geq$ is not a finitely generated ideal of $\Lambda_{\geq }$. This issue prevents us from applying the theory of Noetherian commutative algebra in the study of $\Lambda_{\geq }$. However, it can be resolved by considering the category of $\bR$-graded almost modules
\[\Moda_{\bR-gr}(\Lambda_\geq,I_\geq )= \Mod_{\bR-gr}(\Lambda_{\geq })/\Mod_{\bR-gr}^\partial(\Lambda_{\geq }),\]
where\footnote{In the literature, modules satisfying $I_\geq  M=0$ are called almost zero. The terminology can be defined in a more general situation, which we will explain later.} $\Mod_{\bR-gr}^\partial(\Lambda_{\geq })=\{M\colon I_\geq M=0\}$. In the quotient category, all ideals of $\Lambda_\geq$ are isomorphic to principal ideals (when $\bfk$ is a field). 

Notice that in this case, the idempotent ideal $I_\geq$ is quite canonical. Thus, we often write $\Moda_{\bR-gr}(\Lambda_\geq)=\Moda_{\bR-gr}(\Lambda_\geq,I_\geq )$ for simplicity.

In this case, one particular extension concerns the adic completion of $\Lambda_{\geq }$: For the degree $1$ generator $t\in \Lambda_\geq $, we denote by $\widehat{\Lambda}_\geq $ the $t$-adic completion\footnote{A priori we should take the derived completion, but one can prove this derived completion is isomorphic to the classical completion, which makes no ambiguity.} of ${\Lambda_\geq }$, i.e., the universal Novikov ring appearing in Lagrangian intersection Floer theory~\cite{FOOO}.  Let $\widehat{I}_\geq$ be the $t$-adic completion of ${I}_\geq$.Then we may similarly consider $\Moda_{\bR-gr}(\widehat{\Lambda}_\geq)=\Moda_{\bR-gr}(\widehat{\Lambda}_\geq,\widehat{I}_\geq)$.

\item \textit{Persistence modules and Tamarkin category.}

On the sheaf side, the corresponding category is $ \Sh_{\geq }(\bR)\coloneqq \Sh_{\bR\times [0,\infty)}(\bR)$. However, in applications, several equivalent formalisms appear, which we explain now.

By $\bR$-filtered modules, we mean functors $\arbR \rightarrow \Mod({\bfk})$, where $\arbR$ denotes the symmetric monoidal category defined by the partially ordered group $\arbR$; and the category of $\bR$-filtered modules is denoted by $\Rep(\arbR)=\Fun(\arbR,\Mod({\bfk})).$ 

There exists a full subcategory $\Pers(\arbR)$ of $\Rep(\arbR)$ that consists of persistence modules: i.e., functors that send colimits in $\arbR$ to limits in $\Mod({\bfk})$. It is proven in \cite[Lemma 4.8]{Hochschild-Kuo-Shende-Zhang} or \cite[Proposition 4.22]{Efimov-K-theory} that $\Sh_{\geq }(\bR) \simeq \Pers(\arbR)$. The identification follows from the microlocal cut-off lemma, which shows that $\Sh_{\geq }(\bR)\simeq \Sh(\bR_{(-\infty,0]})$, where $\bR_{(-\infty,0]}$ is the topology on $\bR$ whose open sets are intervals $(-\infty,a)$ for $a\in [-\infty,+\infty]$.

To compare with the adic completion $\widehat{\Lambda}_\geq$, we shall consider the Verdier quotient $\Sh_{>}(\bR)=\Sh_{\geq }(\bR)/\Loc(\bR)$, where the category of (locally) constant sheaves $\Loc(\bR)$ corresponds to the constant functors in $\Pers(\arbR)$. In the literature, this category is called the Tamarkin category $\sT$, which was introduced by Tamarkin \cite{tamarkin2013}.
\end{itemize}

Now, we can state the following summary of Theorem \ref{itheorem: Global apt 1} in the specific case and some extensions of it. We refer to \autoref{section: proof of APT} for their precise statements. We refer to \autoref{Appendix: diagram of modules} for a diagrammatic view of the precise relations between all related categories.
 
\begin{Thm*}[1d APT correspondence]\label{theorem: intro non-equi}For a commutative ring ${\bfk}$:
\begin{enumerate}[fullwidth]
    \item We have an equivalence $ \Sh_{\geq }(\bR) \simeq \Moda_{\bR-gr}(\Lambda_{\geq })$. Moreover, under the equivalence and the identification $ \Sh_{\geq }(\bR)\simeq \Sh(\bR_{(-\infty,0]})$, we have that the $\bR$-graded almostification functor for $\Lambda_{\geq }$ factors through the sheafification functor for $\bR_{(-\infty,0]}$;
    \[\begin{tikzcd}
{\PSh(\bR_{(-\infty,0]}) } \arrow[d]         & \Rep(\arbR)=\Mod_{\bR-gr} (\Lambda_{\geq }) \arrow[d] \arrow[l] \\
{\Sh_{\geq }(\bR)=\Sh(\bR_{(-\infty,0]})} \arrow[r, "\simeq"] & \Moda_{\bR-gr} (\Lambda_{\geq })   .                
\end{tikzcd}\]

\item Let $\Mod_{\bR-gr,comp}(\widehat{\Lambda}_{\geq })$ be the category of $\widehat{\Lambda}_{\geq }$-modules that are $t$-adically (derived) complete and $\Moda_{\bR-gr,comp}(\widehat{\Lambda}_{\geq })$ its image under the almostification functor, then we have
\[\Sh_{>}(\bR)  \simeq \Moda_{\bR-gr,comp}(\widehat{\Lambda}_{\geq })\hookrightarrow \Mod_{\bR-gr}(\widehat{\Lambda}_{\geq }).\]

\item Let $\bR_\delta$ be the additive group $\bR$ as a discrete group, which acts on both sides of the equivalences $\Sh_{\geq }(\bR)\simeq \Moda_{\bR-gr}(\Lambda_{\geq })$ and $\Sh_{>}(\bR)  \simeq \Moda_{\bR-gr,comp}(\widehat{\Lambda}_{\geq })$. We set $\Sh_{\geq }^{\bR_\delta}(\bR)=(\Sh_{\geq }(\bR))^{\bR_\delta}$ and $\Sh_{> }^{\bR_\delta}(\bR)=(\Sh_{> }(\bR))^{\bR_\delta}$. Then by taking the homotopy $\bR_\delta$-fixed points, we have  
\[\Sh_{\geq }^{\bR_\delta}(\bR)\simeq \Moda(\Lambda_{\geq }),\quad \Sh_{>}^{\bR_\delta}(\bR)  \simeq \Moda_{t-comp}(\widehat{\Lambda}_{\geq })\hookrightarrow \Mod(\widehat{\Lambda}_{\geq }).\]
 
\end{enumerate}
\end{Thm*}

\begin{RMK}
\begin{enumerate}[fullwidth]
    \item We will see later that all the sheaf categories are equipped with a convolution monoidal structure and the module categories are equipped with the (completed) tensor product over the corresponding (complete) Novikov rings, and all equivalences are monoidal equivalences. 
    \item The category of $\bR_\delta$-equivariant sheaves $\Sh_{>}^{\bR_\delta}(\bR)$ and its higher-dimensional generalization were first considered by the first-named author in \cite{KuwagakiWKB}. A variant of statement (3) was first proven in \cite{Kuwagaki_almost_equivalence} by the first-named author. Similar results are discussed in \cite[Section 2]{Fukayacategories_prequantizationbundles_KPS} for general partially ordered groups. 
\end{enumerate}
\end{RMK}

\subsubsection{\texorpdfstring{Applications of the 1d APT correspondence: $\arbR$-action on stable categories}{}}
\begin{Def*}For a $\bfk$-linear ($\infty$-)category $\cC$, a (left) $\arbR$-action is a monoidal functor $\arbR\rightarrow \Fun_{\bfk}(\cC,\cC)$. 
\end{Def*}
Roughly speaking, it consists of a family of functors $\nT_a:\cC\rightarrow \cC$ together with (coherent) isomorphisms $\nT_a \circ \nT_b \simeq \nT_{a+b}$ for $a,b\in \bR$ and natural transformations $\tau_c:\id\Rightarrow \nT_c$ for $c\geq 0$. In fact, by the group property, all $\nT_a$ are adjoint equivalences.

In this paper, we are mainly concerned with $\arbR$-action on a ${\bfk}$-linear \textit{\textbf{stable}} category $\cC$. The $1$-categorical version of this notion was considered by \cite{de2018theory}. The stable condition was addressed by Biran--Cornea--Zhang under the name of triangulated persistence category (TPC, for short), see \cite{TPC,TPF}.

Examples of stable categories admitting an $\arbR$-action include: the ($dg$-nerve of) $dg$ category of filtered modules over a filtered $A_\infty$-category (for example, filtered Fukaya category \cite{fukaya_distance,TPF}), category of sheaves equipped with a bi-thickening kernel action \cite{Ks2018,Thickeningkernel} (for example, the Tamarkin category \cite{tamarkin2013,mc-Tamarkin,Hochschild-Kuo-Shende-Zhang}). 

In terms of the notion of $\arbR$-action on stable categories, we can extend many known constructions in the $\infty$-categorical setting, for example, Tamarkin torsion and interleaving distance. 

Moreover, in \autoref{section: TPC}, we will explain that stable categories equipped with an $\arbR$-action are good enhancements of TPCs, and yield certain simplifications of their theory.

Precisely, we have the following result, which is a combination of \autoref{proposition: R action iff Novikov action}, \autoref{prop: stable R action implies Tamarkin module} and \autoref{prop: R action implies persistence}. Here, we refer to \cite{gepner2015enriched} for the notion of enriched $\infty$-categories.

\begin{iThm}An $\arbR$-action on a ${\bfk}$-linear \textit{\textbf{stable}} category $\cC$ is equivalent to an action of (the subcategory of finitely generated objects of) $\Mod_{\bR-gr}(\Lambda_{\geq })$ on $\cC$. In this case, there exists a TPC structure whose underlying triangulated category is the homotopy category ${h\cC}$.

If $\cC$ is presentable, then $\cC$ can be upgraded to be a category enriched over $\Mod_{\bR-gr}(\Lambda_{\geq })$.

If $\cC$ is presentable and the $\arbR$-action on $\cC$ is continuous in the sense that we have $\varinjlim_{n}\nT_{a_n} \simeq \nT_{a}$ for $a_n \nearrow a$, then the $\arbR$-action on $\cC$ lifts to an $\Sh_{\geq }(\bR)$-action and $\cC$ can be upgraded to be a category enriched over $\Sh_{\geq }(\bR)$. 
\end{iThm}
\begin{RMK}
As explained in \cite{TPC,TPF}, TPC is basically a triangulated category equipped with an $\arbR$-action with an {\it extra condition}. One point of the result is that the extra condition is superfluous (such a phenomenon already appears  in \textit{loc. cit.} for pre-triangulated filtered $dg$-categories.) 
\end{RMK}
\begin{RMK}
Here, we explain the role of stability and presentability. Generally speaking, $\arbR$-actions, $\Mod_{\bR-gr}(\Lambda_{\geq })$-actions, and enrichments over $\Mod_{\bR-gr}(\Lambda_{\geq })$ are three different notions. There are categories satisfying one of these conditions but not the others. Their relation follows from the following observation: $\Mod_{\bR-gr}(\Lambda_{\geq })$ is the initial presentable stable category that admits $\arbR$-action.

Therefore, the stability of $\cC$ allows us to adjoin enough finite limits and colimits in a compatible way. This enables us to extend the $\arbR$-action to an action of the subcategory of finitely generated objects of $\Mod_{\bR-gr}(\Lambda_{\geq })$ by Kan extension. Presentability then allows us to adjoin the further colimits needed to obtain an action of the whole category $\Mod_{\bR-gr}(\Lambda_{\geq })$.

On the other hand, an enrichment of $\cC$ over $\Mod_{\bR-gr}(\Lambda_{\geq })$ has a different flavor. A priori, without stability, it is less directly comparable to the other two notions. However, actions and enrichments are closely related through the notion of copower, which is a version of colimit adjoint to the enriched Hom. Presentability then ensures the existence of enough copowers and adjoints, allowing us to pass from copowers to an enrichment.
\end{RMK}

To the authors' knowledge, many known TPCs arise from stable categories admitting an $\arbR$-action. We should also not expect that all TPCs arise in this way (compare with stable $\infty$-categories and triangulated categories). Nevertheless, we believe it would be a meaningful categorical framework for studying TPCs for two reasons: 1) Most of the manual constructions in the theory of TPCs now naturally come from the corresponding properties of $\Mod_{\bR-gr}(\Lambda_{\geq })$ or $\Sh_{\geq }(\bR)$. 2) All localizing invariants (Hochschild homology, higher K-theory\footnote{In this article, we only discuss the nonconnective K-theory.} etc.) are well-defined for dualizable $\cC$ by \cite{Efimov-K-theory} without any manual construction. 

For example: 1) In \cite{Hochschild-Kuo-Shende-Zhang}, the authors define the $\infty$-version of the Tamarkin categories $\sT(U)$ which is naturally equipped with a continuous $\arbR$-action in our sense. The authors consider the categorical trace as a model of Hochschild homology, which is naturally equivalent to the one constructed from Efimov's recipe. Then the authors compute Hochschild cohomology (in the filtered sense!) of $\sT(U)$ in terms of filtered symplectic cohomology.  2) For a ${\bfk}$-{{linear}} dualizable stable category $\cC$, we can discuss the continuous $K$-theory $K^{\cont}(\cC)$, introduced by \cite{Efimov-K-theory}. The following result is a translation of \cite[Proposition 4.22]{Efimov-K-theory}.
\begin{Thm*}[Efimov, {\autoref{prop: continuous K 1d}}]We have the equivalence of spectra
\[K^{\cont}(\Rep(\arbR)) =K^{\cont}(\Mod_{\bR-gr}(\Lambda_{\geq })) = \bigoplus_{\bR} K({\bfk}).\]
In particular, for a principal ideal domain $\bfk$ (or the sphere spectrum $\bfk=\bS$), we have $K_0^{\cont}(\Rep(\arbR))=\bZ[\bR]$, i.e., $\Lambda_\bR$ over $\bZ$.

Consequently, for a dualizable ${\bfk}$-{{linear}} stable category $\cC$ equipped with an $\arbR$-action, the continuous $K$-theory $K^{\cont}(\cC)$ admits an action of $\bigoplus_{\bR}K({\bfk})$. The homotopy groups $\bigoplus_{i}K_i^{\cont}(\cC)$ form an $(\bN,\bR)$-bigraded module over $\bigoplus_{i,\bR}K_i({\bfk})$, and in particular an $\bR$-graded module over $\bigoplus_{\bR}K_0({\bfk})$ (which equals  $\bZ[\bR]$ if $\bfk$ is a principal ideal domain, or $\bS$.)
\end{Thm*}
Because $K^{\cont}_0(\cC)$ is isomorphic to the Grothendieck group of compact objects of $h\cC$, i.e., $K_0((h\cC)^c)$ when $\cC$ is compactly generated, the result could be considered as a generalization of \cite[Proposition 4.2.2]{TPK}.

\subsubsection{Toric mirror symmetry over the Novikov ring}
As another application, we discuss a homological mirror symmetry for B-model toric varieties.

Homological mirror symmetry is, in general, an equivalence between a Fukaya-type category and a coherent-sheaf-type category. A Fukaya-type category is generally defined over the Novikov ring due to the existence of infinitely many holomorphic disks and the Gromov compactness. So, we should consider the mirror B-model defined over the Novikov ring as well.

However, in the literature at present, homological mirror symmetry is not formulated and proved over the Novikov ring, but rather over the Novikov field (=the fraction field). One reason is that Fukaya-type categories over the Novikov ring have too many objects, and it is difficult to set up an appropriately huge comparable B-model. 

Despite this situation, homological mirror symmetry over the Novikov ring has been expected, see for example~\cite{fukaya_distance}.

In this paper, we give a first instance of homological mirror symmetry over the Novikov ring, as an application of our results. Let us consider a toric variety $X_\Sigma$. A mirror A-model $W_\Sigma$ is a Landau--Ginzburg model defined on $(\bC^*)^n$. Since $(\bC^*)^n$ is an exact symplectic manifold, the usual Fukaya--Seidel category can be defined over the base ring if one only deals with exact Lagrangians. So, it is customary to formulate homological mirror symmetry over the base. Here we consider the Fukaya--Seidel category over the Novikov ring $\Fuk(W_\Sigma)$ (more precisely, a conjectural sheaf model of it). Our result is roughly the following:
\begin{iThm}
    A sheaf model of the Fukaya--Seidel category of $W_\Sigma$ over the Novikov ring is embedded into the category of derived complete almost quasi-coherent sheaves over $\sqrt[\infty]{X_\Sigma}\times (\bA^1)^{\Nov}$, where $\sqrt[\infty]{X_\Sigma}$ is the infinite root stack of the toric variety $X_\Sigma$ and $(\mathbb{A}^1)^\Nov=\Spec(\Lambda_\geq)$.
\end{iThm}

We also discuss a conjectural extension of this result to the case of log Calabi--Yau B-models. We refer to \autoref{section: Novikov HMS} for both the precise statement of the result and our conjecture.

\subsection{Motivation, related works and discussion}
We start our investigation from the 1d APT and its applications in quantitative symplectic geometry and homological mirror symmetry over the Novikov ring (rather than the Novikov field) after \cite{Hochschild-Kuo-Shende-Zhang,Kuwagaki_almost_equivalence,Fukayacategories_prequantizationbundles_KPS}. Meanwhile, we are also interested in using APT to give a clean understanding of many results about persistence modules, for example \cite{Ks2018,HAforPERS, milicevic2020convolutionpersistencemodules}. In particular, the discussion in this article can be regarded as a higher algebra version of \cite{HAforPERS}. We notice that almost zero modules appear as ephemeral persistence modules in the literature, where similar results are also proved, see \cite{Berkouk_2021ephemeral,zhou2024ephemeral} and references therein. This led us to the global picture in \cite{Vaintrob-logCCC}, where Theorem \ref{itheorem: Global apt 1} was originally claimed.

One motivation for the global APT is that, as Vaintrob explained in \cite[Section 7]{Vaintrob-logCCC}, one can recover from the global APT the Constructible-Coherent correspondence. The details of Vaintrob's approach for CCC are carried out in \cite{bai2025toricmirrorsymmetryhomotopy}. Notice that the global APT only depends on the fan support $|\Sigma|$. However, the CCC is sensitive to the rational fan $\Sigma$. The difference is the following: Compared with Theorem \ref{itheorem: Global apt 1} which only cares about a polyhedral conic closed set on cotangent directions $T^*_xM_\bR$, the CCC needs an extra choice of stratification on the base $M_\bR$, which is very sensitive with respect to the rational structure of the fan. This is also why we do not need the smoothness in the proof of Theorem \ref{iThm: sheaves with fan support} compared to \cite[Section 4]{bai2025toricmirrorsymmetryhomotopy}. 

Another interesting observation from the APT correspondence is the following: It is proven by Lurie that any presentable stable category that is dualizable with respect to the Lurie tensor product is equivalent to an almost module category for some almost content, see \cite[Theorem 2.10.16]{Fake_sheaves_on_mfd} for a complete proof. However, its proof is existential. Most of the categories that we discuss in this paper are dualizable, especially categories of sheaves with microsupport conditions (a direct proof can be seen in \autoref{appendix: dualizability}), and the APT correspondence gives examples of explicit constructions of Lurie's result.

\subsection*{Structure of the paper}For the convenience of different groups of readers, the paper is divided into two main parts, apart from the preliminary \autoref{Section: microsupport} and \autoref{section: almost ring theory}.

The first part focuses only on 1-dimensional results: It consists of \autoref{section: proof of APT}, \autoref{section: Categories admits R-action}, and \autoref{subsection: Novikov P1}, where \autoref{subsection: Novikov P1} gives an overview of the global theory in dimension 1.

The remaining part discusses higher-dimensional results and global results. In particular, certain mirror symmetry results are discussed in this part. 

For ambitious readers, it is fine to skip the first part and start directly from \autoref{section: HD APT}. Only a few proofs
refer back to the first part.

\subsection*{Convention}All categories in this article are $\infty$-categories. 

For a commutative ring spectrum ${\bfk}$, $\Mod({\bfk})$ denotes the category of module spectra over ${\bfk}$; if ${\bfk}$ is a discrete commutative ring, we write $\Mod({\bfk})=\Mod(H{\bfk})$ for the Eilenberg-MacLane spectrum $H{\bfk}$ of ${\bfk}$. This category can also be constructed as the ($dg$ nerve of) $dg$ derived category of ${\bfk}$. All tensor products are derived and over the base ring $\bfk$. 

Presheaves, sheaves, and representations take values in $\Mod({\bfk})$ unless otherwise specified. The Novikov rings $\Lambda_{\sigma}$ are defined over ${\bfk}$ unless specifically mentioned.

\subsection*{Acknowledgments}BZ thanks Alexander I. Efimov for explaining his work on continuous K-theory, Ning Guo for discussions on almost ring theory, Pierre Schapira for a historical remark on the lens definition of microsupport, and Jun Zhang for explaining his work on triangulated persistence categories. BZ also thanks Laurent C\^ot\'e, St\'ephane Guillermou, Peter Haine, Vivek Shende, Marco Volpe and Qixiang Wang for helpful discussions. We both thank Yuze Sun for pointing out a mistake in a previous version of the draft.
TK was supported by JSPS KAKENHI Grant Numbers 22K13912, 23K25765, and 20H01794. BZ was supported by the Novo Nordisk Foundation grant NNF20OC0066298 and VILLUM FONDEN, VILLUM Investigator grant 37814.

\section{Sheaves and microsupport}\label{Section: microsupport}
For a smooth manifold $M$, we set $\Sh(M)$ to be the category of $\Mod({\bfk})$-valued sheaves. We refer to \cite{6functor-infinity} for $6$-functor formalism in this setting. We note that, for a manifold $M$, $\Sh(M)$ is hypercomplete, i.e., all sheaves $F\in \Sh(M)$ are hypersheaves by \cite[7.2.3.6, 7.2.1.12]{HTT} and \cite[1.6]{Haine-porta-Teyssier-homotopy-inv-constr}. We refer to \cite[Section 6.5]{HTT} for general discussions about the hypercompleteness. An important feature of hypersheaves is: A hypersheaf $F$ in $\Sh(M)$ is determined by its value on a topological basis, see \cite[Proposition 1.13]{Haine-porta-Teyssier-homotopy-inv-constr} and references therein.

Regarding the microlocal theory of sheaves in the $\infty$-categorical setup, we remark that all arguments of \cite{KS90} work well provided we have the non-characteristic deformation lemma \cite[Proposition 2.7.2]{KS90}, which is proven for all hypersheaves valued in a compactly generated stable category by \cite{Amicrolocallemma_infinitycat}. Here, since we work over a manifold, all sheaves are hypersheaves; and our coefficient category $\Mod({\bfk})$ is indeed compactly generated and stable.

To any object $F \in \Sh(M)$, one can associate a conic closed set $\SS(F) \subset T^*M$ ({\cite[Definition 5.1.2]{KS90}}), called \emph{microsupport}. We recall the following definition and result regarding an equivalent definition of microsupport.
\begin{Def}{\cite[Definition 3.1]{guillermouviterbo_gammasupport}} \label{Def: Omega-lens}Let $\Omega \subset T^*M \setminus 0_M$ be an
open conic subset.  We call {\em a locally closed subset $C$ of $M$
with the following properties an $\Omega$-lens} : $\overline{C}$ is compact and there exists an open
neighborhood $U$ of $\overline{C}$ and a function $g\colon U \times [0,1] \to \bR$
\begin{enumerate}
\item $dg_t(x) \in \Omega$ for all $(x,t) \in U \times [0,1]$, where
  $g_t = g|_{U\times\{t\}}$,
\item $\{g_t<0\} \subset \{g_{t'}<0\}$ if $t\leq t'$,
\item the hypersurfaces $\{g_t=0\}$ coincide on $U\setminus \overline{C}$,
\item $C = \{g_1<0\} \setminus \{g_0<0\}$.
\end{enumerate}
    
\end{Def}

\begin{Lemma}{\cite[Lemma 3.2, 3.3]{guillermouviterbo_gammasupport}}\label{lemma: Omega-lens} Let $F \in \Sh(M)$ and let $\Omega \subset T^*M \setminus 0_M$ be an open conic subset.  Then the following are equivalent: 1) $\SS(F) \cap \Omega = \emptyset$, 2) $\HOM(1_C, F) \simeq 0$ for any $\Omega$-lens $C$, 3) $\Gamma_c(M; 1_C \otimes F) \simeq 0$ for any $(-\Omega)$-lens $C$.
\end{Lemma}

For a closed conic set $Z\subset T^*M$, we set $\Sh_Z(M)$ to be the full subcategory of $\Sh(M)$ consisting of $F$ with $\SS(F)\subset Z$. Then \autoref{lemma: Omega-lens} shows that the inclusion $\Sh_Z(M)\subset \Sh(M)$ commutes with both limits and colimits.

\autoref{lemma: Omega-lens} also leads to a generalization of microsupport, which we will explain in \autoref{appendix: cosheaf microsupport}. Importantly, for a cosheaf $G\in \coSh(M)\coloneqq \Sh(M;\Mod(\bfk)^\op)^\op$, we can define a cosheaf microsupport $\coSS(G)$. We set $\coSh_{Z}(M)$ to be the full subcategory spanned by cosheaves $G$ with $\coSS(G)\subset Z$. Combining \autoref{lemma: Omega-lens} and the covariant Verdier duality \autoref{Thm:Covariant Verdier duality}, we have
\begin{Thm}[{\autoref{covariant Verdier duality with microsupport}}]For a closed conic set $Z\subset T^*M$ containing the zero section, we have $\Sh_Z(M)\simeq \coSh_{-Z}(M)$.    
\end{Thm}

Therefore, all of our results about sheaves have their cosheaf versions.

\medskip

For our applications, the most important result is the following microlocal cut-off lemma, which is proven by Kashiwara and Schapira as a combination of \cite[Propositions 3.5.3, 3.5.4, 5.2.3]{KS90}. To state the lemma, we review the notion of $\gamma$-topology \cite[Definition 3.5.1]{KS90}.

For a closed convex cone $\sigma$ whose dual is $\gamma=\sigma^\vee$, an open set $U$ (with respect to the usual topology of $\bR^n$) is called $\gamma$-open if $U+\gamma =U$. Then $\gamma$-open sets form a topology on the set $\bR^n$, which is denoted by $\bR^n_\gamma$, and the identity map induces a  continuous map $\varphi_\gamma: \bR^n \rightarrow \bR^n_\gamma$.

\begin{Lemma}\label{lemma: microlocal cut-off}For a closed convex cone $\sigma$ whose dual is $\gamma=\sigma^\vee$, the adjoint functors $(\varphi_\gamma^*,\varphi_{\gamma*})$ induce an adjoint equivalence 
\[\varphi_{\gamma*}:\Sh_{\bR^n\times (-\sigma)} (\bR^n) \xrightarrow{\cong}  \Sh(\bR^n_\gamma), F\mapsto [U\mapsto F(U+\gamma)].\]
\end{Lemma}
\begin{Coro}The category of sheaves $\Sh(\bR^n_\gamma)$ is hypercomplete.
\end{Coro}
\begin{proof}For an $\infty$-connected morphism $h:F\rightarrow G$ in $\Sh(\bR^n_\gamma)$, we have that $\varphi_\gamma^*h$ is an $\infty$-connected morphism by \cite[Proposition 6.5.1.16-(4)]{HTT}. Since $\Sh (\bR^n)$ is hypercomplete, $\varphi_\gamma^*h$ is an equivalence. So $h=\varphi_{\gamma*}\varphi_\gamma^*h$ is an equivalence, which implies that $\Sh(\bR^n_\gamma)$ is hypercomplete.
\end{proof}

\section{Homological epimorphism and almost ring theory}\label{section: almost ring theory}
Here, let us recall some basic facts about homological epimorphism and almost ring theory. We emphasize that, since we need to work over a ring spectrum, all tensor products are derived even when we work over discrete rings. We follow the discussion in \cite[Section 2.9]{Fake_sheaves_on_mfd}.

\begin{Def}For a morphism of ring spectra $R\rightarrow S$ whose fiber is the $R$-module spectrum ideal $I$, we say $R\rightarrow S$ is a homological epimorphism if the following equivalent conditions are true:
\begin{enumerate}[fullwidth]
    \item The multiplication $S\otimes_R S \rightarrow S$ is an equivalence.
    \item We have $I\otimes_R S \simeq 0$.
    \item The morphism $I\otimes_R I \rightarrow I$ induced by the multiplication is an equivalence.
\end{enumerate}

In this case, we call $I$ an idempotent ideal of $R$. 
\end{Def}
\begin{RMK}A standard example of idempotent ideals is the following: Consider a non-discrete valuation ring $V$; then its maximal ideal $m$ is idempotent. Idempotent ideals often show up in non-Noetherian algebraic geometry.
\end{RMK}
We recall the following lemma \cite[Lemma 2.9.9]{Fake_sheaves_on_mfd}, where the equivalences in the above definition are also discussed.
\begin{Lemma}[{\cite[Lemma 2.9.9]{Fake_sheaves_on_mfd}}]For a morphism of ring spectra $R\rightarrow S$, it is a homological epimorphism if and only if the base change functor
\[\Mod(R)\rightarrow \Mod(S), \quad M\mapsto M\otimes_R S\]
is a left Bousfield localization, i.e., the right adjoint of the base change functor (i.e., the functor of restriction of scalars) is fully faithful.
\end{Lemma}

\begin{Def}If $R\rightarrow S$ is a homological epimorphism with fiber $I$ (i.e., $I$ is an idempotent ideal of $R$), we define the kernel of $\Mod(R)\rightarrow \Mod(S)$ to be $\Moda(R,I)$. 
\end{Def}
Passing to right adjoints, we have a Verdier sequence in $\PrStR$: 
\[\Mod(S) \xrightarrow{res} \Mod(R) \rightarrow \Moda(R,I).\]

In particular, we have the identification $\Mod(S) \simeq \{M:  M\otimes_R I=0\}$. In terms of this formalism, we may suppress the role of $S$, and only consider the pair $(R,I)$ when $I$ is an idempotent ideal. 
\begin{Def}An almost content $(R,I)$ consists of a ring spectrum and an idempotent ideal $I$ (i.e., $R\rightarrow \cofib{I\rightarrow R}$ is a homological epimorphism). 

For an almost content $(R,I)$, we set 
\[\Mod^\partial{(R,I)}=\{M:  M\otimes_R I=0\},\quad \Moda(R,I)=\Mod(R)/\Mod^\partial{(R,I)},\]
where the Verdier quotient is taken in $\PrStR$. 

We call $\Mod^\partial{(R,I)}$ the category of almost zero modules, the quotient $\Moda(R,I)$ the category of almost modules and the quotient functor $(-)^a:\Mod(R)\rightarrow \Moda(R,I)$ the almostification functor. By definition, we have $\Mod^\partial{(R,I)}\simeq \Mod{(\cofib{I\rightarrow R})}$. When $(R,I)$ is clear from the context, we will also write $\Moda(R)$.
\end{Def}

In the commutative setting, we have:
\begin{Prop}For a commutative ring spectrum $R$ and an idempotent ideal $I$, the fully faithful inclusion functor $\Mod^\partial{(R,I)}\subset \Mod(R)$ has a right adjoint. 
\end{Prop}
\begin{proof}Define  $M\mapsto \HOM_{R}(\cofib{I\rightarrow R},M)$, it is straightforward to check that this is the right adjoint we need.
\end{proof}
In fact, this can be proven by a categorical discussion without assuming commutativity; we use the commutative setup here since it is more concrete for modules and sufficient for our applications.
\begin{Coro}\label{coro: almost local objects}For a commutative ring spectrum $R$ and an idempotent ideal $I$, the almostification functor has a fully faithful right adjoint $[M]\rightarrow \HOM_{R}(I,M)$ such that the essential image, i.e., the right orthogonal complement of almost zero modules, is given by $\{M\in \Mod(R): M\simeq \HOM_R(I,M)\}$.   
\end{Coro}
\begin{proof}This is a standard categorical result, see \cite[Lemma 2.9.11]{Fake_sheaves_on_mfd} for example. 
\end{proof}
\begin{Def}\label{def: almost local}For a commutative ring spectrum $R$ and an idempotent ideal $I$, we say $M \in \Mod(R)$ is almost local if it is in the essential image of the right adjoint of the almostification functor, i.e., $ M\simeq \HOM_R(I,M)$ by \autoref{coro: almost local objects}. 
\end{Def}

\begin{RMK}If $R$ is a discrete commutative ring and $I$ is a flat idempotent ideal of $R$, then we have $\cofib{I\rightarrow R}=R/I$, $\Mod(R)\simeq \textnormal{D}(R)$ and $\Mod(R/I)\simeq \textnormal{D}(R/I)$ where $\textnormal{D}$ stands for the (enhanced) derived category of the corresponding rings. In this case, $\Moda{(R,I)}$ is equivalent to the (enhanced) derived category of the abelian category of almost modules discussed in \cite{Gabber_Ramero}.
    
\end{RMK}

\begin{RMK}In this article, we restrict ourselves to the idempotent ideal situation. There is a general setting for higher almost ring theory; we refer to \cite{hebestreit2024notehigherringtheory}.    
\end{RMK}

In this article, we also consider the graded situation: We consider a $\Gamma$-graded ring $R$ for a grading group $\Gamma$, and we consider a homogeneous idempotent ideal $I$. Then we say a graded module $M$ is graded almost zero if $I\otimes_{gr}M=0$, and we define $\Moda_{\Gamma-gr}(R,I)$ as the Verdier quotient by graded almost zero modules. Then we can also define graded almost local modules as objects in the right semi-orthogonal complement of graded almost zero modules. The same argument as in \autoref{coro: almost local objects} can also show that $M$ is graded almost local if and only if $M_{-d}=\HOM_{gr}(R(d),M)\simeq \HOM_{gr}(I(d),M)$ for all $d\in \Gamma$ (here, degree shifts appear since $\HOM_{gr}$ only means degree $0$ morphisms).

\section{1-dimensional APT correspondence}\label{section: proof of APT}

\subsection{\texorpdfstring{{$\bR$-filtration and the polynomial Novikov ring}}{}}\label{section: module discussion 1D}

Let $\arbR$ be the partially ordered group of the real numbers, and we regard it as a symmetric monoidal category in the following way: The object set is $\bR$; one arrow $a\rightarrow b$ is in $\arbR$ if and only if $a\leq b$; the monoidal structure is given by the additive group structure of $\bR$, which is symmetric monoidal because $\bR$ is an abelian group. We also define $\bR^{ds}$ to be the set of reals as a discrete monoidal category.

\begin{Def}\label{def: rep 1d}We set $\Rep(\arbR)\coloneqq \Fun(\arbR,\Mod({\bfk}))$; its objects are called $\bR$-filtered ($\bfk$-)modules. We set the category of $\bR$-graded ($\bfk$-)modules to be $\Mod_{\bR-gr} ({\bfk})\coloneqq\Fun(\bR^{ds},\Mod(\bfk))=\prod_{\bR}\Mod({\bfk})$.
\end{Def}

Therefore, an object $M\in \Mod_{\bR-gr}({\bfk})$ consists of a family of objects $(M(a))_{a\in \bR}$ with $M(a)\in \Mod(\bfk)$; an object $M\in \Rep(\arbR)$ consists of a family of objects $(M(a))_{a\in \bR}$ with $M(a)\in \Mod(\bfk)$, and moreover for every pair $a\leq b$, an associated morphism $M(a\leq b) : M(a)\to M(b)$. 

Since $\arbR$, $\bR^{ds}$, and $\Mod({\bfk})$ are all symmetric monoidal, we can define monoidal structures on $\Mod_{\bR-gr}(\bfk)$ and $\Rep(\arbR)$ respectively: On $\Mod_{\bR-gr}(\bfk)$, we define the usual graded tensor product by $M\otimes_{gr} N (a)=\oplus_{a=s+t} M(s)\otimes N(t)$; on $\Rep(\arbR)$, we define the Day convolution monoidal structure on $\Rep(\arbR)$ via the formula $F\star G (a) \coloneqq \varinjlim_{a\geq s+t } F(s)\otimes G(t)$.

The tensor unit $\Lambda_{\geq}$ with respect to the Day convolution $\star$ is given by the formula $\Lambda_{\geq}(a)={\bfk}$ for $a\geq 0$ and $\Lambda_{\geq}(a)=0$ for $a<0$. The Day convolution tensor unit $\Lambda_{\geq}$ can be upgraded to a commutative algebra object in $\Mod_{\bR-gr}(\bfk)$ (i.e. a graded commutative ring) via the commutative ring spectrum structure of ${\bfk}$: precisely,  for $a,b\geq 0$ , we define $\Lambda_{\geq}(a) \otimes  \Lambda_{\geq}(b)={\bfk}\otimes  {\bfk}\xrightarrow{m_{\bfk}} \Lambda_{\geq}(a+b) ={\bfk}$ where $m_{\bfk}$ is the multiplication of ${\bfk}$; otherwise the multiplication is zero.  

\begin{RMK}Here, we use the notation $\Lambda_{\geq}$ to denote the corresponding graded algebra. We will explain later how it recovers the classical notion of the Novikov ring.   
\end{RMK}

The two categories $\Rep(\arbR)$ and $\Mod_{\bR-gr}(\bfk)$ are closely related. In fact, the functor $\bR^{ds}\to \arbR  , a \mapsto a$ induces the following restriction functor, which is also called the $\bR$-filtered Rees construction:
\[\Rees_\bR: \Rep(\arbR)=\Fun(\arbR,\Mod(\bfk)) \rightarrow \Mod_{\bR-gr}({\bfk})=\Fun(\bR^{ds},\Mod(\bfk)),\quad M\mapsto (M(a))_{a\in \bR} .\]

As a commutative algebra object $\Lambda_{\geq}$ in $\Mod_{\bR-gr}(\bfk)$, we can consider the category of modules in $\Mod_{\bR-gr}(\bfk)$ over $\Lambda_{\geq}$, namely 
\[\Mod_{\bR-gr}(\Lambda_{\geq}) \coloneqq \Mod_{\Lambda_{\geq}}( \Mod_{\bR-gr}({\bfk})),\]
whose objects are pairs $(M,\varphi)$ for $M\in \Mod_{\bR-gr}({\bfk})$ and an $\Lambda_\geq$-action morphism $\varphi: \Lambda_\geq \otimes_{gr}M\to M$ satisfying compatibility for modules.

We have the following result.
\begin{Prop}\label{prop: Rees construction}
    The $\bR$-filtered Rees construction induces a monoidal equivalence
\[\Rees_\bR: \Rep(\arbR)\simeq \Mod_{\bR-gr}(\Lambda_{\geq}).\]
Both categories are compactly generated and hence presentable.
\end{Prop}
\begin{proof}In \cite[Proposition 3.1.6]{Rotation-inv-Lurie}, Lurie proves a $\bZ$-filtered version of the result. We notice that the proof still works after replacing $\bZ$ with $\bR$. Here, we describe the construction precisely.

For $M\in \Rep(\arbR)$, we can equip $\Rees_\bR(M)$ with an action of $\Lambda_{\geq}$ since $M$ is an $\bR$-filtered object: for $b\geq a$, we define a map
\[
\varphi_b: \Lambda(b-a) \otimes M(a) \simeq M(a)\xrightarrow{M(a\leq b)} M(b),\]
then its colimit over $b=(b-a)+a$ gives a morphism
\[\varphi:\Lambda\otimes_{gr} \Rees_\bR(M) \to \Rees_\bR(M).\]

Thus, the restriction functor factors through a functor, which is still denoted by $\Rees_\bR$:
\[\Rep(\arbR)\rightarrow \Mod_{\bR-gr}(\Lambda_{\geq}) \coloneqq \Mod_{\Lambda_{\geq}}( \Mod_{\bR-gr}({\bfk})).\]

The claim of \cite[Proposition 3.1.6]{Rotation-inv-Lurie} is that this lax monoidal functor is a monoidal equivalence.

For the compact generation of $\Mod_{\bR-gr}(\Lambda_{\geq})$, we refer to \cite[2.2.3, 2.2.4]{Rotation-inv-Lurie}.
\end{proof}
\begin{RMK}\label{remark: filtration vs grading}
Here, an $\bR$-graded object $M$ simply means that there is an $\bR$-family of objects $M(a)\in\Mod(\bfk)$ parameterized by $a\in \bR$. When $M$ admits a $\Lambda_\geq$-action, it simply means that we have morphisms $M(a\leq b):M(a)\to M(b)$ for $a\leq b$.

From this point of view, contrary to the abelian world, objects in $\Mod_{\bR-gr}(\Lambda_{\geq})$ have mixed features of both gradings and filtrations, and it makes less sense to distinguish them. Here, we simply use the terminology $\bR$-grading if we want to emphasize the $\Lambda_\geq$-action; we may also call them $\bR$-filtered objects if we want to emphasize it arises from $\Rep(\arbR)$.
\end{RMK}

\begin{RMK}Informally, the functor can be written as $\Rees_{\bR}M=\bigoplus_a M(a) t^a$, which motivates us to call it the Rees construction. 

In the discrete ring case, we have that $\Rees_{\bR}\Lambda_\geq = \bigoplus_{a\geq 0} \Lambda_\geq(a)\, t^a = \bigoplus_{a\geq 0} {\bfk}\, t^a$ is exactly the  polynomial Novikov ring. Then we will tautologically define the spectral polynomial Novikov ring $\Lambda_\geq$ as $\Rees_{\bR}\Lambda_\geq$, which is equivalent, after forgetting the grading, to $(\Sigma^\infty_+ \bR_{\geq 0})\wedge \bfk$ as a commutative ring spectrum by viewing $\bR_{\geq 0}$ as a discrete topological space.     
\end{RMK}

Now, we set $\bR_\delta$ to be the discrete additive group $\bR$. We have $\bR_\delta$-actions on both sides of the equivalence \autoref{prop: Rees construction}. On the representation side, the action is induced from the action on $\arbR$. On the graded module side, the action is given as follows: for $b\in \bR_\delta$, we have
\begin{equation}\label{equation: R-action formula on graded modules}
\nT_b ( M ) (a)\coloneqq M(a+b).
\end{equation}

It is clear that the equivalence in \autoref{prop: Rees construction} is $\bR_\delta$-equivariant.

Moreover, we have
\begin{Coro}\label{coro: non-graded equivalence}The grading-forgetting functor $\Mod_{\bR-gr}(\Lambda_{\geq})\rightarrow \Mod({\bfk})$ induces a monoidal equivalence 
\[\Rep(\arbR)^{\bR_\delta}\simeq\Mod_{\bR-gr}(\Lambda_{\geq})^{{\bR_\delta}}\xrightarrow{\simeq} \Mod(\Lambda_{\geq}).\]    
\end{Coro}
\begin{proof}Let $\cC=\Fun(\bR^{ds},\Mod(\bfk))$. The trivial $\bR_\delta$-action on $\Mod(\bfk)$ gives a trivial action of $\cC$ on $\Mod(\bfk)$ by the left Kan extension: Precisely, we have a monoidal functor $\cC\rightarrow \Fun^L_{\bfk}(\Mod({\bfk}),\Mod({\bfk}))\simeq \Mod({\bfk})$ given by the grading-forgetting functor $\cC\rightarrow \Mod({\bfk})$, $(M_a)\mapsto \bigoplus_a M_a$. 

Similarly, the $\bR_\delta$-action on $\Rep(\arbR)=\Mod_{\Lambda_{\geq}}(\cC)$ induces a $\cC$-action on $\Rep(\arbR)$ by \cite[Corollary 2.4.4]{Rotation-inv-Lurie} (a proof for $\bZ^{ds}$ is given, but the proof for $\bR^{ds}$ is similar) since $\Rep(\arbR)$ is presentable and stable.

Therefore, by viewing $\Lambda_\geq$ as the commutative $\bfk$-algebra after forgetting the grading, we have the following monoidal equivalence, i.e., an equivalence in $\CAlg(\PrSt)$:
\[\Mod_{\Lambda_{\geq}}(\cC) \otimes_{\cC}\Mod(\bfk)\simeq \Mod(\Lambda_{\geq})\]
by \cite[Theorem 4.8.4.6, Corollary 4.8.5.21]{HA}.

On the other hand, since $\bR_\delta$ is a discrete group acting on $\Rep(\arbR)\simeq \Mod_{\Lambda_{\geq}}(\cC)$, the relative tensor product is computed in the following way
\begin{equation*}
    \begin{split}
\Rep(\arbR) \otimes_{\cC}\Mod(\bfk)
=\varinjlim_n \Rep(\arbR) \otimes \cC^{\otimes n} \otimes \Mod(\bfk)
= \varinjlim_n \Rep(\arbR) \otimes \cC^{\otimes n} 
= \Rep(\arbR)_{\bR_\delta} 
=\Rep(\arbR)^{\bR_\delta}.
    \end{split}
\end{equation*}

Here, the first equality is the definition of relative tensor product, the second equality follows since $\Mod(\bfk)$ is the unit of $\otimes=\otimes_\bfk$. In particular, we obtain the Borel construction of the homotopy orbit. Then the third equality follows because, for a discrete group, we can compute the homotopy orbit either using the Borel construction or using a colimit over $B\bR_\delta$. The last equality follows since $\bR_\delta$ is a group, we can switch a colimit over $B\bR_\delta$ to a limit over $B\bR_\delta$ using $\PrSt\simeq (\PrStR)^\op$.
\end{proof}

In commutative algebra, we have the famous Greenlees--May duality \cite[Section 47.12]{StacksProject}. In particular, it applies to our setting: For $\Lambda_\bR=\Lambda_\geq[t^{-1}]$, we identify $\Mod(\Lambda_\bR)$ with a reflective subcategory of $\Mod(\Lambda_\geq)$, then we have the following equivalence
\begin{equation}\label{equation: Greenlees--May duality}
 \Mod(\Lambda_\bR)^\perp\eqqcolon \Mod_{t-comp}(\Lambda_\geq) \simeq \Mod_{t-tor}(\Lambda_\geq) \coloneqq ^{\perp}\Mod(\Lambda_\bR).
\end{equation}

More precisely, since $\Lambda_\bR$ is a compact generator of $ \Mod(\Lambda_\bR)$, $M\in  \Mod_{t-comp}(\Lambda_\geq)$ is equivalent to $\HOM_{\Lambda_\geq}(\Lambda_\bR,M)=0$, and $M\in  \Mod_{t-tor}(\Lambda_\geq)$ is equivalent to $\HOM_{\Lambda_\geq}(M,\Lambda_\bR)=0$. 

Notice that $(t)$ is a principal ideal of $\Lambda_\geq$. It is explained in \cite[Lemma 15.93.18]{StacksProject} that $M\in  \Mod_{t-comp}(\Lambda_\geq)$ is equivalent to $M\cong \varprojlim_n M\otimes_{\Lambda_\geq} (\Lambda_\geq/t^n\Lambda_\geq)$ (where the right hand side is the derived $t$-adic completion of $M$). It is explained in \cite[Lemma 47.8.1, 47.8.2]{StacksProject} that $M\in  \Mod_{t-tor}(\Lambda_\geq)$ is equivalent to $M\cong \varinjlim_n \HOM_{\Lambda_\geq}(\Lambda_\geq/t^n\Lambda_\geq,M)$. Be careful that $\otimes_{\Lambda_\geq}$ and $\HOM_{\Lambda_\geq} $ are both derived.

Under the identification in \autoref{coro: non-graded equivalence}, $\Mod(\Lambda_\bR)$ is identified with the full subcategory of constant functors in $\Rep(\arbR)^{\bR^\delta}$, and we obtain a localization (i.e., the completion functor)
    \begin{equation}
        \Rep(\arbR)^{\bR^\delta}\to \Mod_{t-comp}({\Lambda}_\geq)
    \end{equation}
whose kernel consists of constant functors.

\textbf{Warning}: Objects in $\Mod_{t-tor}$ are called $t$(-power)-torsion modules, whose classical counterpart means that: for all $m\in M$ there exists $n$ such that $t^n m=0$. This corresponds to torsion of quasi-coherent sheaves in algebraic geometry. We will also use another notion of torsion (Tamarkin torsion), so we should be careful to distinguish them.

\subsection{\texorpdfstring{Almost modules over $\Lambda_{\geq }$ and sheaves over $\bR$}{}}

Here, we will take $R=\Lambda_{\geq }$ and $I_\geq$ defined by $I_\geq(a)=\bfk$ for $a>0$ and $I_\geq( a)=0$ otherwise. Then it is clear that they define an almost content in the $\bR$-graded sense. We will fix the content in this section without further notice.

Recall that $\bR_{(-\infty,0]}$ is the topology on $\bR$ whose open sets are of the form $(-\infty,x)$ for $x\in [-\infty,\infty]$. It is clear that $\arbR \rightarrow \operatorname{Open}(\bR_{(-\infty,0]})^{\op}, a \mapsto (-\infty,-a)$ defines a functor. In fact, we may identify $\bR_{(-\infty,0]}$ with $\arbR\cup \{\pm\infty\}$, and then the functor induces a restriction 
\begin{equation}\label{equation: PSH=Rep 1d}
 \PSh(\bR_{(-\infty,0]})\rightarrow  \Rep(\arbR),\quad F\mapsto \widetilde{F}=[a\mapsto F((-\infty,-a))[-1]].   
\end{equation}

For the closed conic set $\bR\times [0,\infty) \subset T^*\bR$, we set $\Sh_{\geq }(\bR)=\Sh_{\bR\times [0,\infty)}(\bR)$. By the microlocal cut-off lemma, i.e., \autoref{lemma: microlocal cut-off}, we naturally identify $\Sh_{\geq }(\bR)$ with $\Sh(\bR_{(-\infty,0]})$. 

\begin{Prop}\label{prop: APT}The functor \eqref{equation: PSH=Rep 1d} induces the following equivalence 
\[\Sh_{\geq }(\bR) \simeq \Sh(\bR_{(-\infty,0]}) \simeq  \Moda_{\bR-gr} (\Lambda_{\geq }).\]

Equipping $\Sh_{\geq }(\bR)$ with the $\star$-convolution, i.e., $F\star G\coloneqq s_!(F\boxtimes G)$ where $s(x,y)=x+y$, we have that the equivalence is monoidal.
\end{Prop}
\begin{proof}By \autoref{lemma: microlocal cut-off}, we only consider $\Sh(\bR_{(-\infty,0]})$ here. The sheaf category is fully faithfully embedded into the presheaf category. Then by \eqref{equation: PSH=Rep 1d}, we have a functor,
\[ \Sh(\bR_{(-\infty,0]}) \rightarrow \Rep(\arbR) , \quad F \mapsto \widetilde{F}(a)\coloneqq F((-\infty,-a))[-1]=\HOM(1_{(-\infty,-a)}[1],F),\]
where $1_{(-\infty,-a)}$ denotes the rank 1 constant sheaf on $(-\infty, -a)$.
Warning: the functor only preserves limits but not colimits, which is crucial to the proof!

Now, we claim that under the functor of \autoref{prop: Rees construction}, the sheaves here are identified with $\bR$-graded almost local modules over $\Lambda_{\geq }$: i.e., we need to show that for all $a\in \bR$ and $F\in \Sh(\bR_{(-\infty,0]})$, we have
\[\HOM_{gr}(\nT_{-a}\Lambda_{\geq }, \Rees_{\bR}\widetilde{F})=\HOM_{gr}(\nT_{-a}I_\geq, \Rees_{\bR}\widetilde{F}) .\]

In fact, a direct computation shows that, for $U_a=1_{(-\infty,-a)}[1]$, we have 
\[\Rees_{\bR}\widetilde{U_a}=\Lambda_{\geq a}=\nT_{-a}\Lambda_{\geq }. \]

Then, as $F\mapsto \widetilde{F}$ and $M\mapsto \Rees_{\bR}M$ are fully faithful functors, we have 
\[\HOM_{gr}(\nT_{-a}\Lambda_{\geq }, \Rees_{\bR}\widetilde{F})=\HOM_{gr}(\Rees_{\bR}\widetilde{U_a}, \Rees_{\bR}\widetilde{F})=\HOM(\widetilde{U_a}, \widetilde{F}) = \HOM(U_a, F).\]

Now, if $F$ is a sheaf, we have 
\[\HOM(U_a, F)=F((-\infty,-a))[-1]= \varprojlim_{b>a}F((-\infty,-b))[-1] = \varprojlim_{b>a}\HOM(U_a, F)=\HOM(\varinjlim_{b>a}U_b, F).\]

Importantly, the colimit here is taken in $\PSh$ rather than $\Sh$! Since the equivalence of \autoref{prop: Rees construction} and the functor \eqref{equation: PSH=Rep 1d} preserve colimits, we have
\[\Rees_{\bR}\widetilde{\varinjlim_{b>a}U_b} =\nT_{-a}I_\geq.\]

Therefore, we have
\[ \HOM_{gr}(\nT_{-a}\Lambda_{\geq }, \Rees_{\bR}\widetilde{F})=\HOM(U_a, F)= \HOM(\varinjlim_{b>a}U_b, F)=\HOM_{gr}(\nT_{-a}I_\geq, \Rees_{\bR}\widetilde{F}) .\]

In summary, if $F$ is a sheaf, the $\bR$-graded $\Lambda_{\geq }$-module $\Rees_{\bR}\widetilde{F}$ is graded almost local. To show the converse, we just reverse our computation.

To show the equivalence is monoidal, we notice that the sheaf side is generated by $\{U_a\}_{a\in \bR}$ by colimits and the module side is generated by $\{\Lambda_{\geq a}\}_{a\in \bR}$ by colimits. Moreover, the $\star$-convolution and Day convolution both commute with colimits and shifts in each variable. Therefore, we only need to show that $U_0 \star U_0 \simeq U_0$ since $\Lambda_\geq $ is the Day convolution unit. However, $U_0 \star U_0 \simeq U_0$ is true by a direct computation.
\end{proof}

\begin{Coro}Under the equivalence of \autoref{prop: APT}, the almostification factors through the sheafification functor via the diagram where the upper horizontal arrow is the left Kan extension of the functor \eqref{equation: PSH=Rep 1d};
\[\begin{tikzcd}
{\PSh(\bR_{(-\infty,0]}) } \arrow[d]         & \Mod_{\bR-gr} (\Lambda_{\geq }) \arrow[d] \arrow[l] \\
{\Sh(\bR_{(-\infty,0]})} \arrow[r, "\simeq"] & \Moda_{\bR-gr} (\Lambda_{\geq }) .                
\end{tikzcd}\]

\end{Coro}
\begin{proof}The sheafification functor $\PSh(\bR_{(-\infty,0]})\rightarrow \Sh(\bR_{(-\infty,0]})$ is the left adjoint of the inclusion  $\Sh(\bR_{(-\infty,0]})\hookrightarrow\PSh(\bR_{(-\infty,0]})$. Moreover, the left Kan extension $\Mod_{\bR-gr} (\Lambda_{\geq }) \rightarrow \PSh(\bR_{(-\infty,0]})=\Fun(\arbR \cup \{\pm \infty\},\Mod(\bfk))$ is also the left adjoint of the restriction \eqref{equation: PSH=Rep 1d}.

On the other hand, in the above proposition, we identify $\Moda_{\bR-gr} (\Lambda_{\geq })$ with the right semi-orthogonal complement of the graded almost zero modules in $\Mod_{\bR-gr} (\Lambda_{\geq })$. Under the identification, the natural inclusion $\Moda_{\bR-gr} (\Lambda_{\geq }) \hookrightarrow \Mod_{\bR-gr} (\Lambda_{\geq })$ is the right adjoint of the almostification functor. Thus, we conclude by the uniqueness of the adjoint.    
\end{proof}

\begin{RMK}
    Another way to understand the presence of almost math is that the sheafification makes $1_{(-\infty,0)}$ fail to be compact; but $\Lambda_\geq$ is compact as a $\Lambda_\geq$-module and non-compact as an almost module (in both graded or non-graded sense). We verify here that what makes $1_{(-\infty,0)}$ non-compact is always almost zero.
\end{RMK}
 
Considering the $\bR_\delta$-action, we have:
\begin{Prop}\label{prop: equivariant APT}The equivalence \autoref{coro: non-graded equivalence} induces the following monoidal equivalence 
\[\Sh^{\bR_\delta}_{\geq }(\bR) \simeq  \Moda (\Lambda_{\geq }).\]
\end{Prop}
\begin{proof}
    Since $\bR_\delta$ is a discrete group, limits and colimits in $\PrSt$ over $\bR_\delta$ are equivalent by \cite[Example 4.3.11]{Hopkins-Lurie-Ambidexterity}. So, we have
    \[\Sh^{\bR_\delta}_{\geq }(\bR)=\varinjlim_{B\bR_\delta}\Sh_{\geq }(\bR)=\Sh_{\geq }(\bR)_{\bR_\delta}.\]

    On the other hand, \autoref{prop: APT} shows that $\Sh_{\geq }(\bR)$ is equivalent to a Verdier quotient in $\PrSt$:
    \[\Sh_{\geq }(\bR)=\Mod_{\bR-gr} (\Lambda_{\geq })/\Mod^\partial_{\bR-gr} (\Lambda_{\geq },I_\geq).\]
    
   A Verdier quotient in $\PrSt$ is a cofiber, hence a colimit.
   
    Since colimits commute, this allows us to use \autoref{coro: non-graded equivalence}. In this situation, we notice that $\Mod^\partial_{\bR-gr} (\Lambda_{\geq },I_\geq)_{\bR_\delta}$ is equivalent to $\Mod^\partial(\Lambda_{\geq },I_\geq) $ under \autoref{coro: non-graded equivalence}, which implies that
    \[\Sh_{\geq }(\bR)_{\bR_\delta}=\Mod_{\bR-gr} (\Lambda_{\geq })_{\bR_\delta}/\Mod^\partial_{\bR-gr} (\Lambda_{\geq },I_\geq)_{\bR_\delta}=\Mod(\Lambda_{\geq })/\Mod^\partial (\Lambda_{\geq },I_\geq).\]

    For the monoidality of the equivalence, we should be aware that the tensor product is identified with the Day convolution on the non-equivariant level by \autoref{prop: APT}, which is also identified with the usual tensor product on $\Mod(\Lambda_\geq)$ that naturally descends to $\Moda(\Lambda_\geq)$.
\end{proof}

Since the almostification functor commutes with both limits and colimits, the Greenlees--May duality \eqref{equation: Greenlees--May duality} is still true for almost modules on both sides. 
\begin{Coro}\label{corollary: adic-completion 1d}Denote by $\Moda(\Lambda_\bR)$ the essential image of $\Mod(\Lambda_\bR)$ under the almostification of $\Mod(\Lambda_\geq) \rightarrow \Moda(\Lambda_\geq)$ in $\Moda(\Lambda_\geq)$, then under the equivalence of \autoref{prop: equivariant APT}, we have
\[\Loc^{\bR_\delta}(\bR) \simeq \Moda(\Lambda_\bR).\]
Consequently, for
\[\Moda_{t-comp}(\Lambda_\geq)\eqqcolon (\Moda(\Lambda_\bR))^\perp\simeq ^\perp(\Moda(\Lambda_\bR))\coloneqq   \Moda_{t-tor}(\Lambda_\geq),\]
we have
\[\Moda_{t-comp}(\Lambda_\geq)\simeq (\Loc^{\bR_\delta}(\bR))^\perp \simeq \Sh^{\bR_\delta}_{>}(\bR) \simeq ^\perp(\Loc^{\bR_\delta}(\bR)) \simeq \Moda_{t-tor}(\Lambda_\geq).\]
\end{Coro}
\begin{proof}To prove the first equivalence, we only need to match generators on both sides. One generator of $\Loc^{\bR_\delta}(\bR)$ is $\bigoplus_\bR 1_\bR[1]$ equipped with the canonical $\bR$-equivariant structure. It is clear, on the sheaf side, that 
\[\bigoplus_\bR 1_\bR[1] \simeq \bigoplus_\bR \varinjlim_n (1_{(-\infty,n)})[1]=  \varinjlim_n\bigoplus_\bR (1_{(-\infty,n)})[1].\]

However, the sheaf $\bigoplus_\bR (1_{(-\infty,n)})[1]$ corresponds to $\varinjlim_n \Lambda_{\geq -n}=\Lambda_\bR$ as almost local modules, which is the generator of $\Moda(\Lambda_\bR)$. Notice that, in $\Mod(\Lambda_\bR)$, there are no almost-zero $\Lambda_\geq$-modules; hence we actually have $\Moda(\Lambda_\bR)\simeq \Mod(\Lambda_\bR)$.    

Under the equivalence \autoref{prop: equivariant APT}, the second statement is the almostification of the Greenlees--May duality \eqref{equation: Greenlees--May duality}.
\end{proof}

Now, we define $\widehat{\Lambda}_\geq$ to be the $t$-adic derived completion of ${\Lambda}_\geq$, i.e. $\widehat{\Lambda}_\geq=\varprojlim_n{\Lambda}_\geq\otimes_{\Lambda_\geq} (\Lambda_\geq/t^n\Lambda_\geq)=\varprojlim_n(\Lambda_\geq/t^n\Lambda_\geq)$ (derived limit). In the case where $\bfk$ is a discrete ring, notice that the inverse system $\Lambda_\geq/t^{n+1}\Lambda_\geq \to \Lambda_\geq/t^n\Lambda_\geq$ satisfies the Mittag-Leffler condition. Hence the derived limit is concentrated in degree $0$, and
$
\widehat{\Lambda}_\geq\cong \varprojlim_n^0(\Lambda_\geq/t^n\Lambda_\geq),
$
the non-derived limit, which gives the classical completion. Thus, the notation $\widehat{\Lambda}_\geq$ is compatible with the standard notion of the Novikov series ring (i.e. the classical $(t)$-adic completion), and it makes sense to simply refer to it as the completion without further indication.

In $\widehat{\Lambda}_\geq$, we consider the ($\bR$-graded) ideal $\widehat{I}_\geq$, which is defined as the derived completion of $I_\geq$ and is still idempotent. In the case where $\bfk$ is a discrete ring, one can show similarly that $\widehat{I}_\geq$ is the ideal of Novikov series with zero constant term. Thus, we may consider the almost content $(\widehat{\Lambda}_\geq,\widehat{I}_\geq)$, and then study the module category $\Mod(\widehat{\Lambda}_\geq)$ and the almost module category $\Moda(\widehat{\Lambda}_\geq)$. In $\Mod(\widehat{\Lambda}_\geq)$, we also consider the full subcategory $\Mod(\widehat{\Lambda}_\bR)$ of modules over $\widehat{\Lambda}_\bR=\widehat{\Lambda}_\geq[t^{-1}]$ (which is called the Novikov field if $\bfk$ is a field), and its orthogonal complement. In particular,
\[
\Mod_{t-comp}(\widehat{\Lambda}_\geq)
\coloneqq
(\Mod(\widehat{\Lambda}_\bR))^\perp
\subset
\Mod(\widehat{\Lambda}_\geq)
\]
is the category of derived complete modules over $\widehat{\Lambda}_\geq$. We denote its almostification by $\Moda_{t-comp}(\widehat{\Lambda}_\geq)$.

We have the following result concerning $\widehat{\Lambda}_\geq$, which was proven by the first-named author \cite{Kuwagaki_almost_equivalence} via a different method.

\begin{Coro}Let $\widehat{\Lambda}_\geq$ be the $t$-adic completion of ${\Lambda}_\geq$, we have
\[\Sh^{\bR_\delta}_{>}(\bR) \simeq \Moda_{t-comp}(\widehat{\Lambda}_\geq),\quad \Sh^{\bR_\delta}_{>}(\bR) \simeq_a \Mod_{t-comp}(\widehat{\Lambda}_\geq),\]
where $\simeq_a$ is the almost equivalence defined in \cite{Kuwagaki_almost_equivalence} (for $\mathbb{G}=\bR$ therein).
\end{Coro}
\begin{proof}By \cite[Lemma 15.93.25]{StacksProject}, we have
\[\Mod_{t-comp}({\Lambda}_\geq) \simeq \Mod_{t-comp}(\widehat{\Lambda}_\geq).\]

Its almostification gives (we use $(\widehat{\Lambda}_\geq,\widehat{I}_\geq)$ as the almost content for the right-hand side)
\[\Moda_{t-comp}({\Lambda}_\geq) \simeq \Moda_{t-comp}(\widehat{\Lambda}_\geq).\]

Then, by \autoref{corollary: adic-completion 1d}, we have
\[\Sh^{\bR_\delta}_{>}(\bR) \simeq \Moda_{t-comp}(\widehat{\Lambda}_\geq).\]

It remains to verify that $\Moda_{t-comp}(\widehat{\Lambda}_\geq)\simeq_a \Mod_{t-comp}(\widehat{\Lambda}_\geq)$, which follows from the definition of almost equivalence $\simeq_a$ in \cite{Kuwagaki_almost_equivalence}.
\end{proof}

Lastly, we remark that the graded version is also true (we refer again to \autoref{remark: filtration vs grading} for the notion of grading):
\begin{Prop}\label{corollary: adic-completion 1d graded version}Denote by $\Moda_{\bR-gr}(\Lambda_\bR)$ the essential image of $\Mod_{\bR-gr}(\Lambda_\bR)$ under the almostification of $\Mod_{\bR-gr}(\Lambda_\geq) \rightarrow \Moda_{\bR-gr}(\Lambda_\geq)$ in $\Moda_{\bR-gr}(\Lambda_\geq)$, then under the equivalence \autoref{prop: APT}, we have $\Loc(\bR) \simeq \Moda_{\bR-gr}(\Lambda_\bR)$. 

Then for
\[\Moda_{\bR-gr,t-comp}(\Lambda_\geq)\eqqcolon (\Moda_{\bR-gr}(\Lambda_\bR))^\perp\simeq ^\perp(\Moda_{\bR-gr}(\Lambda_\bR))\coloneqq   \Moda_{\bR-gr,t-tor}(\Lambda_\geq),\]
we have
\[\Moda_{\bR-gr,t-comp}(\Lambda_\geq)\simeq (\Loc(\bR))^\perp \simeq \Sh_{>}(\bR) \simeq ^\perp(\Loc(\bR)) \simeq \Moda_{\bR-gr,t-tor}(\Lambda_\geq),\]
and 
\[\Sh_{>}(\bR) \simeq \Moda_{\bR-gr,t-comp}(\widehat{\Lambda}_\geq),\quad \Sh_{>}(\bR) \simeq_a \Mod_{\bR-gr,t-comp}(\widehat{\Lambda}_\geq),\]
where $\Mod_{\bR-gr,t-comp}(\widehat{\Lambda}_\geq)\coloneqq (\Mod_{\bR-gr}(\widehat{\Lambda}_\bR))^\perp
\subset
\Mod_{\bR-gr,t-comp}(\widehat{\Lambda}_\geq)$ and $\simeq_a$ is the notion of almost equivalence of \cite{Kuwagaki_almost_equivalence} (for $\mathbb{G}=\{0\}$ therein).

\end{Prop}
\begin{RMK}We can describe the complete condition, in terms of $\Sh_{>}(\bR)$ by requiring that for $F\in \Sh_{\geq}(\bR)$ with $\Gamma(\bR,F)=0$, and in terms of grading by requiring that for $M\in \Mod_{\bR-gr}(\widehat{\Lambda}_\geq)$ with $\varprojlim_{a\rightarrow -\infty}M(a)=0$.  
\end{RMK}
By a standard categorical discussion, we have
\begin{Coro}There exist embeddings
  \[\Sh_{>}(\bR) \hookrightarrow \Mod_{\bR-gr}(\widehat{\Lambda}_\geq),\quad \Sh^{\bR_\delta}_{>}(\bR) \hookrightarrow \Mod(\widehat{\Lambda}_\geq),\]  
  whose essential images consist of almost local and complete $\widehat{\Lambda}_\geq$-modules.
\end{Coro}

\section{\texorpdfstring{$\arbR$-action on categories}{}
}\label{section: Categories admits R-action}
\subsection{\texorpdfstring{Categories equipped with an $\arbR$-action and $\bR$-filtered categories}{}}
For a category $\cC$, we say that an $\arbR$-action on $\cC$ is a monoidal functor $\nT:\arbR \rightarrow \Fun(\cC,\cC)$. We usually denote $\nT(r): \cC\to\cC$ by $\nT_r$, and $\nT(0 \leq c)$ by $\tau_c$. The $\arbR$-action automatically factors through $\nT:\arbR \rightarrow \Fun^{L}(\cC,\cC)$ since the underlying poset of $\arbR$ is a group.

In the following discussion, we often assume $\cC$ is presentable. The presentability allows us to switch between the two notions: 1) A category that admits a left action of a monoidal category. In the $\infty$-categorical setting, we refer to \cite[Section 4.8.3]{HA}. 2) A category that is enriched over a monoidal category. In the $\infty$-categorical setting, we refer to \cite{gepner2015enriched}. 

We refer to \cite{Heine_enriched_category2023} for the precise statement of the equivalence of these two notions (where a necessary condition is also given). In practice, presentability can be easily achieved as we usually take $\cC$ as the ($dg$-nerve of the $dg$) category of $A_\infty$-modules over a small $A_\infty$-category, or a category of sheaves over a topological space.

By \cite[Proposition 2.2.3]{Rotation-inv-Lurie}, there exists a small ${\bfk}$-linear stable category $\operatorname{Rep}^{fin}(\arbR)$ of finitely generated $\bR$-filtered modules such that ${\Rep}(\arbR)\simeq \Ind(\operatorname{Rep}^{fin}(\arbR))$. We define the corresponding polynomial Novikov module categories via $ \Mod^{fin}_{{\bR-gr}}(\Lambda_{\geq}) \simeq {\Rep}^{fin}(\arbR)$. We have that $\Lambda_{\geq a}=\nT_{-a}\Lambda_{\geq}=t^a\Lambda_\geq$ for the degree $a$ generator $t^a$ of $\Lambda_\geq$ and $\Lambda_{[a, b)}=t^a\Lambda_{\geq}/t^b\Lambda_{\geq}$ are in $ \Mod^{fin}_{{\bR-gr}}(\Lambda_{\geq})$, and $ \Mod^{fin}_{{\bR-gr}}(\Lambda_{\geq})$ is the smallest ${\bfk}$-linear that stable category that contains all $\Lambda_{\geq a}$.

Applying {\cite[Proposition 2.2.7]{Rotation-inv-Lurie}} to $\cC=\arbR$, we have the following 
\begin{Prop}For a ${\bfk}$-linear stable category $\cD$, we have the equivalence
\[\Fun^{ex}(\operatorname{Rep}^{fin}(\arbR),\cD)\simeq \Fun(\arbR,\cD).\]    
\end{Prop}

\begin{Coro}For a presentable ${\bfk}$-linear stable category $\cD$, we have the equivalence
\[\Fun^{L}({\Rep}(\arbR),\cD)\simeq \Fun(\arbR,\cD).\]    
\end{Coro}
\begin{proof}When $\cD$ is presentable, we have
 \[\Fun^{L}({\Rep}(\arbR),\cD)\simeq \Fun^{L}(\Ind(\operatorname{Rep}^{fin}(\arbR)),\cD)\simeq \Fun^{ex}(\operatorname{Rep}^{fin}(\arbR),\cD).\qedhere\]       
\end{proof}

As discussed in Section~\ref{section: module discussion 1D}, we can equip ${\Rep}(\arbR)$ with a monoidal structure given by the Day convolution. If $\cD$ is monoidal, then the above equivalence preserves monoidal functors by a similar discussion as \cite[Corollary 2.4.4]{Rotation-inv-Lurie} (replace $\bZ$ therein by $\arbR$). Then we apply the above result to $\cD=\Fun^{ex}(\cC,\cC)$ or $\cD=\Fun^L(\cC,\cC)$ to conclude that
\begin{Prop}\label{proposition: R action iff Novikov action}An $\arbR$-action on the ${\bfk}$-linear stable category $\cC$ is equivalent to a ${\Rep}^{fin}(\arbR)\simeq \Mod^{fin}_{{\bR-gr}}(\Lambda_{\geq})$-action on $\cC$. When $\cC$ is presentable, such an action is also equivalent to a ${\Rep}(\arbR)\simeq \Mod_{{\bR-gr}}(\Lambda_{\geq})$-action on $\cC$; in the presentable case, $\cC$ can be upgraded to be a category enriched over ${\Rep}(\arbR)\simeq \Mod_{{\bR-gr}}(\Lambda_{\geq})$ in the sense of \cite{gepner2015enriched}, which we call $\bR$-filtered categories.

Under the correspondence, the functor $\nT_c$ is given by $t^{-c}\Lambda_{\geq}$-action. 
\end{Prop}
\begin{RMK}By the graded version of Greenlees--May duality \eqref{equation: Greenlees--May duality}, if $\Lambda_\bR$ acts trivially on $\cC$, then the presentable stable category $\cC$ can be upgraded to be a $\widehat{\Lambda}_\geq$-enriched category.
\end{RMK}

\begin{eg}For a locally compact Hausdorff space $X$, we consider the category of sheaves $\Sh(X;\cA)$ with values in $\cA={\Rep}(\arbR)\simeq \Mod_{{\bR-gr}}(\Lambda_{\geq})$ or $\cA={\Pers}(\arbR)\simeq \Moda_{{\bR-gr}}(\Lambda_{\geq})\simeq \Sh_{\geq }(\bR)$. Then $\Sh(X;\cA)$ has an $\bR$-action inherited from $\cA$. It is clear that it is a module over $\cA$ via $\Sh(X;\cA)\simeq \Sh(X)\otimes \cA$. When $X=M$ is a smooth manifold, we proved that $\Sh(M;\Sh_{\geq }(\bR))\simeq \Sh_{\geq }(M\times \bR)$ (cf. the proof of \cite[Proposition 5.5]{Hochschild-Kuo-Shende-Zhang} or \cite[Theorem 7.4]{Cutoff-Zhang}).
\end{eg}

\subsection{Continuous stabilization and persistence categories}
We first recall results from \cite{Efimov-K-theory}.

\begin{Def}\label{definition: compactly assembled/dualizble category}Let $\cA$ be an accessible $\infty$-category with filtered colimits. Then $\cA$ is called compactly assembled if the colimit functor $\varinjlim:\Ind(\cA)\to\cA$ has a left adjoint $\hat{\cY}: \cA\to \Ind(\cA)$\footnote{The usual Yoneda embedding is the right adjoint of the colimit functor.}. 

We say a presentable ${\bfk}$-linear stable category $\cC$ is dualizable if it is dualizable with respect to the ${\bfk}$-linear Lurie tensor product. 
\end{Def}
\begin{Prop}[{\cite[Proposition D.7.3.1]{SAG}}]\label{prop: dual iff stable+assemb}A presentable ${\bfk}$-linear stable category $\cC$ is dualizable if and only if it is compactly assembled. 
\end{Prop}

\begin{eg}The poset $\arbRpt\coloneqq \arbR\cup\{+\infty\}$ is compactly assembled, and we have $\hat{\cY}(x)\coloneqq ``\varinjlim" y$ taken over $y<x$. Moreover, the symmetric monoidal structure on $\arbR$ extends to a symmetric monoidal structure on $\arbRpt$. Then $\arbRpt$ is a symmetric monoidal compactly assembled category. But $\arbRpt$ is not stable: Suppose $\arbRpt$ is stable, then $\arbRpt$ should have a zero object, in particular, an initial object, which does not exist since the smallest real number does not exist. 
\end{eg}
\begin{eg}For an almost content $(R,I)$, it is explained in \cite[Proposition 2.9.12]{Fake_sheaves_on_mfd} that $\Moda(R,I)$ is dualizable. In general, all dualizable categories are of this form \cite[Theorem 2.10.16]{Fake_sheaves_on_mfd}.    
\end{eg}

In \cite{Efimov-K-theory}, Efimov describes the left adjoint to the (non-full) inclusion $\Cat^{\dual}\to\Compass$ from the dualizable categories to the compactly assembled categories:

For an accessible $\infty$-category $\cA$ with filtered colimits, we denote by $\Stab^{\cont}(\cA)$ the (presentable stable) category of cofiltered limit-preserving functors $\cA^{\op}\rightarrow\Sp$. 

The inclusion $\Stab^{\cont}(\cA)\rightarrow\Fun(\cA^{\op},\Sp)$ commutes with limits and filtered colimits, hence it has a left adjoint $F\mapsto F^{\sharp}$ by the adjoint functor theorem.

For $x\in\cA,$ we denote by $h_x:\cA^{\op}\to\Sp$ the ``representable'' presheaf, given by $h_x(y)=\bS[\cA(y,x)],$ $y\in\cA.$ For any presentable stable category $\cD$ we have an equivalence
\begin{equation}\label{eq:cont_stab_general} {\Fun^{L}(\Stab^{\cont}(\cA),\cD)}\xrightarrow{\simeq} { \Fun^{\cont}(\cA,\cD)} ,\quad {F} \mapsto[x\mapsto {F} (h_x^{\sharp})].\end{equation}

\begin{Prop}[{\cite[Proposition 1.35]{Efimov-K-theory}}]\label{prop: continuous_stab}For a compactly assembled category $\cA$, $\Stab^{\cont}(\cA)$ is dualizable. The left adjoint to the inclusion $\Cat^{\dual}\to \Compass$ is given by $\cA\mapsto\Stab^{\cont}(\cA)$. Moreover, when $\cA$ is monoidal, $\Stab^{\cont}(\cA)$ is also monoidal.
\end{Prop}

The construction clearly works ${\bfk}$-linearly by replacing $\Sp$ with $\Mod({\bfk})$; we will use the ${\bfk}$-linear version below without further mention.

For $\cA=\arbRpt$, \autoref{prop: APT} can be interpreted as follows:
\begin{Prop}We have $\Stab^{\cont}(\arbRpt) \simeq \Sh_{\geq }(\bR)$. Under the identification, the monoidal structure on $\arbRpt$ induces the convolution on $\Sh_{\geq }(\bR)$.

In particular, by \eqref{eq:cont_stab_general}, for a presentable ${\bfk}$-linear category $\cD$, we have 
\begin{equation}\label{eq: stablization for R}
    \Fun^{L}(\Sh_{\geq }(\bR),\cD)\xrightarrow{\simeq} \Fun^{\cont}(\arbRpt,\cD),\quad F\mapsto(x\mapsto F(h_x^{\sharp})).
\end{equation}
\end{Prop}

Now, let us apply the construction to our discussion for categories with $\arbR$-action.
\begin{Def}We say an $\arbR$-action on a stable category $\cC$ is continuous if the monoidal functor $\arbR \rightarrow \Fun^L(\cC,\cC)$ maps filtered colimits to filtered colimits.     
\end{Def}

In this case, there exists a canonical way to extend a continuous $\arbR$-action to a continuous $\arbRpt$-action (defined in the same way) by Kan extension. Then we have
\begin{Prop}\label{prop: stable R action implies Tamarkin module}For a presentable ${\bfk}$-linear stable 
 category $\cC$, a continuous $\arbR$-action on $\cC$ is equivalent to a $\Sh_{\geq }(\bR)$-action on $\cC$. In this case, $\cC$ can be upgraded to be a $\Sh_{\geq }(\bR)$-enriched category.
\end{Prop}
\begin{proof}We extend the $\arbR$-action to an $\arbRpt$-action. Now, we apply \eqref{eq: stablization for R} to $\cD=\Fun^L(\cC,\cC)$.
\end{proof}

\begin{eg}\label{example: Tamarkin categories for opens}For an open set $U\subset T^*M$, we defined the $\infty$-version of the Tamarkin category $\sT(U)$, and proved that it has a $\Sh_{>}(\bR)$-action in \cite{Hochschild-Kuo-Shende-Zhang}. Thus, $\sT(U)$  automatically has a $\Sh_{\geq }(\bR)$-action (recall that $\Sh_{>}(\bR)$ is a monoidal quotient of $\Sh_{\geq }(\bR)$). The $\Sh_{\geq }(\bR)$-action is indeed the convolution action. In this example, the $\Sh_{\geq }(\bR)$-action descends to a $\Sh_{>}(\bR)$-action, which implies that $\sT(U)$ is (filtered-)$\widehat{\Lambda}_\geq$-enriched.   
\end{eg}

\subsection{Tamarkin torsion objects}\label{section:Tamarkin torsion}
For an $\arbR$-action on a stable category $\cC$, equivalently, a $\Mod^{fin}_{{\bR-gr}}(\Lambda_{\geq})$-action on $\cC$, we can study Tamarkin torsion objects \cite{tamarkin2013}.

The natural transformation $\tau_c:\id \rightarrow \nT_c$ given by $0\leq c$ gives a morphism $\tau_c(F): F\rightarrow \nT_c(F)$ for $F\in \cC$. 

\begin{Def}\label{Def:Tamarkin torsion}For $c>0$, we say $F\in \cC$ is Tamarkin $c$-torsion (short for $c$-torsion, or just torsion) if $\tau_c(F)=0$.   
\end{Def}

\begin{eg}\label{example: torsion objects}In $\Mod^{fin}_{{\bR-gr}}(\Lambda_{\geq})$, we know that $\Lambda_{[0,c)} \coloneqq \Lambda_{\geq }/t^c\Lambda_{\geq }$ is $c$-torsion. Torsion objects are plentiful since for every $F\in \cC$ in a stable ${\bfk}$-linear category $\cC$ with an $\bR$-action, we have $\Lambda_{[0,c)}\cdot F$ is $c$-torsion.    
\end{eg}

\begin{Prop}The full subcategory $\operatorname{Tor}(\cC)$ of $\cC$ spanned by torsion objects is stable.    
\end{Prop}
\begin{proof}Consider the following diagram:
\[\begin{tikzcd}
A \arrow[d, "\tau_c(A)"] \arrow[r] & B \arrow[d, "\tau_c(B)"] \arrow[r] & C \arrow[d, "\tau_c(C)"] \\
T_{c}A \arrow[d] \arrow[r]         & T_{c}B \arrow[d] \arrow[r]         & T_{c}C \arrow[d]         \\
{A[1]\oplus T_{c}A} \arrow[r]      & {B[1]\oplus T_{c}B} \arrow[r]      & {C[1]\oplus T_{c}C}     .
\end{tikzcd}\]
Suppose the first row is a fiber sequence, then the second and third rows are fiber sequences. Choose $c$ such that $A,B$ are both $c$-torsion. Then the first two columns are fiber sequences. We thus conclude the third column is a fiber sequence and $C$ is $c$-torsion by the 9-lemma or because colimits commute. Thus, $\operatorname{Tor}$ is a stable subcategory.    
\end{proof}

\begin{Def}We define the torsion-free category $\cC_\infty$ as the Verdier quotient $ \cC/\operatorname{Tor}(\cC) $ in the category of ${\bfk}$-linear stable categories.    
\end{Def}
\begin{RMK}Caution! There is no reason for $\cC_\infty$ to be presentable even when $\cC$ is.    
\end{RMK}

The following is a chain-level generalization of \cite[Proposition 5.7]{GS2014} (for $\cC=\Sh_{\geq }(M\times \bR)$).
\begin{Prop}\label{proposition: torsion free hom computation}For $F,G\in \cC_\infty$, we have
\[\HOM_{\cC_\infty}(F,G)=\varinjlim_{c\rightarrow \infty}\HOM_{\cC}(F,\nT_{c}G).\]
\end{Prop}
\begin{proof}By \cite[Theorem I.3.3]{Nikolaus-Scholze-Ontopologicalcyclichomology}, the left-hand side is computed by
\[\varinjlim_{(T\rightarrow G) \in \operatorname{Tor(\cC)}_{/G}}\HOM_{\cC}(F,\cofib{T\rightarrow G}).\]

Now, for $c\geq 0$, we have a morphism $c\mapsto [\Lambda_{[-c,0)} \cdot G[-1] \rightarrow G]$, which induces a functor $[0,\infty) \rightarrow \operatorname{Tor}(\cC)_{/G}$ (notice that $\Lambda_{[-c,0)} \cdot G[-1]$ is Tamarkin torsion). Moreover, the cofiber of $\Lambda_{[-c,0)} \cdot G[-1] \rightarrow G$ is $\nT_{c}G$, so we only need to show that the functor is cofinal.
 
To prove the cofinality, we use Quillen's theorem A. It remains to check that, for all $X=[T\rightarrow G]\in \operatorname{Tor}(\cC)_{/G}$, the fiber product $[0,\infty) \times_{\operatorname{Tor}(\cC)_{/G}} {X\backslash}(\operatorname{Tor}(\cC)_{/G})$ is contractible. For a given $X=[T\rightarrow G]$ where $T$ is a torsion object, objects of the fiber product are decompositions of $X$ of the form
\[
    T\rightarrow  \Lambda_{[-c,0)} \cdot G[-1]   \rightarrow G.
\]

We set $S$ to be the essential image of the projection functor  $[0,\infty) \times_{\operatorname{Tor}(\cC)_{/G}} {X\backslash}(\operatorname{Tor}(\cC)_{/G}) \rightarrow [0,\infty)$, and we set $c_X=\inf S$ (here as a subset of $\bR$).

Since $T$ is torsion, we have $\Lambda_{[-c,0)} \cdot T[-1]= T \oplus \nT_{c}T[-1]$ for some $c \geq 0$. Then $T\rightarrow G$ factors through
\begin{equation}\label{equation: torsion factor through}
    T\rightarrow T \oplus \nT_{c}T[-1] = \Lambda_{[-c,0)} \cdot T[-1] \rightarrow \Lambda_{[-c,0)} \cdot G[-1] \rightarrow G.
\end{equation}

It implies that $S$ is non-empty and $c_X <\infty$.

Moreover, if $c\in S$ and $c\leq d$, then we have a morphism in $[0,\infty) \times_{\operatorname{Tor}(\cC)_{/G}} {X\backslash}(\operatorname{Tor}(\cC)_{/G}) $
\[\begin{tikzcd}
T \arrow[d,equal] \arrow[r] &   {\Lambda_{[-c,0)} \cdot G[-1]} \arrow[r] \arrow[d] & G \arrow[d,equal] \\
T \arrow[r]                     & {\Lambda_{[-d,0)} \cdot G[-1]} \arrow[r]           & G     .     
\end{tikzcd}
\]

So, we have $d\in S$ and $S$ is an interval of the form $(c_X,\infty)$ or $[c_X,\infty)$. We do not need to know if $c_X\in S$.

The same construction also gives a functor from $S$ to $[0,\infty) \times_{\operatorname{Tor}(\cC)_{/G}} {X\backslash}(\operatorname{Tor}(\cC)_{/G}) $, which serves as the quasi-inverse of the projection functor since \eqref{equation: torsion factor through} is the only possible form of objects in $[0,\infty) \times_{\operatorname{Tor}(\cC)_{/G}} {X\backslash}(\operatorname{Tor}(\cC)_{/G})$ by the universal property of fiber/cofiber. Consequently, we have $S\simeq [0,\infty) \times_{\operatorname{Tor}(\cC)_{/G}} {X\backslash}(\operatorname{Tor}(\cC)_{/G}) $. So, since $S$ is an interval of the form $(c_X,\infty)$ or $[c_X,\infty)$ (in any case, a connected poset), we conclude that $[0,\infty) \times_{\operatorname{Tor}(\cC)_{/G}} {X\backslash}(\operatorname{Tor}(\cC)_{/G}) $ is contractible. 
\end{proof}

\subsection{Interleaving distance}\label{section: distance}
For a category $\cC$ equipped with an $\arbR$-action, we can study the interleaving distance on the set of objects. The idea was systematically studied by \cite{Ks2018,asano2017persistence,Thickeningkernel,TPC,guillermouviterbo_gammasupport} etc. It has important applications in quantitative symplectic geometry in \textit{loc. cit.}, \cite{zhang2020quantitative,2020topological_persistence}, and \cite{fukaya_distance}\footnote{Results therein are stated in terms of the (complete) Novikov ring language, which is proven to be equivalent to the present language using APT.} (and some references therein) where the role of almost mathematics in symplectic geometry was first observed.

Here, we collect ideas from them, and state some results well-known to experts. We refer to \textit{loc. cit.} for their proofs. 

\begin{Def}For $X,Y\in \cC$, we say $X,Y$ are $(a,b)$-isomorphic if there exist $\alpha: X\rightarrow \nT_a Y$ and $\beta: Y\rightarrow \nT_b X$ such that $\nT_a(\beta)\circ \alpha \simeq \tau_{a+b}(X)$ and $\nT_b(\alpha)\circ \beta \simeq \tau_{a+b}(Y)$. We define the interleaving distance as 
\[d_{\dint}(X,Y)=\inf\{a+b: X,Y\text{ are }(a,b)-\text{isomorphic}\}.\]
\end{Def}
Then we have the following result.
\begin{Prop}With the notation above, we have
\begin{enumerate}[fullwidth]
    \item $d_{\dint}$ is a pseudo-distance on the object set of $\cC$ (here $d_{\dint}$ may take the value $\infty$).
    \item For an $\bR$-equivariant functor $\alpha:\cC \rightarrow \cC'$ and $X,Y\in \cC$, we have
    \[d_{\dint}(\alpha(X),\alpha(Y))\leq d_{\dint}(X,Y).\] 
    
    \item If $\cC$ is stable, then for a fiber sequence $A\rightarrow B \rightarrow C$, we have
    \[d_{\dint}(A,0)\leq d_{\dint}(B,0)+d_{\dint}(C,0).\]
    \item If $\cC$ is stable, for composable morphisms $u,v$, we have
    \[d_{\dint}(\cofib{v\circ u},0)\leq d_{\dint}(\cofib{ u},0)+d_{\dint}(\cofib{v},0).\]

    \item If $\cC$ is presentable ${\bfk}$-linear stable, then for $X_0,X_1,Y_0,Y_1 \in \cC$, the $\Mod_{{\bR-gr}}(\Lambda_{\geq})$-enriched hom object $\underline{\HOM}$ satisfies
    \[d_{\dint}(\underline{\HOM}(X_0,Y_0) ,\underline{\HOM}(X_1,Y_1))  \leq  d_{\dint}(X_0,X_1)+d_{int}(Y_0,Y_1).\]
\end{enumerate}
\end{Prop}
We refer to \cite[Lemma 6.4]{guillermouviterbo_gammasupport} for the proof of (1)-(4), and \cite[Proposition 4.11]{asano2017persistence} for (5) (notice that (5) is in fact a result in $\Mod_{{\bR-gr}}(\Lambda_{\geq})$).

Using the distance, we can study the $d_{\dint}$-metric completeness on the set of objects $ \operatorname{Ob}(\cC)$. Then we have the following proposition, which comes from Guillermou--Viterbo~\cite{guillermouviterbo_gammasupport} and Asano--Ike~\cite{asanoike2022complete}: 

\begin{Prop}For a stable category $\cC$ equipped with an $\arbR$-action, we have
\begin{enumerate}[fullwidth]
    \item (\cite[Lemma 6.21]{guillermouviterbo_gammasupport}) Let $\bN\rightarrow \cC$ be an inductive system, with $f_n:X_n\rightarrow X_{n+1}$. If $\sum_k d_{\dint}(\cofib{f_k},0) < \infty$, then $X_n$ is $d_{\dint}$-convergent to $\varinjlim X_k$ (if it exists). Moreover, for the natural map $i_N:X_N\rightarrow \varinjlim X_k$, we have $d_{\dint}(\cofib{i_N},0)\leq 2\sum_{k \geq N}d_{\dint}(\cofib{f_k},0)$.
    \item If $\cC$ is closed under countable direct sums, then $d_{\dint}$ is Cauchy complete.
    \item (\cite[Lemma 6.24]{guillermouviterbo_gammasupport}) If the $\arbR$-action on $\cC$ is continuous, and $\{X_n\} \subset \operatorname{Ob}(\cC)$ is $d_{\dint}$-convergent to $X\in \operatorname{Ob}(\cC)$, then after passing to a subsequence, there exists a non-positive increasing sequence of reals $\varepsilon_n \rightarrow 0$ such that $X=\varinjlim_n \nT_{\varepsilon_n }(X_n)$. (We will construct the inductive system in the proof.)
\end{enumerate}
\end{Prop}
The proofs therein can be translated to our setting directly with minor modifications. We present them here.
\begin{proof}
\begin{enumerate}[fullwidth]
\item We fix $N \geq 1$. Set $f_n^m:X_n\rightarrow X_m$ for $m\geq n$. Therefore, we have $f_n^{m+1}=f_m\circ f_n^m$. Because colimits commute, we have
\[\varinjlim_{k\geq N} \cofib{f_N^k} \simeq \cofib{\varinjlim_{k\geq N} (f_N^k)}\simeq \cofib{\varinjlim_{k\geq N} X_N \rightarrow \varinjlim_{k\geq N} X_k}=\cofib{i_N}.\]
On the other hand, we have the fiber sequence
\[\oplus_{k\geq N}\cofib{f_N^k} \rightarrow \oplus_{k\geq N}\cofib{f_N^k} \rightarrow \varinjlim_{k\geq N} \cofib{f_N^k} .\]
Therefore, we have
\[d_{\dint}(\cofib{i_N},0)=d_{\dint}(\varinjlim_{k\geq N} \cofib{f_N^k},0)\leq 2 d_{\dint}(\oplus_{k\geq N}\cofib{f_N^k},0).\]

Now we estimate $d_{\dint}(\oplus_{k\geq N}\cofib{f_N^k},0)$. Since $f_N^{m+1}=f_N\circ f_N^m$, we have
\[d_{\dint}(\cofib{f_N^{m+1}},0) \leq d_{\dint}(\cofib{f_N^{m}},0)+d_{\dint}(\cofib{f_N},0) \leq \cdots \leq \sum_{n\geq N}^m d_{\dint}(\cofib{f_n},0) .\]
Therefore, we have
\[d_{\dint}(\oplus_{k\geq N}\cofib{f_N^k},0) \leq d_{\dint}(\cofib{f_N^{k}},0) \leq  \sum_{n\geq N}^\infty d_{\dint}(\cofib{f_n},0)  .\]

\item The condition allows us to take sequential colimits in $\cC$. This is then a consequence of (1) with the help of a version of \cite[Lemma 6.8]{guillermouviterbo_gammasupport} or \cite[Lemma 4.1, 4.2]{asanoike2022complete}, where the proof therein uses only the stability of $\cC$ and the existence of $\bR$-action.

\item Let $\alpha_n=2^{-n}$, after passing to subsequences, we can assume that $d_{\dint}(X_n,X)\leq \alpha_n$. Then by definition of $d_{\dint}$, there exists $\beta_n$ with $0\leq \beta_n \leq \alpha_{n+1}$ and morphisms 
\[v_n: X\rightarrow \nT_{\beta_n }X_n, \quad w_n: \nT_{\beta_n}X_n \rightarrow \nT_{\alpha_{n+1} }X\]
such that $w_n\circ v_n=\tau_{\alpha_{n+1}}(X)$.

Let $\varepsilon_n=\beta_n-\alpha_n$. We first consider the inductive sequence
\begin{align*}
 \cdots \rightarrow   \nT_{(-\alpha_{n}) }X \rightarrow \nT_{\varepsilon_n }X_n \rightarrow \nT_{(-\alpha_{n+1}) }X \rightarrow\nT_{\varepsilon_{n+1} }X_{n+1}\rightarrow \nT_{(-\alpha_{n+2}) }X \rightarrow \cdots ,
\end{align*}
where morphisms are defined by translations of $v_n,w_n$ with respect to canonical positive numbers.

Now, for two subsequences of the inductive system $(\nT_{\varepsilon_n }X_n)_n$ and $(\nT_{(-\alpha_n) }X)_n$, 
we have
\[\varinjlim_n \nT_{\varepsilon_n }X_n \simeq \varinjlim_n \nT_{(-\alpha_n) }X.\]
So, we only need to show that $\varinjlim_n \nT_{(-\alpha_n) }X\simeq X$, where the consecutive morphisms are $\nT_{(-\alpha_{n})}\tau_{\alpha_{n+1}}(X)$ since $w_n\circ v_n=\tau_{\alpha_{n+1}}(X)$. This is true because $-\alpha_n \nearrow  0$ and the $\bR$-action is continuous.\qedhere
\end{enumerate}
    
\end{proof}

\subsection{Enhanced triangulated persistence category}\label{section: TPC}

We have the following result concerning an enhancement of triangulated persistence categories (TPC) in the sense of \cite{TPF}. We are not going to fully explain the definition therein, but instead use the following version as explained in \cite[Remark 2.23-(d)]{TPF}. A TPC can be regarded as a triangulated 1-category $\mathcal{T}$ equipped with a triangulated $\arbR$-action, and in addition: For every object $F$, there exists a distinguished triangle $F \rightarrow \nT_r (F) \rightarrow K \xrightarrow{+1}$ such that $K$ is Tamarkin torsion\footnote{In \textit{loc. cit.}, $r$-torsion objects are called $r$-acyclic objects.} (\autoref{Def:Tamarkin torsion} works for triangulated categories). However, in \autoref{example: torsion objects}, we observe that this extra condition is automatic for the enhanced version.
\begin{Prop}\label{prop: R action implies persistence}If $\cC$ is a stable ${\bfk}$-linear category equipped with an $\arbR$-action, then its homotopy category ${h\cC}$ is a TPC.
\end{Prop}
\begin{proof}By \cite[Theorem 1.1.2.14]{HA}, we have that $\mathcal{T}={h\cC}$ is a triangulated category such that distinguished triangles descend from fiber sequences. The $\arbR$-action on $\cC$ naturally descends to $\mathcal{T}={h\cC}$. 

We observe in \autoref{example: torsion objects} that for any $F\in \cC$, we have a fiber sequence $\Lambda_{\geq}\cdot F \rightarrow t^{-c}\Lambda_{\geq}\cdot F \rightarrow \Lambda_{[-c,0)}\cdot F $ for $c\geq 0$, and $\Lambda_{[-c,0)}\cdot F $ is Tamarkin torsion. The fiber sequence descends to a distinguished triangle in ${h\cC}$ where $K=\Lambda_{[-c,0)}\cdot F $. So, $\mathcal{T}={h\cC}$ is a TPC.\end{proof}

\subsection{K-theory}
Here, we relate a result of Biran--Cornea--Zhang to continuous $K$-theory.

For a dualizable (recall \autoref{definition: compactly assembled/dualizble category}) presentable stable category $\cC$ equipped with an $\arbR$-action, we can define its algebraic $K$-theory using Efimov's continuous $K$-theory \cite{Efimov-K-theory}. 

\begin{Def}For a dualizable stable presentable category $\cC$ equipped with an $\arbR$-action, we define its ``persistence" $K$-theory as its continuous $K$-theory $K^{\cont}(\cC)$.
\end{Def}
As $K^{\cont}$ is a lax monoidal functor, for $\cC$ equipped with an $\arbR$-action, we have that $K^{\cont}(\cC)$ is a module over the ring spectrum $K^{\cont}(\Mod_{\bR-gr}(\Lambda_{\geq}))$ by \autoref{proposition: R action iff Novikov action}. Then, in particular, $\pi_*K^{\cont}(\cC)$ is a graded module over $\pi_*K^{\cont}(\Mod_{\bR-gr}(\Lambda_{\geq}))$, and then a module over $\pi_0K^{\cont}(\Mod_{\bR-gr}(\Lambda_{\geq}))$.

The following computation is a translation of \cite[Proposition 4.22-(2)]{Efimov-K-theory}. The form presented here is superficially different from that in \textit{loc. cit.}, but is essentially equivalent. For the convenience of readers, we sketch its proof here.
\begin{Prop}[Efimov]\label{prop: continuous K 1d}We have 
\[K^{\cont}(\Mod_{\bR-gr}(\Lambda_{\geq}))=\bigoplus_{\bR}K({\bfk}).\]   

In particular, for a principal ideal domain $\bfk$ or the sphere spectrum $\bS$, we have
\[K_0^{\cont}(\Mod_{\bR-gr}(\Lambda_{\geq}))=\bigoplus_{\bR}K_0({\bfk})=\bZ[\bR],\]
and for each $a$, the $a$-graded piece of the right-hand side is generated by $\Lambda_{\geq a}$.
\end{Prop}
\begin{proof}By \autoref{prop: APT} or \cite[Proposition 4.22-(1)]{Efimov-K-theory}, we have a fiber sequence in $\Catdual$ (rather than $\PrSt$):
\[\Sh_{\geq }(\bR) \rightarrow \Mod_{\bR-gr}(\Lambda_{\geq}) \rightarrow  \Mod_{\bR-gr}({\bfk}).\]
As $K^{\cont}$ is a localizing invariant on $\Cat ^{\dual}$, we have a fiber sequence
\[K^{\cont}(\Sh_{\geq }(\bR)) \rightarrow K^{\cont}(\Mod_{\bR-gr}(\Lambda_{\geq})) \rightarrow  K^{\cont}(\Mod_{\bR-gr}({\bfk})).\]

Therefore, by \cite[Proposition 4.22-(2)]{Efimov-K-theory}: $K^{\cont}(\Sh_{\geq }(\bR) )=0$, we have
\[K^{\cont}(\Mod_{\bR-gr}(\Lambda_{\geq}))=K^{\cont}(\Mod_{\bR-gr}({\bfk})).\]

As explained in the proof of \cite[Proposition 4.22-(2)]{Efimov-K-theory}, the category of compact objects of $\Mod_{\bR-gr}({\bfk})=\prod_{\bR} \Mod({\bfk})$ admits an $\bR$-indexed orthogonal decomposition 
\[\Mod_{\bR-gr}({\bfk})^\omega = \bigoplus_{\bR} \Mod({\bfk})^\omega.\]
Since for a compactly generated stable category $\cC$ we have $K^{\cont}(\cC)=K(\cC^\omega)$ and the algebraic $K$-theory functor is localizing, we conclude:
\begin{equation}\label{equation: a randow equation in K computation}
    K(\Mod_{\bR-gr}({\bfk})^\omega) = K(\bigoplus_{\bR} \Mod({\bfk})^\omega)= \bigoplus_{\bR} K( \Mod({\bfk})^\omega)= \bigoplus_{\bR} K( {\bfk}).
\end{equation}

In fact, it is also explicitly presented in the proof of \cite[Proposition 4.22-(2)]{Efimov-K-theory} that $\Mod_{\bR-gr}(\Lambda_{\geq})^\omega$ admits an $\bR$-indexed orthogonal decomposition, where the $a$-graded piece is the full subcategory spanned by $\{M\otimes  \Lambda_{\geq a}: M\in \Mod({\bfk})^\omega\}$. The proof there shows that the forgetful functor $\Mod_{\bR-gr}(\Lambda_{\geq}) \rightarrow  \Mod_{\bR-gr}({\bfk})$ induces a $K$-equivalence on each graded piece. Therefore, in the direct sum decomposition $K^{\cont}(\Mod_{\bR-gr}({\bfk}))=K(\Mod_{\bR-gr}({\bfk})^\omega)=\bigoplus_{\bR} K( {\bfk}) $, the $a$-graded piece is generated by $\Lambda_{\geq a}$.
\end{proof}

\begin{Coro}The result \cite[Proposition 4.2.2]{TPK} is true: The Grothendieck group of the category of finitely generated filtered chain complexes over a field is isomorphic to $\bZ[\bR]$.    
\end{Coro}
\begin{proof}We notice that $\cD=\Mod^{fin}_{{\bR-gr}}(\Lambda_{\geq})$ is the enhanced version of the category of finitely generated filtered complexes, and $\cE=(\Mod_{\bR-gr}({\Lambda_\geq }))^\omega$ is the idempotent completion of $\cD$. Then by a theorem of Thomason (see \cite[Theorem A.3.2]{Hermitian-K-theory} and references therein), we have $K_0(h(\cD))=K_0(\cD) \hookrightarrow$ $K_0(\cE)\simeq \oplus_\bR K_0 ({\bfk})=\bZ[\bR]$, where the $a$-graded piece is generated by $\Lambda_{\geq a}$ (see \eqref{equation: a randow equation in K computation}). On the other hand, $\{\Lambda_{\geq a}\}_{a\in\bR}$ are all in $\cD$, as are their $K$-classes; this shows that $K_0(h(\cD))=\bZ[\bR]$.
\end{proof}

\section{Higher-dimensional APT}\label{section: HD APT}
In this section, we discuss the higher-dimensional case. 

\subsection{Filtered modules}
This section is parallel to \autoref{section: module discussion 1D}. In the following, we will only state the results, since their proofs are identical to the corresponding $1$-dimensional version with suitable changes of the notation.

Let $\gamma\subset \bR^n$ be a convex cone, i.e., for $v,w  \in \gamma$ and $r>0$, we have $r v\in \gamma$ and $v+w\in \gamma$. We use $\gamma$ to define a partially ordered group structure on $\bR^n$. Namely, the group structure is the usual addition, and the partial order is defined by $a\leq b$ if and only if $b\in a+\gamma$. We denote the corresponding symmetric monoidal category by $\arbR^{n}_\gamma$. We also set $(\bR^{n})^{ds}$ to be the set $\bR^n$ as a discrete monoidal category.

\begin{Def}We set $\Rep(\arbR^{n}_\gamma)\coloneqq \Fun(\arbR^{n}_\gamma,\Mod({\bfk}))$ and $\Mod_{\bR^n-gr}(\bfk)\coloneqq \Fun((\bR^{n})^{ds},\Mod({\bfk}))$.    
\end{Def}

They are symmetric monoidal via the Day convolution and the graded tensor product respectively. The tensor unit $\Lambda_{\gamma}$, which is also a commutative algebra object in the category of $\bR^n$-graded ${\bfk}$-module spectra, is defined via the formula $\Lambda_{\gamma}(a)={\bfk}$ for $a\in \gamma$ and $\Lambda_{\gamma}(a)=0$ otherwise. Forgetting the grading, $\Lambda_{\gamma}\simeq \Sigma_+^\infty {\gamma} \wedge \bfk$ is a commutative ring spectrum by viewing $\gamma$ as a discrete topological space.

\begin{RMK}In \autoref{def: rep 1d}, $\arbR$ is $\arbR_{[0,\infty)}^1$ here. Then we have $\Lambda_{[0,\infty)}=\Lambda_{\geq}$.
\end{RMK}

\begin{Prop}\label{prop: Rees construction HD}
    The $\bR^n$-filtered Rees construction induces a monoidal equivalence
\[\Rep(\arbR^{n}_\gamma)\simeq \Mod_{\bR^n-gr}(\Lambda_{\gamma})\coloneqq \Mod_{\Lambda_{\gamma}}(\Mod_{\bR^n-gr}(\bfk)).\]
In particular, both categories are compactly generated and hence presentable.
\end{Prop}

Now, we set $\bR_\delta^n$ to be the discrete additive group $\bR^n$. We have $\bR_\delta^n$-actions on both sides of the equivalence in \autoref{prop: Rees construction HD}. On the representation side, the action is induced from the action on $\arbR^n_\gamma$. On the graded module side, the action is given by shifting grading, for $b\in \bR_\delta^n$, we have
\[\nT_b ( M ) (a)\coloneqq M(a+b).\]

It is clear that the equivalence in \autoref{prop: Rees construction HD} is $\bR_\delta^n$-equivariant. Moreover, we have
\begin{Coro}\label{coro: non-graded equivalence HD}The forgetting the grading functor  $\Mod_{\bR^n-gr}(\Lambda_{\gamma})\rightarrow \Mod({\bfk})$ induces a monoidal equivalence 
\[\Mod_{\bR^n-gr}(\Lambda_{\gamma})^{{\bR_\delta^n}}\xrightarrow{\simeq} \Mod(\Lambda_{\gamma})=\Mod_{\Lambda_{\gamma}}(\Mod({\bfk})).\]

In particular, we have an equivalence of compactly generated categories
\[\Rep(\arbR^n_\gamma)^{\bR_\delta^n}\simeq \Mod(\Lambda_{\gamma}).\]    
\end{Coro}

\subsection{Almost modules and sheaves: ``Affine" case}

In this section, we consider a closed convex cone $\sigma \subset  \bR^n $. We recall the following well-known properties of cones.
\begin{Prop}[{\cite[Exercise 2.31]{Boyd_2004}}]For a closed convex cone $\sigma$ and its dual cone
\[\sigma^\vee=\{v \in \bR^n: \langle v, w \rangle \geq 0, \forall w\in \sigma\},\]
we have
\begin{enumerate}
    \item $\sigma=( \sigma ^\vee)^\vee$;
    \item $\sigma$ is proper (i.e., $\sigma\cap (-\sigma)=0$) \footnote{It is also called pointed in the literature. Here, we prefer to reserve the terminology \textit{pointed} for cones containing $0$.} if and only if $\sigma^\vee$ has non-empty interior.
\end{enumerate}
\end{Prop}
Now, we assume that $\sigma$ is proper, which implies that $\gamma =\sigma^\vee$ has non-empty interior. 

Recall that, in the previous subsection, we studied the polynomial Novikov ring $\Lambda_{\gamma}$ and its (graded) modules. Now, since $\gamma =\sigma^\vee$ has non-empty interior $\mathring{\gamma}$, we consider the $\bR^n$-graded ideal $I_\gamma=\bfk[{\mathring{\gamma}}]$ of $\Lambda_{\gamma}$: i.e., $ I_\gamma(x)=\bfk$ for $x \in  \mathring{\gamma} $ and $I_\gamma(x)=0$ otherwise. It is clear that $\mathring{\gamma}+\mathring{\gamma} = \mathring{\gamma}$. Therefore, we have $I_\gamma=I_\gamma \otimes_{\Lambda_{\gamma}} I_\gamma $, i.e., $I_\gamma$ is an idempotent ideal of $\Lambda_{\gamma}$. 

Now, we can consider $(\Lambda_{\gamma},I_\gamma)$ as an almost content and study almost modules over it. In this section, we would like to prove that
\begin{Thm}\label{prop: APT HD}For a proper closed convex cone $\sigma$ and $\gamma=\sigma^\vee$, we have an equivalence
\[\Sh_{\bR^n\times (-\sigma)} (\bR^n) \simeq \Moda_{{\bR^n-gr}} (\Lambda_{-\gamma},I_{-\gamma}),\, F\mapsto \{F(\mathring{{\gamma}}-a)\}_{a\in \bR^n}.\] 

We equip $\Sh_{\bR^n\times (-\sigma)} (\bR^n) $ with the $\star$-convolution, and the almost module category $\Moda_{{\bR^n-gr}} (\Lambda_{-\gamma},I_{{-\gamma}})$ with the tensor product of  modules over $\Lambda_{-\gamma}$. Then the equivalence is a monoidal equivalence.
\end{Thm}
\begin{RMK}Many minus signs appear in the statements of \autoref{prop: APT HD} and \autoref{corollary: Kan extension restrict to equivalence HD}. The reason is that we want to reduce the number of minus signs in their proofs. In later results, we will try to avoid this inconvenience by putting a suitable sign there.    
\end{RMK}

The main technical result to prove this theorem is the microlocal cut-off lemma \autoref{lemma: microlocal cut-off}. Here, let us review the notion of ${\gamma}$-topology. An open set $U$ (with respect to the usual topology of $\bR^n$) is called ${\gamma}$-open if $U+{\gamma} =U$. The set of ${\gamma}$-open sets forms a topology on $\bR^n$, which is denoted by $\bR^n_{{\gamma}}$, and there is a natural continuous map $\varphi_{{\gamma}}: \bR^n \rightarrow \bR^n_{{\gamma}}$. \autoref{lemma: microlocal cut-off} shows that, for $\sigma=\gamma^\vee$, the pair 
$(\varphi_{{\gamma}}^*,\varphi_{{\gamma}*})$ defines an adjoint equivalence $\Sh_{\bR^n\times (-\sigma)} (\bR^n) \simeq  \Sh(\bR^n_{{\gamma}})$.

Therefore, we only need to relate $\Sh(\bR^n_{{\gamma}})$ and $\Moda_{{\bR^n-gr}} (\Lambda_{-\gamma},I_{{-\gamma}})$. 

We start by recalling that $\gamma$ has non-empty interior $\mathring{\gamma}$, which is a $\gamma$-open set. Then we can define a functor: 
\[\iota: \arbR^{n}_{-\gamma} \rightarrow \operatorname{Open} (\bR^n_{{\gamma}})^{\op},\quad \iota(a)=\mathring{{\gamma}}-a.\]

It is indeed a functor since if $a\rightarrow b$ i.e., $b-a\in -\gamma$, then $\mathring{\gamma}-a\subset \mathring{\gamma}-b$.

Pre-composing with $\iota$ and shifting homological degree by $n$, we obtain a natural ``restriction" functor $\iota^*$ that fits into the commutative diagram, where $\Gamma$ is defined by $\Gamma(F) (a) \coloneqq \Gamma(\mathring{{\gamma}}-a,F)[-n]$;
\[
\begin{tikzcd}
{\PSh(\bR^n_{{\gamma}})} \arrow[rr, "\iota^*"]                    &  & {\Mod_{{\bR^n-gr}} (\Lambda_{-\gamma})=\Rep(\arbR^{n}_{-\gamma})} \\
{\Sh(\bR^n_{{\gamma}}).} \arrow[u, hook] \arrow[rru, "\Gamma"] &  &                                \end{tikzcd}\]

\begin{Lemma}\label{lemma: gamma open basis}The set of ${\gamma}$-open sets $\{\mathring{{\gamma}}-a\}_{a\in \bR^n}$ forms a basis of the ${\gamma}$-topology.
\end{Lemma}
\begin{proof}It is clear that $\{\mathring{{\gamma}}-a\}_{a\in \bR^n}$ are ${\gamma}$-open sets. It remains to prove that for any ${\gamma}$-open set $U$ and $x\in U$, there exists $a \in \bR^n$ such that $x\in \mathring{{\gamma}}-a \subset U$.

To prove it, we notice that $U$ is open in the usual topology by definition. So, we can pick an open ball neighborhood $B(x,r) \subset U$ of $x$, $d \in \mathring{{\gamma}}$ with $|d| <r $, and $a\coloneqq d-x$. Then we have $x=d-a \in \mathring{{\gamma}}-a$ and $\mathring{{\gamma}}-a \subset U$. 
\end{proof}
\begin{Coro}\label{corollary: Kan extension restrict to equivalence HD}The right Kan extension  $\iota_* :\Mod_{{\bR^n-gr}} (\Lambda_{{-\gamma}}) \rightarrow \PSh(\bR^n_{{\gamma}})$ of $\iota^*$ restricts to an adjoint equivalence $\iota_* :\Moda_{{\bR^n-gr}} (\Lambda_{{-\gamma}},I_{{{-\gamma}}}) \rightarrow \Sh(\bR^n_{{\gamma}})$. Moreover, the following diagram, where the right vertical hook arrow is the right adjoint of the almostification functor, is commutative;
\[
\begin{tikzcd}
{\PSh(\bR^n_{{\gamma}})} \arrow[rr, "\iota^*"]                    &  & {\Mod_{{\bR^n-gr}} (\Lambda_{{-\gamma}})} \\
{\Sh(\bR^n_{{\gamma}})} \arrow[u, hook] \arrow[rru, "\Gamma"]  \arrow[rru, "\Gamma"] \arrow[rr, "\iota^{*}","\simeq"']&  &    {\Moda_{{\bR^n-gr}} (\Lambda_{{-\gamma}},I_{{{-\gamma}}})}.\arrow[u, hook]                            \end{tikzcd}\]

\end{Coro}
\begin{proof}
The right Kan extension can be computed using the formula \[\iota_{*}M (U) =  \varprojlim_{({\mathring{{\gamma}}-a})\subset U } M(a)[n].\]

By the definition of Kan extension, the pair $(\iota^*,\iota_*)$ is an adjunction. We need to check that the unit and counit of the adjunction pair restrict to equivalences, which follows from the following two verifications:

If $M$ is almost local, we have
\[\iota^*(\iota_*M)(a)= (\iota_*M)(\mathring{{\gamma}}-a)[-n]= \varprojlim_{({\mathring{{\gamma}}-b})\subset ({\mathring{{\gamma}}-a})} M(b)=
 \varprojlim_{b\leq a }\HOM_{\Mod}(t^{b}\Lambda_{{{-\gamma}}},M)=M(a),\]
 where we use the almost local condition in the last equality; 
 
 If $F$ is an object in ${\Sh(\bR^n_{{\gamma}})}$, we have
\[F(U)=\varprojlim_{({\mathring{{\gamma}}-a})\subset U } F({\mathring{{\gamma}}-a}) = \varprojlim_{({\mathring{{\gamma}}-a})\subset U } \iota^*F(a)[n]=\iota_* \iota^*F(U),\]
where we use \autoref{lemma: gamma open basis} and the sheaf condition in the first equality.

The commutativity of the upper triangle is the definition of $\Gamma$. The commutativity of the lower triangle follows from the computation.
\end{proof}

\begin{proof}[Proof of \autoref{prop: APT HD}]The equivalence is a combination of \autoref{lemma: microlocal cut-off} and \autoref{corollary: Kan extension restrict to equivalence HD}.  

 Now, we discuss the monoidal structure. The idea is the same as in the 1-dimensional case \autoref{prop: APT}, it remains to show that 
 \[1_{\mathring{\gamma}}[n]\star 1_{\mathring{\gamma}}[n]\simeq 1_{\mathring{\gamma}}[n],\]
which is a straightforward computation.
\end{proof}

Lastly, we consider the ${\bR^n_\delta}$-action on both categories of \autoref{prop: APT HD} and take the ${\bR^n_\delta}$-homotopy fixed point.
\begin{Coro}\label{cor: Novikov-mirro-affine}
We have a monoidal equivalence
\[{\Sh^{\bR^n_\delta}_{\bR^n\times \sigma}(\bR^n)} \simeq \Moda(\Lambda_{\gamma}, I_{{{\gamma}}}). \]
\end{Coro}
\begin{RMK}In fact, \autoref{prop: APT HD} and \autoref{cor: Novikov-mirro-affine} slightly generalize statements of \cite{Vaintrob-logCCC} in the ``affine" case since he only states for proper closed \textit{polyhedral} cones $\sigma$, even though our proof has no essential difference between polyhedral and general proper closed cases. Moreover, our proof clearly shows that \autoref{prop: APT HD} and \autoref{cor: Novikov-mirro-affine} are actually monoidal equivalences, which is not very clear from Vaintrob's sketch of the proof. We also remark that an abelian-categorical and non-equivariant version of the results is proven in \cite{Berkouk_2021ephemeral} in terms of ephemeral persistence modules.
\end{RMK}

Here, we include the following application of \autoref{prop: APT HD} due to Efimov.

\begin{Prop}[Efimov]\label{prop: continuous K hd}For a non-zero proper closed convex cone $\sigma$, we have 
\[K^{\cont}(\Sh_{\bR^n\times \sigma}(\bR^n))=0,\quad K^{\cont}(\Mod_{\bR^n-gr}(\Lambda_{\sigma^\vee}))=\bigoplus_{\bR^n}K({\bfk}).\]

Under the second equivalence, the $a$($\in \bR^n$)-graded piece is generated by $\bfk[a+\sigma^\vee]$.
\end{Prop}
Its proof is the same as \autoref{prop: continuous K 1d} by replacing $[0,\infty)$ with $\gamma=\sigma^\vee$ and $\Lambda_{\geq a}$ by $\bfk[a+\gamma]=t^a\Lambda_{\gamma}$. So, we leave the details to the interested reader. 

\begin{RMK}\label{remark: HD k-theory}The higher-dimensional version is not explicitly written in \cite{Efimov-K-theory}, but it was shared with us by Alexander I. Efimov during the \textit{Masterclass: Continuous K-theory}, held at the University of Copenhagen in June 2024. We thank him for his generosity.
\end{RMK}

\begin{RMK}By \autoref{remark: HD k-theory} and \cite[Proposition 4.11]{Efimov-K-theory}, we can compute the continuous $K$-theory (in fact, all localizing invariants) of $\QCoha_{\bT^\Nov}(X_\Sigma^\Nov,\partial_\Sigma) \simeq \Sh_{M_\bR\times |\Sigma|}(M_\bR)$ in terms of $\Gamma_c(|\Sigma|,K(\bfk))$. Another proof can be found in \cite{Zhang_K_Remark}.  
\end{RMK}

\subsection{Cut-off for fans} \label{subsection: cut-off for fans}

We start from some notation. We fix a vector space $N_\bR$ of finite dimension and denote $M_\bR=N_\bR^\vee$. Here, for a closed convex cone $\sigma\subset N_\bR$, we treat its dual $\sigma^\vee$ as a closed convex cone inside $M_\bR$. We say a closed convex cone is polyhedral if it is a finite intersection of half-spaces. With this notation, we have a natural identification $T^*M_\bR=M_\bR\times N_\bR$.

\begin{Def}\label{def: fan}For a finite set of proper polyhedral closed convex cones in $N_\bR$, we say $\Sigma$ is a fan\footnote{In the usual literature of toric variety, the terminology ‘fan’ is used for what we call rational fan. Here, we do not require fan to be rational.} if: i) For $\sigma \in \Sigma$, and $\tau \prec \sigma$ is a face of $\sigma$, then we have $\tau \in \Sigma$. ii) If $\sigma_1,\sigma_2\in \Sigma$, then we have $\sigma_1\cap \sigma_2$ is a face of both $\sigma_1$ and $\sigma_2$. 

The support $|\Sigma|$ of $\Sigma$ is defined as $|\Sigma|=\cup_{\sigma \in \Sigma} \sigma$. We say $\Sigma$ is complete if $|\Sigma|=N_\bR$.
\end{Def}

Recall that, for a closed set $Z$, we set $1_Z$ to be the constant sheaf supported on the closed set $Z$ with stalk $\bfk$. If $Z\subset W$ is an inclusion of closed sets, then we have a morphism $1_{W}\rightarrow 1_Z$. We treat a fan as a direct set with respect to the face relation. We have the following technical lemma. 
\begin{Lemma}\label{lemma: star convolution limit}If $\Sigma$ is a complete fan, then we have an equivalence in $\Sh(N_\bR)$:
\[1_{N_\bR} \simeq \varprojlim_{\sigma \in \Sigma^\op} 1_{\sigma},\]
where morphisms are induced by closed inclusions induced by the face relation.    
\end{Lemma}
\begin{proof}It is clear that we have a morphism $1_{N_\bR} \rightarrow \varprojlim_{\sigma} 1_{\sigma}$. We can show it is a stalkwise-equivalence by the hypercompleteness of $\Sh(N_\bR)$. By the finiteness of the limit, it means that we need to prove that 
\[1\simeq \varprojlim_{\sigma} (1_{\sigma})_n,\quad \forall n\in N_\bR.\]

The basic observation is that a complete fan $\Sigma$ induces a regular CW decomposition of $S^{n-1}\subset N_\bR$ that consists of unit vectors indexed by $\Sigma\setminus \{ \{0\}\}$. Then the computation below reduces to computations of certain relative homology, and we remark that $1$ in the result of the limit corresponds to the fundamental class of $S^{n-1}$. We learned this idea from \cite[Section 3]{bressler2003intersection}.

We discuss two cases: 

1) If $n=0$, since $0\in \sigma$ for all $\sigma \in \Sigma$, then we need to show that $1\simeq \varprojlim_{\sigma} 1$. But $\Sigma\setminus \{ \{0\}\}$ is the index set for the CW decomposition, then the limit $\varprojlim_{\sigma} 1$ computes the reduced $\bfk$-coefficient homology of $S^{n-1}$, which gives $1\simeq \varprojlim_{\sigma} 1$. 

2) For $n\neq 0$, there exists a unique non-zero cone, which we denote by $\sigma(n)$, such that $n\in \operatorname{RelInt}(\sigma(n))$. We set $\operatorname{Star}(n)=\{\tau \neq \{0\}:\sigma(n) \prec \tau  \}$ and  $\partial\operatorname{Star}(n)=\{\tau\in \operatorname{Star}(n), n\notin \tau\}$ that can also be treated as two sub CW complexes of $\Sigma\setminus \{ \{0\}\}$. Then we have $(1_{\tau})_n\simeq 1$ if and only if ${\tau \in \operatorname{Star}(n) \setminus \partial \operatorname{Star}(n)}$, and otherwise $(1_{\tau})_n\simeq 0$. Therefore, we have that $\varprojlim_{\tau} (1_{\tau})_n\simeq \varprojlim_{\tau \in \operatorname{Star}(n) \setminus \partial \operatorname{Star}(n)} 1$, which computes the relative $\bfk$-coefficient homology of the pair $(|\operatorname{Star}(n)|,|\partial\operatorname{Star}(n)|)$ (here $|\bullet|$ denotes the topological space of the corresponding CW complex). We conclude by showing that $(|\operatorname{Star}(n)|,|\partial\operatorname{Star}(n)|)\simeq (D^{n-1},S^{n-2})$. In fact, all cones are star-shaped, and if $\tau \prec \sigma(n)$, $\tau$ can be deformation retracted into $\sigma(n)$. Then $|\operatorname{Star}(n)|$ is star-shaped. In particular, we have $(|\operatorname{Star}(n)|,|\partial\operatorname{Star}(n)|)\simeq (D^{n-1},S^{n-2})$.     
\end{proof} 
Using the Fourier--Sato transformation of sheaves, where we refer to \cite[Section 3.7]{KS90} for more details, we have the following.
\begin{Coro}\label{corollary: kernel for convolution unit}If $\Sigma$ is a complete fan, then we have an equivalence in $\Sh(M_\bR)$:
\[1_{0} \simeq \varprojlim_{\sigma\in  \Sigma^\op} 1_{\sigma^\vee}[\dim N_\bR],\]
where morphisms are induced by open inclusions induced by the face relation.     
\end{Coro}
\begin{proof}We apply the (inverse) Fourier--Sato transformation to $1_{N_\bR} \simeq \varprojlim_{\sigma } 1_{\sigma}$. Precisely, consider the subset $Z=\{(m,n)\in M_\bR \times N_\bR:\langle n,m\rangle \geq 0\}$ and consider the exact functor
\[F \mapsto p_{M!}(p_{N}^* F \otimes 1_{Z})[\dim N_\bR],\]
where $p_M,p_N$ are the projections from $M_\bR\times N_\bR$ to the corresponding factors. Then a computation similar to \cite[Lemma 3.7.10]{KS90} shows that $1_{N_\bR}$ is mapped to $1_{0}$, and $1_{\sigma}$ is mapped to $1_{\sigma^\vee}[\dim N_\bR]$ where closed inclusions are mapped to open inclusions.
\end{proof}

Now, we consider the assignment $\Sigma \rightarrow \PrSt$,
\[\sigma\mapsto \Sh_{M_\bR \times {(-\sigma)}}(M_\bR) ,\quad \tau\subset \sigma\mapsto \Sh_{M_\bR \times {(-\tau)}}(M_\bR) \xrightarrow{i_{\tau\sigma}} \Sh_{M_\bR \times {(-\sigma)}}(M_\bR)\]
defines a functor, where $i_{\tau\sigma}$ is the canonical inclusion. In fact, the diagram is in ${\operatorname{Pr_{st}^{L,R}}}$, i.e., $i_{\tau\sigma}$ admits both left and right adjoints $i_{\tau\sigma}^l$ and $i_{\tau\sigma}^r$: Since all microsupport conditions are closed, the inclusion commutes with both limits and colimits, we obtain the adjoints by the adjoint functor theorem. 

By the definition of microsupport, we have a fully faithful embedding 
\[i_\sigma: \Sh_{M_\bR \times {(-\sigma)}}(M_\bR) \rightarrow  \Sh_{M_\bR \times {(-|\Sigma|)}}(M_\bR),\]
which is compatible with the direct diagram over $\Sigma$: $i_\sigma i_{\tau\sigma}=i_\tau$. Moreover, this fully faithful functor admits both left and right adjoints $i_\sigma^l$ and $i_\sigma^r$ by the adjoint functor theorem since $\sigma$ is closed. Then we have a natural functor in ${\operatorname{Pr_{st}^{L,R}}}$:
\begin{equation}\label{equation: i functor for fan}
   i_\Sigma:\varinjlim_{\sigma\in \Sigma} \Sh_{M_\bR \times {(-\sigma)}}(M_\bR) \rightarrow \Sh_{M_\bR \times {(-|\Sigma|)}}(M_\bR). 
\end{equation}

Denote by $i_\Sigma^l$ and $i_\Sigma^r$ the adjoints of $i_\Sigma$, by \cite{Adjoint_descent}, we have
\[i_\Sigma^r i_\Sigma=\varprojlim_\sigma i_\sigma^r i_\sigma,\quad i_\Sigma^l i_\Sigma=\varinjlim_\sigma i_\sigma^l i_\sigma,\quad i_\Sigma i_\Sigma^l=\varprojlim_\sigma i_\sigma i_\sigma^l, \quad i_\Sigma i_\Sigma^r=\varinjlim_\sigma i_\sigma i_\sigma^r.\] 

\begin{Thm}\label{Thm: sheaves with fan support}For a fan $\Sigma$, the functor $i_\Sigma$ induces equivalences
\[ \Sh_{M_\bR\times (-|\Sigma|)}(M_\bR)=\varinjlim_{\sigma\in \Sigma } \Sh_{M_\bR\times (-\sigma)}(M_\bR),\quad \Sh^{\bR^n_\delta}_{M_\bR\times (-|\Sigma|)}(M_\bR)=\varinjlim_{\sigma\in \Sigma } \Sh^{\bR^n_\delta}_{M_\bR\times (-\sigma)}(M_\bR)\]  
in ${\operatorname{Pr_{st}^{L,R}}}$. Moreover, both equivalences are monoidal.
\end{Thm}

\begin{proof}We only need to show that the non-equivariant functor $i_\Sigma$ is an equivalence. For the equivariant version, as the homotopy fixed point of $\bR^n_\delta$ can be computed as a colimit in $\PrSt$, the result follows from the non-equivariant version since colimits commute. 

The fully faithfulness of the functor $i_\Sigma$ is clear: In fact, we have
\[i_\Sigma^r i_\Sigma=\varprojlim_\sigma i_\sigma^r i_\sigma=\id,\quad i_\Sigma^l i_\Sigma=\varinjlim_\sigma i_\sigma^l i_\sigma=\id.\]
since all $i_\sigma$ are fully faithful. 

Next, we show the essential surjectivity. 

By \cite{CompletionFan}, there exists a complete fan $\Theta$ such that $\Sigma$ is a subfan of $\Theta$. We first prove the essential surjectivity for $\Theta$. A generic cone in $\Theta$ is denoted by $\theta$.

We notice that $i_\theta i_\theta^l  \simeq 1_{\theta^\vee}[n] \star (\bullet)$ by \cite[Equation (3.1.5)]{guillermou2019sheaves} and a higher-dimension version of \cite[Proposition 5.11]{Hochschild-Kuo-Shende-Zhang} that switches $\circ$ and $\star$-convolutions. Therefore, by \autoref{corollary: kernel for convolution unit} and the limit formula for $i_\Theta i_\Theta^l$, we have
\[i_\Theta i_\Theta^l=\varprojlim_\theta i_\theta i_\theta^l=\id_{\Sh(M_\bR)},\]
which implies that $i_\Theta$ is essentially surjective. 

Then we use the essential surjectivity of $i_\Theta$ to show the essential surjectivity of $i_\Sigma$. 

The essential surjectivity of $i_\Theta$ means that $i_\Theta i_\Theta^l=\id_{\Sh(M_\bR)}$, we pass to right adjoints of the equivalence: 
\[\id_{\Sh(M_\bR)}=\id_{\Sh(M_\bR)}^r\simeq i_\Theta i_\Theta^r=\varinjlim_\theta i_\theta i_\theta^r.\]

Notice that $i_\theta i_\theta^r=\varphi_{\theta^\vee}^*\varphi_{\theta^\vee*}$ in the notation of \autoref{lemma: microlocal cut-off}. As the colimit over $\Theta$ is finite, we can evaluate at an open set $U$ for $F\in \Sh(M_\bR)$ to obtain 
\[F(U)=\varinjlim_{\theta \in \Theta} \varphi_{\theta^\vee}^*\varphi_{\theta^\vee*}F(U)=\varinjlim_{\theta \in \Theta} F(U+\theta^\vee).\]

Now, we consider the following constructions: For each ray $\rho \in \Theta(1)$, we fix a cone generator $u_\rho$ such that $\rho=\bR_{\geq 0}u_\rho$. For $d\in (\bR\cup \{\infty\})^{\Theta(1)}$, we define an open polyhedron
\begin{equation}\label{equation: fan open sets}
   \Delta_{\Theta}(d)= \cap_{\rho\in \Theta(1)} H_>(d_\rho) \subset M_\bR, 
\end{equation}
where $ H_>(d_\rho) = \{m\in M_\bR: \langle m, u_\rho\rangle > -d_\rho\}$.

For any $\theta\in \Theta$, we set $d(\theta) \in (\bR\cup \{\infty\})^{\Theta(1)}$ such that $d(\theta)_\rho=d_\rho$ if $\rho$ is a face of $\theta$, and $d(\theta)_\rho=\infty$ otherwise.

Now take $U=\Delta_{\Theta}(d)$, we notice that
\[\Delta_\Theta(d(\theta))=\Delta_{\Theta}(d)+\mathring{\theta^\vee},\]
and we have for $F\in \Sh(M_\bR)$ that
\begin{equation}\label{eq: surjectivity for Theta}
  F(\Delta_{\Theta}(d))=\varinjlim_{\theta \in \Theta} F(\Delta_{\Theta}(d)+\theta^\vee)= \varinjlim_{\theta \in \Theta} F(\Delta_\Theta(d(\theta))).  
\end{equation}

Now, we assume in addition that $F\in \Sh_{M_\bR\times (-|\Sigma|)}(M_\bR)$, i.e. $\SS(F)\subset M_\bR\times (-|\Sigma|)$, we claim that $\Sigma$ is cofinal in the colimit diagram \eqref{eq: surjectivity for Theta}. 

We prove the claim by proving the following: Under the microsupport condition $\SS(F)\subset M_\bR\times (-|\Sigma|)$, we have the natural restriction map
\[F(\Delta_\Theta(d( \sigma))) \rightarrow F(\Delta_\Theta(d( \theta)))\]
is an equivalence for all $\sigma \prec \theta$ with $\sigma \in \Sigma$ and $\theta\in \Theta\setminus \Sigma$. In fact, we only need to prove it for the case when $\theta$ is a maximal cone (maximal with respect to the face relation). In this case, as $\Theta$ is complete, maximal cones are all top-dimensional and, in particular, have non-trivial interiors. 

As $\SS(F)\subset M_\bR\times (-|\Sigma|)$, we have $\SS(F) \cap M_\bR\times (-\mathring{\theta})=\varnothing$. Moreover, we observe that 
\[\Delta_\Theta(d( \sigma))+\theta^\vee = \Delta_\Theta(d( \sigma)), \quad\Delta_\Theta(d( \theta))+\theta^\vee = \Delta_\Theta(d( \theta)).\]
That is, both $\Delta_\Theta(d( \sigma))$ and $\Delta_\Theta(d( \theta))$ are $\theta^\vee$-open sets.

Then we use \cite[Proposition 5.2.1]{KS90} to conclude that for sufficiently large bounded open sets $X\subset M_\bR$ (to guarantee the required compactness therein), we have an equivalence
\[F(\Delta_\Theta(d( \sigma)) \cap X)  \xrightarrow{\simeq } F(\Delta_\Theta(d( \theta)) \cap X).\]
It is clear that the equivalence is functorial with respect to $X$. Then we can take an increasing open exhaustion $X_n$ of $M_\bR$, and we have
\[F(\Delta_\Theta(d( \sigma)))=\varprojlim_n F(\Delta_\Theta(d( \sigma)) \cap X_n) \xrightarrow{\simeq } \varprojlim_n F(\Delta_\Theta(d( \theta)) \cap X_n)=  F(\Delta_\Theta(d( \theta)) ).\]

Therefore, we have that
\begin{equation}\label{eq: surjectivity for Sigma}
    F(\Delta_{\Theta}(d))= \varinjlim_{\sigma \in \Sigma} F(\Delta_\Theta(d(\sigma))) =\varinjlim_{\sigma \in \Sigma} F(\Delta_{\Theta}(d)+\sigma^\vee),
\end{equation}
i.e. $\Sigma$ is cofinal in the colimit diagram \eqref{eq: surjectivity for Theta}.

Moreover, because $\left\{ {\Delta_{\Theta}(d)}: d\in \bR^{\Theta(1)}\right\}$ is a set of bounded polytopes that form a basis of $M_\bR$, we can replace $\Delta_{\Theta}(d)$ by an arbitrary open set $U\subset M_\bR$ in \eqref{eq: surjectivity for Sigma}:
\[F(U)=\varinjlim_{\sigma \in \Sigma} F(U+\sigma^\vee),\]
for $F\in \Sh_{M_\bR\times (-|\Sigma|)}(M_\bR)$.

This formula is equivalent to saying that 
\[\id_{\Sh_{M_\bR\times (-|\Sigma|)}(M_\bR)}\simeq \varinjlim_\sigma i_\sigma i_\sigma^r=i_\Sigma i_\Sigma^r ,\]
and $i_\Sigma$ is essentially surjective.

To see the monoidality, we pass to left adjoints $\id\simeq i_\Sigma i_\Sigma^l=\varprojlim_\sigma i_\sigma i_\sigma^l$ where $i_\sigma i_\sigma^l$ is presented by an idempotent $\star$-convolution functor, which is monoidal. Then $i_\Sigma$ is monoidal by adjoint equivalences.
\end{proof}

It seems that the theorem can be extended to certain non-polyhedral fan-type structures. Let us see one example:
\begin{Coro}\label{coro: cut-off mult cones}Suppose $\{\tau_1,\cdots,\tau_n\}$ is a finite set of proper closed convex cones that are not necessarily polyhedral. Suppose they pairwise intersect at the origin, and assume that there exists a fan $\Sigma$ such that for each $\tau_k$ there exists $\sigma_k \in \Sigma$ such that $\tau_k\subset \sigma_k$, and the $\sigma_k$ also intersect pairwise at the origin. Let $\tau_0=\sigma_0=\{0\}$ and let $K$ be the index category such that $j\rightarrow k$ if and only if $j=0$ and $k=1,\dots,n$. Then we have
\[\varinjlim_{K} \Sh_{M_\bR \times {\tau_k}}(M_\bR) \simeq \Sh_{M_\bR \times { \cup_{k}\tau_k}}(M_\bR).\]
\end{Coro}
\begin{proof}The argument for the fully faithfulness is the same as the fan case. For the essential surjectivity, we apply \autoref{Thm: sheaves with fan support} to the subfan $\Sigma'$ generated by $\sigma_k$ (then switch the sign). As the index category is sufficiently simple, we can apply a similar non-characteristic deformation argument to show that
\[\id_{\Sh_{M_\bR\times { \cup_{k}\tau_k}}(M_\bR)}\simeq \varinjlim_K i_{\tau_k} i_{\tau_k}^r.\qedhere\]    
\end{proof}

\begin{RMK}If $n=2$, the category $K$ is of the form $1 \leftarrow 0 \rightarrow 2$, so the colimit diagram is a pushout diagram. In this sense, we obtain a splitting result similar to \cite[Proposition 3.1.10]{guillermou2019sheaves}.

However, we do not know how to generalize the corollary if intersections of cones are more complicated because we do not know any combinatorial structure like a fan to organize the data.
\end{RMK}

\begin{RMK}It is explained in \cite[Corollary 3.4.3]{guillermou2019sheaves} that for a closed convex cone $\sigma$, and $F\in \Sh(\bR^n)$, we have $\SS(F)\subset \bR^n\times \sigma $ implies that $\SS(H^iF) \subset \bR^n\times \sigma$ for all $i\in \bZ$. (The converse is also true by \cite[Exercise V.6]{KS90}.)

However, \cite[Corollary 3.4.3]{guillermou2019sheaves} is not true for more general microsupport conditions, and \autoref{Thm: sheaves with fan support} gives a natural explanation of this issue (for polyhedral microsupport conditions): For $F$ with $\SS(F)\subset \bR^n\times |\Sigma|$ for a fan $\Sigma$, we can functorially decompose $F$ into $\{F_\sigma\}_{\sigma\in \Sigma}$ such that $F=\varinjlim_{\sigma} F_\sigma$ and $\SS(F_\sigma)\subset \bR^n\times \sigma $. Notice that, since $\Sigma$ is not filtered, $H^i$ does not commute with $\varinjlim_{\sigma}$, and then we cannot assert that $\SS(H^iF)\subset \bR^n\times |\Sigma|$.
\end{RMK}

\section{Global APT}
We assume that ${\bfk}$ is a discrete commutative ring in this section. The only reason for this restriction is that we do not know enough spectral algebraic geometry to cover some algebro-geometric constructions (especially the root stack). We expect these geometric constructions could be worked out in light of \cite{Geometry_of_filtration} and \cite[Section 3]{bai2025toricmirrorsymmetryhomotopy}; if so, all results here hold over commutative ring spectra.

Recall the notion of fan we defined in \autoref{def: fan}: We fix a finite-dimensional vector space $N_\bR$ and set $M_\bR=N_\bR^\vee$. A fan $\Sigma$ is a finite set of proper closed convex polyhedral cones in $N_\bR$ that is closed under taking faces and intersections.
 
In the toric literature, rational fans are also discussed. We say $N_\bR$ is $N$-rational if for a (fixed) lattice $N\subset N_\bR$, we have $N_\bR=N\otimes \bR$. In this case, we denote $M=\HOM_{\bZ}(N,\bZ)$ . 

\begin{Def}\label{def: polyhedral fan}For a fan $\Sigma$ in $N_\bR$, we say $\Sigma$ is a $N$-rational fan if $N_\bR$ is $N$-rational and for all $\sigma\in \Sigma$, $\sigma$ is generated by finitely many vectors in $N\hookrightarrow N_\bR$. 
\end{Def}

Recall that, if $\Sigma$ is rational, one can associate a toric variety $X_\Sigma$, we refer to \cite{Cox_ToricBook} for a general account of toric varieties. 

Our goals of this section are the following:
\begin{itemize}[fullwidth]
    \item We construct the Novikov toric scheme $X_\Sigma^\Nov$ for a fan $\Sigma$, which admits a Novikov torus action $\bT^\Nov$ on $X_\Sigma^\Nov$, and a closed subscheme $\partial_\Sigma$ of $X_\Sigma^\Nov$ defined by an idempotent ideal sheaf. In the case where $\Sigma$ is rational, we explain the relation between $X_\Sigma^\Nov$ and the usual toric variety $X_\Sigma$;
    \item We define almost quasi-coherent sheaves $\QCoha(X_\Sigma^\Nov)$ over a Novikov toric scheme and prove the global version of APT (which is called log-perfectoid mirror symmetry in \cite{Vaintrob-logCCC}): 
    \[\QCoha_{\bT^\Nov}(X_\Sigma^\Nov)\simeq \Sh_{M_\bR\times |\Sigma|}(M_\bR),\quad \QCoha(X_\Sigma^\Nov)\simeq \Sh^{\bR_\delta^n}_{M_\bR\times |\Sigma|}(M_\bR)\]
    that recovers \autoref{prop: APT HD} if $\Sigma=\Sigma_\sigma$ is the face fan of a proper closed convex cone $\sigma \subset N_\bR$.
\end{itemize}

\subsection{\texorpdfstring{Warm-up: the Novikov projective line $\bP^\Nov$.}{}}\label{subsection: Novikov P1}
We start with a warm-up example: the Novikov projective line $\bP^\Nov$. In this example, we will see how the whole story runs. To simplify the notation, we will omit superscripts indicating dimensions of relevant geometry. 

We take $M_\bR=N_\bR=\bR$ and the fan $\Sigma=\{\{0\}, \bR_{\leq }, \bR_{\geq }\}$. These data are rational. It is known that the toric variety associated with $\Sigma$ is the projective line. Here, we construct $X_\Sigma^\Nov=\bP^\Nov$.

For each cone $\sigma\in \Sigma$, we can define three rings associated with $\sigma^\vee$:
\[\Lambda_{\bR}={\bfk}[\{0\}^\vee]={\bfk}[\bR],\quad \Lambda_{\leq }={\bfk}[\bR_{\leq }],\quad \Lambda_{\geq }={\bfk}[\bR_{\geq }] .\]
Then we can define their prime spectra as the (1-dimensional) Novikov torus and the Novikov affine lines (we distinguish two copies of affine lines by a sign):
\[ {\bT^\Nov} \coloneqq \Spec(\Lambda_{\bR}), \quad{\bA^\Nov_{\leq }}\coloneqq \Spec (\Lambda_{\leq }), \quad{\bA^\Nov_{\geq }}\coloneqq\Spec(\Lambda_{\geq }).\]

We have the following observation that we made in the study of $1$-dimensional APT:
\[\Lambda_{\bR}=\Lambda_{\leq }[\frac{1}{t^{-1}}]=\Lambda_{\geq }[\frac{1}{t}].\]

So, we could regard the natural morphisms of schemes induced by inclusion of rings $\Lambda_{\leq }, \Lambda_{\geq }\subset \Lambda_{\bR }$ as open inclusions:
\[  {\bA^\Nov_{\leq }} \supset  \bT^\Nov  \subset  {\bA^\Nov_{\geq }}.  \]

Then we glue ${\bA^\Nov_{\leq }}$ and ${\bA^\Nov_{\geq }}$ 
along $\bT^\Nov$ to obtain
\[{\bP^\Nov}\coloneqq {\bA^\Nov_{\leq }} \cup_{ \bT^\Nov} {\bA^\Nov_{\geq }}. \]

For each $\sigma\in \Sigma$, we have the associated idempotent ideal that defines an almost content. In the case of $\bP^\Nov$, they are
\[I_\bR={\bfk}[\mathring{\bR}],\quad I_{\leq}={\bfk}[\mathring{\bR_{\leq }}],\quad I_{\geq }={\bfk}[\mathring{\bR_{\geq }}] .\]

They define closed subschemes of the corresponding affine schemes for each $\sigma$, and are compatible with the gluing data. So, they can be glued to a closed subscheme $\partial_\Sigma$ of ${\bP^\Nov}$ such that 
\[\partial_\Sigma\cap \bT^\Nov=V(I_\bR)=\varnothing, \quad \partial_\Sigma\cap {\bA^\Nov_{\leq }} = V(I_{  \leq }),\quad \partial_\Sigma\cap {\bA^\Nov_{\geq }} = V(I_{  \geq }),\]
where $V(I)\subset \Spec(R)$ means the closed set associated with an ideal $I\subset R$.

We define the category of almost quasi-coherent sheaves over affine covers as the corresponding almost module categories. Precisely, for $S=\bR, {\leq }, {\geq }$, we set
\[\QCoha(\Spec(\Lambda_{S})) \coloneqq \Moda(\Lambda_{S}, I_{S}). \]

Quasi-coherent sheaves that are $\bT^\Nov$-equivariant correspond to $\bR$-graded modules, i.e., we have
\[\QCoha_{\bT^\Nov}(\Spec(\Lambda_{S})) \coloneqq \Moda_{\bR-gr}(\Lambda_{S}, I_{{S}}) \]
for $S=\bR, {\leq }, {\geq }$.

Then we define the category of $\bT^\Nov$-equivariant almost quasi-coherent sheaves over ${\bP^\Nov}$  via the pullback diagram in $\PrStR \simeq (\PrSt)^{\op}$:
\[\begin{tikzcd}
\QCoha_{\bT^\Nov}({\bP^\Nov},\partial_\Sigma) \arrow[d] \arrow[r] & \Moda_{\bR-gr}(\Lambda_{{\leq }}, I_{\leq}) \arrow[d] \\
\Moda_{\bR-gr}(\Lambda_{{\geq }}, I_{\geq}) \arrow[r]           & \Moda_{\bR-gr}(\Lambda_{\bR}, I_{{\bR}}) .      
\end{tikzcd}\]

Similarly, we can define the non-equivariant version via non-graded almost modules.

So, one $\bT^\Nov$-equivariant almost quasi-coherent sheaf over $({\bP^\Nov},\partial_\Sigma)$ consists of three $\bR$-graded almost modules over three polynomial Novikov rings together with compatible gluing data.

However, by the $1$-dimensional APT, \autoref{prop: APT}, the diagram is equivalent to
\begin{equation}\label{equation: qcoh diagram of novikov P1}
\begin{tikzcd}
\QCoha_{\bT^\Nov}({\bP^\Nov},\partial_\Sigma) \arrow[d] \arrow[r] & \Sh_{\leq }(\bR) \arrow[d] \\
\Sh_{\geq }(\bR) \arrow[r]           & \Sh_{=0 }(\bR)=\Loc(\bR) .      
\end{tikzcd}    
\end{equation}

Therefore, by \autoref{Thm: sheaves with fan support} (passing to right adjoints), we have
\begin{Prop}\label{Prop: APT for P1}We have the monoidal equivalences
\[\QCoha_{\bT^\Nov}({\bP^\Nov},\partial_\Sigma) \simeq \Sh(\bR),\quad \QCoha({\bP^\Nov},\partial_\Sigma) \simeq \Sh^{\bR_\delta}(\bR).\] 
\end{Prop}

\subsection{Novikov toric scheme}\label{subsection: def Novikov toric scheme}
We present the construction of the Novikov toric scheme $X_\Sigma^\Nov$ associated to a fixed fan $\Sigma$.

For a given $\sigma\in \Sigma$, we first define the affine Novikov toric scheme $U_\sigma$: Recall that the semigroup ring $\Lambda_{\sigma^\vee}=\bfk[\sigma^\vee]
$, we define
\[U_\sigma^\Nov\coloneqq \Spec(\Lambda_{\sigma^\vee})=\Spec({\bfk}[\sigma^\vee]).\]
It is straightforward to check that $U_\sigma^\Nov$ is a non-Noetherian, normal, reduced, and irreducible affine scheme.

One example to which we should pay attention is the Novikov torus (of dimension $n$):
\[\bT^\Nov\coloneqq U_{\{0\}}^\Nov= \Spec[\Lambda_{M_\bR}].\]
It is an affine group scheme (still non-Noetherian), and it acts effectively on $U_\sigma$ for every $\sigma \in \Sigma$ since $\Lambda_{\sigma^\vee}=\bfk[\sigma^\vee]
$ is $M_\bR$-graded. 

Next, we consider their gluing. We recall the following separating lemma, whose proof can be found in \cite{Cox_ToricBook}.
\begin{Lemma}[{{\cite[Proposition 1.3.13]{Cox_ToricBook}}}]\label{lemma: separting lemma}For two cones  $\sigma_1,\sigma_2\in\Sigma$, and $\tau=\sigma_1\cap \sigma_2 \in\Sigma$, we have that, for $m$ in the relative interior of $ \sigma_1^\vee \cap (-\sigma_2^\vee) = (\sigma_1-\sigma_2)^\vee $, we have $\sigma_1 \cap H_m =  \tau =   \sigma_2 \cap H_{-m}$ where $H_m=H_{-m}=\{n\in N_\bR:\langle m,n\rangle =0\}$ is a hyperplane.
\end{Lemma}

With the notation above, we have $\tau^\vee=\sigma_1^\vee+\bR(-m)$ and $\tau^\vee=\sigma_1^\vee +\sigma_2^\vee$. Then we can construct the gluing data of the Novikov toric scheme.

By $\tau^\vee=\sigma_1^\vee+\bR(-m)$, we have for $t^m=\prod_{i=1}^n t_i^{m_i}\in \Lambda_{M_\bR}={\bfk}[M_\bR]$ that 
\[\Lambda_{\tau^\vee}={\bfk}[\sigma_1^\vee+\bR(-m)]={\bfk}[\sigma_1^\vee][\frac{1}{t^m}].\]
Therefore, we have
\begin{equation*}
    \begin{split}
        U_\tau^\Nov  
        \simeq  \Spec ({\bfk}[\sigma_1^\vee][\frac{1}{t^m}])
        \simeq U_{\sigma_1}^\Nov \cap \{t^m \neq 0\} \overset{open}{\subset} U_{\sigma_1}^\Nov.
    \end{split}
\end{equation*}
So, we define an isomorphism of schemes, by $\tau^\vee=\sigma_1^\vee +\sigma_2^\vee$,
\[g_{21}: U_{\sigma_1}^\Nov \cap \{t^m \neq 0\} \simeq U_\tau^\Nov \simeq U_{\sigma_2}^\Nov \cap \{t^{-m} \neq 0\}.\]

As for usual toric varieties, for $\sigma_{i} \in \Sigma$ where $i=1,2,3$, we have the cocycle condition, whenever the following maps are defined:
\[g_{21}=g_{12}^{-1},\quad g_{31}=g_{32}g_{21}.\]
\begin{Def}\label{definition: Novikov toric scheme}We define the Novikov toric scheme associated to a fan $\Sigma$:
\[X_\Sigma^\Nov\coloneqq  \varinjlim_\sigma U_\sigma^\Nov =\bigsqcup_{\sigma} U_\sigma^\Nov \left. \Big/ \right. (\sigma_1,x_1)\simeq (\sigma_2,g_{21}(x_1)) .\]    
\end{Def}
\begin{Thm}We have that $X_\Sigma^\Nov$ is a non-Noetherian irreducible, integral, separated, normal scheme admitting an effective $\bT^\Nov$-action.
\end{Thm}
\begin{proof}The proof is the same as that of the construction of the usual toric varieties, we refer to \cite[Theorem 3.1.5]{Cox_ToricBook}. We only need to notice that it is non-Noetherian since it has an affine cover $U_\sigma^\Nov$ such that all opens are non-Noetherian.
\end{proof}

Now, we explain the almost data associated with a fan. 

For $\sigma\in \Sigma$, we consider the ideal
\[ I_{\sigma^\vee}\coloneqq {\bfk}[\mathring{\sigma^\vee}] \subset \Lambda_{\sigma^\vee}={\bfk}[{\sigma^\vee}] 
 .\]

Then $I_{\sigma^\vee}$ is an idempotent ideal $I_{\sigma^\vee} \otimes_{\Lambda_{\sigma^\vee}} I_{\sigma^\vee}=I_{\sigma^\vee}$ since $\mathring{\sigma^\vee}+\mathring{\sigma^\vee}=\mathring{\sigma^\vee}$ (one can check this using Day convolution by \autoref{prop: Rees construction HD}).

Therefore, we can define 
\[\partial_\sigma \coloneqq V(I_{\sigma^\vee})\subset U_\sigma^\Nov,\]
the closed subscheme determined by the closed immersion $\Spec(\Lambda_{\sigma^\vee}/I_{\sigma^\vee}) \rightarrow \Spec(\Lambda_{\sigma^\vee})$. We will not distinguish them afterward.

For two cones $\sigma_1,\sigma_2\in \Sigma$ with $\tau=\sigma_1\cap \sigma_2$, as noted in \autoref{lemma: separting lemma}, we can pick $m$ such that $m\in \mathring{\sigma_1^\vee} \cap \mathring{(-\sigma_2^\vee)}$. This is because the hyperplane $H_m$ satisfies $H_{rm}=H_m$ for $r>0$. Therefore, we can glue $\partial_\sigma$ using the gluing data for $X_\Sigma^\Nov$ (cf. \autoref{definition: Novikov toric scheme}) to define
\begin{equation}\label{equation: boundary ideal}
\partial_\Sigma\coloneqq \varinjlim_\sigma   \partial_\sigma=\bigsqcup_{\sigma} \partial_\sigma /\simeq.    
\end{equation}

By definition, we have:
\begin{Prop}The scheme $\partial_\Sigma$ is a closed subscheme of $X_\Sigma^\Nov$ that is locally given by $\partial_\sigma$ in $U_\sigma^\Nov$. Moreover, the ideal sheaf associated with $\partial_\Sigma$ is idempotent.   
\end{Prop}

For a subgroup $A\subset M_\bR$, we can define, in a similar manner, an $A$-Novikov scheme $X^{A-\Nov}_\Sigma$ by gluing spectra $U^{A-\Nov}_\sigma\coloneqq\Spec(\Lambda_{\sigma^\vee}^A)$ of $\Lambda_{\sigma^\vee}^A=\bfk[\sigma^\vee\cap A]$ for all $\sigma\in \Sigma$
. When $A\subset M_\bR$ is dense subgroup, we define the closed subscheme $\partial_\Sigma^A$ via the idempotent ideal $I_{\sigma^\vee}^A={\bfk}[\mathring{\sigma^\vee}\cap A]$ and $\partial_\sigma^A=V(I_{\sigma^\vee}^A)$. Notice that even if $A$ is not dense, $\partial_\Sigma^A$ can still be defined, but its ideal sheaf will not be idempotent.

If $A\subset B \subset M_\bR$ are two subgroups of $M_\bR$, then the inclusion induces a morphism of schemes $X^{B-\Nov}_\Sigma\rightarrow X^{A-\Nov}_\Sigma$ that restricts on every affine chart associated with $\sigma$ to the morphism induced by the evident ${\bfk}$-algebra morphism
\[{\bfk}[\sigma^\vee\cap A] \rightarrow {\bfk}[\sigma^\vee\cap B].\]

Now, we assume that $\Sigma$ is $N$-rational for a lattice $N \subset N_\bR$. Then $M \subset M_\bQ=M\otimes \bQ\subset M_\bR=M\otimes \bR$ are two subgroups, and we have $X^{M-\Nov}_\Sigma=X_\Sigma$ is the usual toric variety and there is a scheme morphism
\[X^{M_\bQ-\Nov}_\Sigma\rightarrow X^{M-\Nov}_\Sigma=X_\Sigma.\]
We thank Alexander I. Efimov for explaining to us the following observations.
\begin{Prop}\label{prop: Q-Nov=root scheme}For a rational fan $\Sigma$, we have 
\[X^{M_\bQ-\Nov}_\Sigma \simeq \varprojlim_k X^{\frac{M}{k}-\Nov}_\Sigma \]
as schemes over $X_\Sigma$ via the inclusions ${M}\subset\frac{M}{k} \subset M_\bQ$ for all $k\in \bN$.
\end{Prop}
 
\begin{proof}Here, we set ${X_\Sigma^{[\infty]}}=\varprojlim_k X^{\frac{M}{k}-\Nov}_\Sigma$ and ${U_\sigma^{[\infty]}}=\varprojlim_k U^{\frac{M}{k}-\Nov}_\sigma$ for the affine version.

On each affine piece associated with $\sigma$, ${U_\sigma^{[\infty]}}$ is the pro-scheme 
\[\varprojlim_{k}\Spec({\bfk}[\sigma^\vee \cap (\frac{1}{k}M)])=\Spec(\varinjlim_{k}{\bfk}[\sigma^\vee \cap (\frac{1}{k}M)])= \Spec({\bfk}[\sigma^\vee \cap (\varinjlim_{k}\frac{1}{k}M)])= \Spec({\bfk}[\sigma^\vee \cap M_\bQ]).\]

So, we have ${U_\sigma^{[\infty]}}\simeq U^{M_\bQ-\Nov}_\sigma$. On the other hand, it is straightforward to see the equivalences are compatible with the gluing morphisms. So, we glue ${U_\sigma^{[\infty]}}\simeq U^{M_\bQ-\Nov}_\sigma$ together to obtain ${X_\Sigma^{[\infty]}}\simeq X_\Sigma^{M_\bQ-\Nov}$, and it is compatible with morphisms to $X_\Sigma$. 
\end{proof}

\begin{RMK}\label{remark: infinite root stack}The scheme  ${X_\Sigma^{M_\bQ-\Nov}}$ is deeply related to the infinite root stack defined in \cite{root_stack}. In fact, let 
\[\mu_\infty(M)=\underline{\HOM}_{\bZ}(M_\bQ /M,\mathbb{G}_{m,\bfk})=\varprojlim_{k}\underline{\HOM}_{\bZ}((\frac{1}{k} M)/M,\mathbb{G}_{m,\bfk})\]
be the Cartier dual of the ind-finite group scheme $M_\bQ /M \simeq (\bQ/\bZ)^{\operatorname{rk} M}$, which is a pro-finite group scheme that acts on $X_\Sigma^{M_\bQ-\Nov}$.

Then the infinite root stack is the fpqc quotient stack
\[ \sqrt[\infty]{X_\Sigma}\simeq [{X_\Sigma^{M_\bQ-\Nov}}//\mu_\infty(M)].\]

The equivalence was proven explicitly for the affine toric case in \cite[Proposition 3.10]{root_stack}, and can be glued together via 
\cite[Definition 3.14]{root_stack}.

\end{RMK}

\subsection{Almost quasi-coherent sheaves and global APT}\label{subsection: def almost qcoh}
In this subsection, we discuss the category of almost quasi-coherent sheaves.

\begin{Def}\label{definition: almost quasi-coherent sheaves for a fan}For a fan $\Sigma$, we define 
\begin{equation*}
    \begin{split}
       \QCoha_{\bT^\Nov}(X_\Sigma^\Nov,\partial_\Sigma) &\coloneqq \varprojlim_\sigma\Moda_{\bR^n-gr} (\Lambda_{\sigma^\vee},I_{{\sigma^\vee}}),\\ \QCoha(X_\Sigma^\Nov,\partial_\Sigma) &\coloneqq \varprojlim_\sigma\Moda (\Lambda_{\sigma^\vee},I_{{\sigma^\vee}} )  
    \end{split}
\end{equation*}
as limits in $\PrStR$.
\end{Def}

Passing to left adjoints, we have a colimit diagram in $\CAlg(\PrSt)$. Then the category of almost quasi-coherent sheaves admits a symmetric monoidal structure.

On the other hand, for $(X_\Sigma^\Nov,\partial_\Sigma)$, we can ``define" $\QCoh$ directly via the limits in $\PrStR$:
\[\QCoh_{\bT^\Nov}(X_\Sigma^\Nov) \coloneqq \varprojlim_\sigma\Mod_{\bR^n-gr} (\Lambda_{\sigma^\vee}),\quad \QCoh(X_\Sigma^\Nov) \coloneqq \varprojlim_\sigma\Mod (\Lambda_{\sigma^\vee} ), \]
and then define almost zero sheaves $\QCoh^{\partial_\Sigma}_{\bT^\Nov}(X_\Sigma^\Nov,\partial_\Sigma)$ and ${\QCoh^{\partial_\Sigma}}(X_\Sigma^\Nov)$ locally, namely, $F \in \QCoh^{\partial_\Sigma}(X_\Sigma^\Nov)$ (i.e. $F$ is almost zero) if it is almost zero relative to $(\Lambda_{\sigma^\vee},I_{{\sigma^\vee}} )$ on each affine chart $U_\sigma^\Nov$. 

\begin{RMK}Since $\{U_\sigma^\Nov\}_\sigma$ is a Zariski covering of $X_\Sigma^\Nov$, the above ``definition" is compatible with the usual definition of $\QCoh$. So, we make the above ad hoc definition.    
\end{RMK}

The following proposition can be understood as an equivalent definition of $\QCoha$, or a result about commutativity of the almostification and Zariski descent.
\begin{Prop}For a fan $\Sigma$, we have monoidal equivalences
\begin{equation*}
    \begin{split}
       \QCoha_{\bT^\Nov}(X_\Sigma^\Nov,\partial_\Sigma) &\simeq \QCoh_{\bT^\Nov}(X_\Sigma^\Nov)/{\QCoh_{\bT^\Nov}^{\partial_\Sigma}}(X_\Sigma^\Nov)\\    
        \QCoha(X_\Sigma^\Nov,\partial_\Sigma) &\simeq \QCoh(X_\Sigma^\Nov)/{\QCoh^{\partial_\Sigma}}(X_\Sigma^\Nov).
    \end{split}
\end{equation*}
\end{Prop}
\begin{proof}The limits $\varprojlim_\sigma$ here are all defined in $\PrStR$. Passing to left adjoints via anti-equivalence between $\PrStR$ and $\PrSt$, we have that limits $\varprojlim_\sigma$ are colimits $\varinjlim_\sigma$ in $\PrSt$. Then the proposition follows, because the Verdier quotient is defined as a cofiber in $\PrSt$ and colimits commute. The equivalences are monoidal since all colimits are actually taken in $\CAlg(\PrSt)$.
\end{proof}

Now the global APT correspondence follows formally by definition.
\begin{Thm}\label{theorem: Novikov mirror symmetry}For a fan $\Sigma$, there exist monoidal equivalences
\[\QCoha_{\bT^\Nov}(X_\Sigma^\Nov,\partial_\Sigma) \simeq \Sh_{M_\bR\times |\Sigma|}(M_\bR),\quad \QCoha(X_\Sigma^\Nov,\partial_\Sigma) \simeq \Sh^{\bR^n_\delta}_{M_\bR\times |\Sigma|}(M_\bR).\]    
\end{Thm}
\begin{proof}We apply \autoref{Thm: sheaves with fan support}. Then it reduces to the affine results \autoref{prop: APT HD} and \autoref{cor: Novikov-mirro-affine}.    
\end{proof}

\subsection{Variant of the results} 
\subsubsection{\texorpdfstring{From $\bR$ to $\bQ$}{}}
Recall that we defined $X_\Sigma^{A-\Nov}$ and $\partial_\Sigma^A$ for a dense subgroup $A\subset M_\bR$ in \autoref{subsection: def Novikov toric scheme}, and we explained in \autoref{prop: Q-Nov=root scheme} that, when $\Sigma$ is rational, we have $X_\Sigma^{M_\bQ-\Nov}\simeq {X_\Sigma^{[\infty]}}$. 

We can also define $\QCoha_{\bT^\Nov}(X_\Sigma^{A-\Nov},\partial_\Sigma^A)$ and $\QCoha(X_\Sigma^{A-\Nov},\partial_\Sigma^A)$ as in \autoref{definition: almost quasi-coherent sheaves for a fan}. 

Here, we have the following results.
\begin{Thm}\label{theorem: Q-Novikov mirror symmetry}For a fan $\Sigma$ and a dense subgroup $A\subset M_\bR$, there exist monoidal equivalences
\[\QCoha_{\bT^{A-\Nov}}(X_\Sigma^{A-\Nov},\partial_\Sigma^A) \simeq \Sh_{M_\bR\times |\Sigma|}(M_\bR),\quad \QCoha(X_\Sigma^{A-\Nov},\partial_\Sigma^A) \simeq \Sh^{A\cap \bR^n_\delta}_{M_\bR\times |\Sigma|}(M_\bR).\]  \end{Thm}
\begin{proof}Reviewing the proof of \autoref{theorem: Novikov mirror symmetry}, we only need to prove the following affine result (compare to \autoref{prop: APT HD} and \autoref{cor: Novikov-mirro-affine}): For a proper convex cone $\sigma$, and $I_{\sigma^\vee}^A=\bfk[\mathring{\sigma^\vee}\cap A]$. Then we have
\[\Sh_{M_\bR\times \sigma} (M_\bR) \simeq \Moda_{{A-gr}} (\Lambda_{\sigma^\vee}^A,I_{\sigma^\vee}^A), \quad \Sh^{A\cap \bR^n_\delta}_{M_\bR\times \sigma} (M_\bR) \simeq \Moda (\Lambda_{\sigma^\vee}^A,I_{\sigma^\vee}^A).\]  

Recall the proof of \autoref{prop: APT HD}. This is true, because the family of $\sigma^\vee$-open sets $\{\mathring{\sigma^\vee}-a\}_{a\in A}$ also forms a basis of the $\sigma^\vee$-topology of $M_\bR$ by the density of $A$.
\end{proof}

In particular, we take $A=M_\bQ=M\otimes\bQ$.
\begin{Coro}\label{theorem: Q-Novikov mirror symmetry-rational}For a rational fan $\Sigma$, there exist monoidal equivalences
\[\QCoha_{\bT^{M_\bQ-\Nov}}(X_\Sigma^{M_\bR-\Nov},\partial_\Sigma^{M_\bQ}) \simeq \Sh_{M_\bR\times |\Sigma|}(M_\bR),\quad \QCoha(X_\Sigma^{M_\bR-\Nov},\partial_\Sigma^{M_\bQ}) \simeq \Sh^{ \bQ^n_\delta}_{M_\bR\times |\Sigma|}(M_\bR).\]   \end{Coro}

\subsubsection{\texorpdfstring{Lattice quotient and root stack}{}}
The results in this section mostly make sense for rational fans $\Sigma$, on which we focus. In this case, we have $M_\bR=M\otimes \bR$ for a lattice $M\subset M_\bR$. 

Recall that $\sqrt[\infty]{X_{\Sigma}}$ is the infinite root stack of the toric variety $X_\Sigma$, which is equivalent to 
\[ \sqrt[\infty]{X_\Sigma}\simeq [{X_\Sigma^{M_\bQ-\Nov}}//\mu_\infty(M)]\]
by \autoref{remark: infinite root stack}. In this case, we also define $\sqrt[\infty]{\partial_\Sigma}$ as the image of $\partial^{M_\bQ}_\Sigma$ under the quotient.

We have the following lemma.
\begin{Lemma}We have a monoidal equivalence
 \[\QCoh_{\bT^{M_\bQ-\Nov}}(X_{\Sigma}^{M_\bQ-\Nov})^M\simeq\QCoh (\sqrt[\infty]{X_{\Sigma}} )
.\]   
\end{Lemma}
\begin{proof}
We first prove the affine case. In this case, we have that 
\[\QCoh_{\bT^{M_\bQ-\Nov}}(U_{\sigma}^{M_\bQ-\Nov})^M =\Mod_{M_\bQ-gr}(\Lambda_{\sigma^\vee})^M \simeq \Mod_{M_\bQ/M-gr}(\Lambda_{\sigma^\vee})  ,\]
where the first equality is the definition, and the second equality follows from a proof similar to \autoref{coro: non-graded equivalence}.

On the other hand, being equipped with a $M_\bQ/M$-grading is equivalent to requiring $\Spec({\bfk}[M_\bQ/M])$-equivariance.

Now, notice that 
\[\Spec({\bfk}[M_\bQ/M])=\varprojlim_k \Spec({\bfk}[\frac{M}{k}/M])=\varprojlim_k \mu_k(M)=\mu_\infty(M).\]

Then we have 
\[\QCoh_{\bT^{M_\bQ-\Nov}}(U_{\sigma}^{M_\bQ-\Nov})^M \simeq \Mod_{M_\bQ/M-gr}(\Lambda_{\sigma^\vee}) \simeq \QCoh_{\mu_\infty(M)}(U_{\sigma}^{M_\bQ-\Nov} ) = \QCoh(\sqrt[\infty]{U_{\sigma}}  ). \]

For the global case, both sides are glued over $\varprojlim_\sigma$ in $\PrStR$, and the fixed point $(\bullet)^M$ can also be computed as a limit in $\PrStR$. Then we conclude by noticing that the affine equivalences are compatible with the gluing data.
\end{proof}

Consequently, we have: 
\begin{Coro}\label{theorem: root stack-Novikov mirror symmetry-rational}We have a monoidal equivalence
\[\QCoha(\sqrt[\infty]{X_\Sigma},\sqrt[\infty]{\partial_\Sigma}) \simeq \Sh_{M_\bR/M\times |\Sigma|}(M_\bR/M).\]
\end{Coro}
\begin{proof}By the above lemma and \autoref{theorem: Q-Novikov mirror symmetry}, it remains to prove that $\Sh_{M_\bR/M\times |\Sigma|}(M_\bR/M)\simeq \Sh_{M_\bR\times |\Sigma|}(M_\bR)^M$.

Notice that $M$ is a lattice, and $M$ acts on $M_\bR$ freely. Then we have that the projection $p:M_\bR \rightarrow M_\bR/M$ induces the equivalence:
\[p^*:\Sh(M_\bR/M) \rightarrow \Sh(M_\bR) \rightarrow \Sh(M_\bR)_M\simeq \Sh(M_\bR)^M.\]    
To incorporate the microsupport condition, we use the standard microsupport estimates of $p^*$ \cite[Proposition 5.4.5]{KS90}. \end{proof}

In general, a subgroup $H$ of $\bR^n$ is a direct product $D\times L$ such that $D$ is dense in a proper linear subspace $E\subset \bR^n$, and $L$ is a lattice. Note that $H$ does not necessarily span the whole $\bR^n$. We refer to \cite[Theorem 21]{Siegel_1989_geometry_of_numbers} for related discussions. By our discussion so far, it is not hard to prove variants of these theorems for general $H$. Here, we emphasize one case that we need later; we leave the details for general case to the reader.

\begin{Prop}\label{prop: split group}For a complete fan $\Sigma$, we have an equivalence 
\[\Sh_{>}^{\bR_\delta}(M_{\bR}/M\times\bR_t)
    \cong \QCoha_{comp_{\bA^1}}(\sqrt[\infty]{X_{\Sigma}}\times (\bA^1)^\Nov, \sqrt[\infty]{\partial_\Sigma}\times 0).\]
\end{Prop}
\begin{proof}We apply \autoref{theorem: Novikov mirror symmetry} to the product fan $\Sigma \times \{[0,\infty)\}$ to obtain the equivalence 
\[\Sh_{>}(M_{\bR}\times\bR_t)
    \cong \QCoha_{\bT^{M_\bR\times \bR-\Nov},comp_{\bA^1}}(X_{\Sigma}^\Nov\times (\bA^1)^\Nov,  {\partial_\Sigma}\times 0).\]

Then we apply the homotopy orbit functor $\varinjlim_{B(M\times \bR^\delta)}$ to it. Notice that we may take the homotopy orbit separately for either $M$ or $\bR^\delta$, then both \autoref{theorem: root stack-Novikov mirror symmetry-rational} and \autoref{prop: equivariant APT} imply that
\[\Sh_{>}^{M\times \bR_\delta}(M_{\bR}\times\bR_t)
    \cong \QCoha_{comp_{\bA^1}}(\sqrt[\infty]{X_{\Sigma}}\times (\bA^1)^\Nov, \sqrt[\infty]{\partial_\Sigma}\times 0).\]

And we conclude by noticing that
\[\Sh_{>}^{M\times \bR_\delta}(M_{\bR}\times\bR_t)\simeq \Sh_{>}^{\bR_\delta}(M_{\bR}/M\times\bR_t)\]
since $M\subset M_\bR$ is a lattice.\end{proof}

\subsection{A Fukaya-Sheaf comparison}
As an application of \autoref{theorem: root stack-Novikov mirror symmetry-rational}, we give the following Fukaya-Sheaf comparison  result  based on \cite{GPS3}.
\begin{Prop}\label{Stupid sheaf-fukaya correspondence}Let $X_\Sigma$ be a smooth toric variety  and let $\sqrt[k]{\Lambda_\Sigma}$ be the FLTZ skeleton of $k^{th}$ root $\sqrt[k]{X_\Sigma}$ (i.e., $X_\Sigma^{M/k-\Nov}$) of $X_\Sigma$, we have a fully faithful embedding
\[\Sh(M_\bR/M)\hookrightarrow \varinjlim_{k}\mathscr{W}\Fuk(T^*(M_\bR/M);\sqrt[k]{\Lambda_\Sigma} ) ,  \]
which is right adjoint to a Bousfield localization.

Here, $\mathscr{W}\Fuk$ means the category of all modules over the usual wrapped Fukaya category.
\end{Prop}

\begin{proof}
By the usual toric mirror symmetry (for example \cite{KuwagakiCCC,ToricMSrevisited-Shende}), we have
\[\QCoh (\sqrt[k]{X_{\Sigma}} ) \simeq \Sh_{\sqrt[k]{\Lambda_\Sigma}}(M_\bR/M).\] 

On the other hand, by \autoref{lemma: root stack qcoh limit} below, we have
\[\QCoh (\sqrt[\infty]{X_{\Sigma}} )\simeq \varinjlim_{k}\QCoh (\sqrt[k]{X_{\Sigma}} ) \simeq \varinjlim_{k}\Sh_{\sqrt[k]{\Lambda_\Sigma}}(M_\bR/M).\] 

Therefore, using \cite{GPS3}, we have
\[\QCoh (\sqrt[\infty]{X_{\Sigma}} )\simeq\varinjlim_{k}\mathscr{W}\Fuk(T^*(M_\bR/M),\sqrt[k]{\Lambda_\Sigma}).\]

Lastly, we have
\[\Sh(M_\bR/M)\simeq \QCoha (\sqrt[\infty]{X_{\Sigma}} )\hookrightarrow \QCoh (\sqrt[\infty]{X_{\Sigma}} )\simeq\varinjlim_{k}\mathscr{W}\Fuk(T^*(M_\bR/M),\sqrt[k]{\Lambda_\Sigma}).\]

The statement about localization follows from that the almostification fucntor $\QCoh$ to $\QCoha$ is a Bousfield localization.
\end{proof}
\begin{RMK}On the sheaf side, one can think of $\varinjlim_{k}\Sh_{\sqrt[k]{\Lambda_\Sigma}}(M_\bR/M)$ as the category of constructible sheaves with microsupport in $\cup_k\sqrt[k]{\Lambda_\Sigma} $ (but this is not equivalent to $\Sh(M_\bR/M)$ as we have seen in the proof.) Therefore, one can think of the colimit of wrapped Fukaya categories as a version of ``wrapped Fukaya category with the stop $\cup_k\sqrt[k]{\Lambda_\Sigma} $".

On the coherent side, it is clear that the essential image is characterized by almost local quasi-coherent sheaves. However, it is not clear to us how to characterize those ``almost local" Lagrangians geometrically.     
\end{RMK}

\begin{Lemma}\label{lemma: root stack qcoh limit}For a toric variety $X_\Sigma$, we have the colimit in $\PrSt$:
\[\QCoh (\sqrt[\infty]{X_{\Sigma}} ) \simeq \varinjlim_k\QCoh (\sqrt[k]{X_{\Sigma}} ).\] 
\end{Lemma}
\begin{proof}
By \cite[Proposition 2.7]{Scherotzke_Sibilla_Talpo_2020}, we have 
\[\Perf (\sqrt[\infty]{X_{\Sigma}} ) \simeq\varinjlim_k\Perf (\sqrt[k]{X_{\Sigma}} )\] 
is a colimit in $\CatPerf$ of small idempotent complete stable categories and exact functors. However, we know that the ind completion functor induces an equivalence $\Ind:\CatPerf\simeq \PrStCG$, and the inclusion $\PrStCG \rightarrow \PrSt$ preserves colimits. So, we have
\[\Ind\Perf (\sqrt[\infty]{X_{\Sigma}} ) \simeq\varinjlim_k \Ind\Perf (\sqrt[k]{X_{\Sigma}} ),\]
and we conclude by using the equivalence $\Ind\Perf\simeq \QCoh $ since all stacks here are quasi-compact and quasi-separated.
\end{proof}

\section{Homological mirror symmetry over the Novikov ring for toric varieties}\label{section: Novikov HMS}
Usually, homological mirror symmetry (HMS, for short) is discussed over the base field or over the Novikov field $\widehat{\Lambda}$, the fraction field of $\widehat{\Lambda}_\geq$. Since Fukaya categories are in general defined over the Novikov ring, it is natural to seek a version over the Novikov ring (see some discussion in \cite{fukaya_distance}). Here, we discuss a possible formulation of it for toric varieties. To avoid technical difficulties, we assume that $\Sigma$ is a smooth complete fan. In the formulation of HMS due to Fang--Liu--Treumann--Zaslow~\cite{FLTZ11}, we have
\begin{equation}\label{equation: CCC over Lambda}
    \QCoh(X_\Sigma, \widehat{\Lambda}_{\geq })\cong \Sh_{\Lambda_\Sigma}(M_\bR/M,\widehat{\Lambda}_{\geq })
\end{equation}
by \cite{KuwagakiCCC} when $\widehat{\Lambda}_{\geq }$ is defined over a regular local ring $\bfk$. We proceed in the following way: We start from the toric HMS over $\bfk$ (although \cite{KuwagakiCCC} was written over $\bC$, its proof works over a regular local ring. See also~\cite{ToricMSrevisited-Shende}.) Then we tensor both sides of the toric HMS with $\Mod(\widehat{\Lambda}_\geq)$ over $\bfk$. In this case, we notice that $X_\Sigma/\widehat{\Lambda}_{\geq }$ is the base change of $X_\Sigma/\bfk$ by treating $\widehat{\Lambda}_{\geq }$ as a $\bfk$-algebra, then we conclude by \cite[Corollary 9.4.3.8-(a)]{HA}.
\begin{RMK}Here, the main reason why we assume $\bfk$ to be a regular local ring is that we are not sure whether the Zariski descent of \cite{Gaitsgory-Rozenblyum} is true over $\widehat{\Lambda}_{\geq }$ since it is not Noetherian. But we believe this is true since $\widehat{\Lambda}_{\geq }$ is a regular coherent ring in the sense of Bertin. If \cite{Gaitsgory-Rozenblyum} works for general valuation rings, we believe that the argument in \cite{KuwagakiCCC} works directly for the proof of \eqref{equation: CCC over Lambda}.
\end{RMK}

The right-hand side of \eqref{equation: CCC over Lambda} is equivalent to the Fukaya--Seidel category of the Givental--Hori--Vafa mirror $W_\Sigma$ of $X_\Sigma$~\cite{GPS3} defined over the Novikov ring. \emph{But this is not a correct formulation of the Fukaya category over the Novikov ring.}

To have a genuine sheaf analogue of the Fukaya category over the Novikov ring, we again use Tamarkin’s idea. What we should have is the following: The subcategory of $\Sh_{>}^{\bR_\delta}(M_{\bR}/M\times\bR_t)$ spanned by sheaf quantizations of the end-conic Lagrangians ending in $\Lambda_\Sigma$. We denote it by $\operatorname{SQ}(W_\Sigma, \widehat{\Lambda}_{\geq })$.
\begin{Conjecture}\label{Fukaya-SheafcomparsionLG}
$\operatorname{SQ}(W_\Sigma, \widehat{\Lambda}_{\geq })$ is equivalent to the derived category of the infinitesimally wrapped Fukaya category $\Fuk(W_\Sigma,\widehat{\Lambda}_{\geq })$ of $T^*(M_\bR/M)$ over the Novikov ring generated by Lagrangian submanifolds that are compactly Hamiltonian isotopic to Lefschetz thimbles of $W_\Sigma$.
\end{Conjecture}
We would like to explain the evidence for the conjecture: The Fukaya--Sheaf comparison over the Novikov ring conjectured in  \cite{IK-NovikovTamarkinCategory} states that the subcategory of sheaf quantizations in $\Sh_{>}^{\bR_\delta}(M_{\bR}/M\times\bR_t)$ is equivalent to the infinitesimally wrapped Fukaya category of $T^*M_\bR/M$. The conjecture is largely verified in \cite{Fukayacategories_prequantizationbundles_KPS}, and the full proof will be given in \cite{KPS2}. See also \cite{GVZ} for a relevant statement. Assuming these results, Conjecture~\ref{Fukaya-SheafcomparsionLG} is simply about the generation: Lagrangians ending at $\Lambda_\Sigma$ are almost/densely generated by the Lefschetz thimbles. This result is known in the situation of partially wrapped Fukaya categories~\cite{GPS2}. We expect a similar proof to work in this situation. In summary, Conjecture~\ref{Fukaya-SheafcomparsionLG} seems plausible.

Here we use the sheaf model $\operatorname{SQ}(W_\Sigma, \widehat{\Lambda}_{\geq })$ to discuss HMS. Now, we use \autoref{prop: split group} to obtain:
\begin{equation}
\operatorname{SQ}(W_\Sigma,\widehat{\Lambda}_{\geq })\hookrightarrow \Sh_{>}^{\bR_\delta}(M_{\bR}/M\times\bR_t)\cong  \QCoha_{comp_{\bA^1}}(\sqrt[\infty]{X_{\Sigma}}\times (\bA^1)^\Nov, \sqrt[\infty]{\partial_\Sigma}\times 0).
\end{equation}
This is already a version of HMS, but the right-hand side is not very sensitive to the choice of $\Sigma$. We refine this as follows.

Note that the base change
\begin{equation}
    \pi^A\colon \operatorname{SQ}(W_\Sigma, \widehat{\Lambda}_{\geq })\hookrightarrow \Sh_{>}^{\bR_\delta}(M_{\bR}/M\times\bR_t)\xrightarrow{\otimes_{\widehat{\Lambda}_{\geq}}\widehat{\Lambda}} \Sh(M_\bR/M,\widehat{\Lambda})
\end{equation}
lands in $\Sh_{\Lambda_\Sigma}(M_\bR/M,\widehat{\Lambda})$. The mirror is given by the composition of 
\begin{equation}
\begin{split}
    \pi^B\colon \QCoha_{comp_{\bA^1}}(\sqrt[\infty]{X_{\Sigma}}\times (\bA^1)^\Nov, \sqrt[\infty]{\partial_\Sigma}\times 0)&=\QCoha(\sqrt[\infty]{X_{\Sigma}}, \sqrt[\infty]{\partial_\Sigma})\otimes \Moda_{t-comp}(\widehat\Lambda_{\geq})\\
    &\rightarrow \QCoha(\sqrt[\infty]{X_{\Sigma}}, \sqrt[\infty]{\partial_\Sigma},\widehat\Lambda)
\end{split}
\end{equation}
where the base ring is changed to $\widehat{\Lambda}$.

Note also that we have an embedding $\QCoh(X_\Sigma,\widehat\Lambda)\hookrightarrow  \QCoha(\sqrt[\infty]{X_{\Sigma}}, \sqrt[\infty]{\partial_\Sigma},\widehat\Lambda)$, which is mirror to $\Sh_{\Lambda_\Sigma}(M_\bR/M)\hookrightarrow \Sh(M_\bR/M)$. So we set
\begin{equation}
    (\pi^B)^{-1}(\QCoh(X_\Sigma,\widehat\Lambda))\subset \QCoha_{comp_{\bA^1}}(\sqrt[\infty]{X_{\Sigma}}\times (\bA^1)^\Nov, \sqrt[\infty]{\partial_\Sigma}\times 0)
\end{equation}
Summarizing the above discussions, we obtain the following version of HMS:
\begin{Coro}\label{cor:HMS_sheaf}
There exists a $\widehat{\Lambda}_{\geq }$-linear embedding:
        \begin{equation}
        \operatorname{SQ}(W_\Sigma,\widehat{\Lambda}_{\geq })\hookrightarrow     (\pi^B)^{-1}(\QCoh(X_\Sigma,\widehat{\Lambda})).
    \end{equation}
\end{Coro}

Now we can formulate a conjecture for the log Calabi--Yau setup. Let $(X,D)$ be a log CY-manifold, namely, $D$ is an anticanonical divisor of $X$. Let $W_X\colon \check X \to \bC$ be a mirror of $(X,D)$. Let $\Fuk(W_X,\widehat{\Lambda}_{\geq })$ be the derived Fukaya category over $\widehat{\Lambda}_{\geq}$ generated by Lagrangians compactly Hamiltonian isotopic to Lefschetz thimbles of $W_X$. 

On the other hand, we consider the infinite root stack $\sqrt[\infty]{X}$ associated to $(X,D)$. We can similarly consider
\begin{equation}
\begin{split}
    \pi^B\colon \QCoha_{comp_{\bA^1}}(\sqrt[\infty]{X}\times (\bA^1)^\Nov, \sqrt[\infty]{D}\times 0)&=\QCoha(\sqrt[\infty]{X}, \sqrt[\infty]{D})\otimes \Moda_{t-comp}(\widehat\Lambda_{\geq})\\
    &\rightarrow \QCoha(\sqrt[\infty]{X}, \sqrt[\infty]{D},\widehat\Lambda).
\end{split}
\end{equation}
Now we can state the following conjecture:
\begin{Conjecture}[HMS of log CY over the Novikov ring]\label{conj:HMS_LogCY}
There exists a $\widehat{\Lambda}_{\geq }$-linear embedding:
        \begin{equation}
        \Fuk(W_X,\widehat{\Lambda}_{\geq })\hookrightarrow     (\pi^B)^{-1}(\QCoh(X,\widehat{\Lambda})).
    \end{equation}
\end{Conjecture}
Note that Corollary~\ref{cor:HMS_sheaf} is a special case of Conjecture~\ref{conj:HMS_LogCY} under Conjecture~\ref{Fukaya-SheafcomparsionLG}.

\begin{RMK}
    Usually, mirror symmetry predicts a duality between two degenerating families around the large complex-structure limit and the large-volume limit. In terms of homological mirror symmetry, one can expect an equivalence between Fukaya category over the Novikov ring and the derived category of coherent sheaves on a degenerating family over $\Spec \widehat{\Lambda}_{\geq}$ of varieties. For example, one can see that Fukaya's discussion~\cite{fukaya_distance} is made along this type of HMS.

Our formulation here is slightly different from this expectation. On the symplectic side, we have a model of the Fukaya category over the Novikov ring, but we have a smooth (non-degenerating) family $\sqrt[\infty]{X}\times (\bA^1)^{\Nov}$. One possibly correct explanation is that our HMS is ``almost HMS": Working with almost modules, we are neglecting the central fiber, and this makes our family almost equivalent to a smooth family. Clarifying this point should also be related to the characterization of the essential image of our mirror functor. 
\end{RMK}

\newpage
\appendix

\section{APT correspondence: diagrammatic view}\label{Appendix: diagram of modules}
\addtocontents{toc}{\protect\setcounter{tocdepth}{1}}

We can package the APT correspondence for $\sigma=[0,\infty)$ into the following diagrams.

\begin{Thm}
    The diagrams (\ref{diagram:module}) and (\ref{diagram:sheaf}) (resp. (\ref{diagram:Rmodule}) and (\ref{diagram:Rsheaf})) are equivalent.
\end{Thm}
As one can see, the diagram (\ref{diagram:module}) and (\ref{diagram:sheaf}) (resp. (\ref{diagram:Rmodule}) and (\ref{diagram:Rsheaf})) do not have exactly the same shape. The statement of the theorem says that they are equivalent to each other on the part where they are comparable. The missing arrows are absent because they do not have standard interpretations in their worlds.

Our convention is that the upper arrow is the left adjoint and the lower arrow denotes the right adjoint of the upper arrow.

\subsection*{Module world diagram}
\begin{equation}\label{diagram:module}
    \xymatrix{
    & \Moda_{\bR-gr} (\Lambda_{\geq})  \ar@/^25pt/[ld]_{(-)_*} \ar@/_10pt/[ld]_{(-)_{!}}  \ar@/_/[rd]_{completion} &
    \\
    \Rep(\overset{\rightarrow}{\bR})=\Mod_{\bR -gr}(\Lambda_{\geq})  \ar@/^/[ru]_{almostification} \ar@/^/[rd]^{completion} & \operatorname{Pers}(\overset{\rightarrow}{\bR})\ar@{=}[u]&\ar@/_/[lu]_{inclusion} \Moda_{\bR-gr, comp} (\widehat{\Lambda}_{\geq}) \ar@/^25pt/[ld]_{(-)_*} \ar@/_10pt/[ld]_{(-)_{!}} \\
    & \Mod_{\bR-gr, comp} (\Lambda_{\geq}) \ar@/^/[lu]^{inclusion}  \ar[ru]_{almostification} &
    }
\end{equation}

\subsection*{Sheaf world diagram}
\begin{equation}\label{diagram:sheaf}
    \xymatrix{
    & 
    \Sh(\bR_{(-\infty,0]})=\Sh_{\geq }(\bR)\ar@/_10pt/[ld]_{inclusion} \ar[rd]_{quotient} &
    \\
    \Rep(\overset{\rightarrow}{\bR}) \ar@/^/[ru]_{sheafification} \ar[rd]_{quotient} & \operatorname{Pers}(\overset{\rightarrow}{\bR})\ar@{=}[u]&\ar@/_/[lu]_{(-)\star \bfk_{[0,\infty)}}\ar@/^18pt/[lu]^{\mathcal{H} om^\star (\bfk_{[0,\infty),-)}} \Sh_{>}(\bR) \ar[ldd]_{inclusion} \\
    \operatorname{PSh(\bR_{(-\infty,0]})} \ar[rd]_{quotient}\ar[u]_{restriction}& \Rep(\arbR)/\{ \text{constants} \}&\\
    &\operatorname{PSh}_{>}(\bR_{(-\infty,0]}):=\operatorname{PSh}(\bR_{(-\infty,0]})/\{ \text{constants} \}
    \ar@/_/[lu]_{(-)\star \bfk_{[0,\infty)}}\ar@/^18pt/[lu]^(.7){\mathcal{H} om^\star (\bfk_{[0,\infty),-)}} \ar@/_/[ruu]_{sheafification}\ar[u]_{restriction} &&
    }
\end{equation}

\subsection*{Module without grading world diagram}
\begin{equation}\label{diagram:Rmodule}
    \xymatrix{
    & \Moda (\Lambda_{\geq})  \ar@/^25pt/[ld]_{(-)_*} \ar@/_10pt/[ld]_{(-)_{!}}  \ar@/_/[rd]_{completion} &
    \\
    \Mod(\Lambda_{\geq})  \ar@/^/[ru]_{almostification} \ar@/^/[rd]^{completion} & &\ar@/_/[lu]_{inclusion} \Moda_{t-comp} (\widehat{\Lambda}_{\geq}) \ar@/^25pt/[ld]_{(-)_*} \ar@/_10pt/[ld]_{(-)_{!}} \\
    & \Mod_{t-comp} (\widehat{\Lambda}_{\geq}) \ar@/^/[lu]^{inclusion}  \ar[ru]_{almostification} &
    }
\end{equation}

\subsection*{$\bR$-equivariant sheaf world diagram}
\begin{equation}\label{diagram:Rsheaf}
    \xymatrix{
    & 
    \Sh^{\bR_\delta}(\bR_{(-\infty,0]})=\Sh^{\bR_\delta}_{\geq }(\bR)\ar@/_10pt/[ld]_{inclusion} \ar[rd]_{quotient} &
    \\
    \Rep(\arbR)^{\bR_\delta} \ar@/^/[ru]_{sheafification} \ar[rd]_{quotient} & &\ar@/_/[lu]_{(-)\star_\bR \bigoplus_{c\in\bR}\bfk_{[c,\infty)}} \ar@/^18pt/[lu]^(.7){\mathcal{H} om^{\star \bR} (\bigoplus_{c\in \bR}\bfk_{[c,\infty)},-)} \Sh^{\bR_\delta}_{>}(\bR) \ar[ldd]_{inclusion} \\
    \operatorname{PSh}^{\bR_\delta}(\bR_{(-\infty,0]}) \ar[dr]_{quotient}\ar[u]_{restriction}&\Rep(\arbR)^{\bR_\delta}/\{ \text{constants}\}  &\\
    & \operatorname{PSh}^{\bR_\delta}_{>}(\bR_{(-\infty,0]})\ar@/_/[lu]_{(-)\star_\bR \bigoplus_{c\in\bR}\bfk_{[c,\infty)}}\ar@/^18pt/[lu]^(.7){\mathcal{H} om^{\star \bR} (\bigoplus_{c\in \bR}\bfk_{[c,\infty)},-)} \ar@/_/[ruu]_{sheafification}\ar[u]_{restriction}&&
    }
\end{equation}

Here $(-)_*$ is the right adjoint of almostification explained in Corollary~\ref{coro: almost local objects}, $(-)_!$ is the left adjoint of almostification given by $\otimes_{\Lambda_{\geq}}I_\geq$, $\star$ and $\star_\bR$ are the convolution and the equivariant convolution functors, respectively, and $\mathcal{H} om^\star$ denotes their right adjoints.

{\protect\setcounter{tocdepth}{2}}

\section{Dualizability}\label{appendix: dualizability}
Recall \autoref{definition: compactly assembled/dualizble category} for dualizable categories. It is proven by Lurie and Clausen (cf.  \cite[Theorem 2.2.15]{Fake_sheaves_on_mfd}) that dualizable categories are $\omega_1$-presentable and can be represented by the following fiber sequence:
\[\cC \xrightarrow{\hat{\cY}} \Ind(\cC^{\omega_1})\rightarrow \Ind(\cC^{\omega_1})/\cC \]
such that $\Ind(\cC^{\omega_1})/\cC\simeq \Ind(\cD)$ is compactly generated for a small category $\cD$ and $\Ind(\cC^{\omega_1})\rightarrow \Ind(\cC^{\omega_1})/\cC$ is a left Bousfield localization. Such functors are called (categorical) homological epimorphisms. The notion of homological epimorphism is a categorical version of $R\rightarrow \cofib{I\rightarrow R}$ for an almost content $(R,I)$. In fact, one can also prove that $\Ind(\cC^{\omega_1})\rightarrow \Ind(\cC^{\omega_1})/\cC$ can be presented by $\Mod_A\rightarrow \Mod_B$ for a homological epimorphism $A\to B$ of ring spectra. 

In this sense, all dualizable categories are categories of almost modules (the converse is also true). We refer to \cite[Theorem 2.10.16]{Fake_sheaves_on_mfd} for a detailed proof of this characterization. However, the proof therein is existential.

From this point of view, \autoref{prop: APT HD} and \autoref{theorem: Novikov mirror symmetry} give explicit constructions for a class of dualizable categories to be presented as categories of almost modules: $\Sh_{\bR^n\times X}(\bR^n)$ and ${\Sh^{\bR^n_\delta}_{\bR^n\times X}(\bR^n)} $ for $X=\sigma$ a closed proper convex cone or $X=|\Sigma|$ the support of a fan $\Sigma$. Here, we give a direct proof for their dualizability.

In general, for a locally compact Hausdorff space $X$, it is proven that $\Sh(X)$ is a dualizable category \cite[Def. 21.1.2.1, Thm. 21.1.6.12, Prop. 21.1.7.1]{SAG}. One can also write down the unit and counit using  $6$-operations (One reference for these formulas is \cite[Proposition 3.1]{Hochschild-Kuo-Shende-Zhang}). 

\begin{Lemma}For a closed cone $X\subset \bR^n$, the categories $\Sh_{\bR^n\times X}$ and $\Sh^{\bR^n_\delta}_{\bR^n\times X}(\bR^n) $ are dualizable.
\end{Lemma}

\begin{proof}
The dualizability of $\Sh_{\bR^n\times X}(\bR^n)$ follows from the fact that it is a reflective subcategory of the dualizable category $\Sh(\bR^n)$. See also \cite[Remark 3.7]{Hochschild-Kuo-Shende-Zhang}. 

For equivariant categories, we argue as follows. The ${\bR^n_\delta}$-action on $\Sh(\bR^n)$ is given by the functor $B{\bR^n_\delta} \rightarrow \PrSt,\quad a\mapsto \nT_{a*}\simeq \nT_{-a}^*$. By microsupport estimation of push-forward, ${\bR^n_\delta}$-acts on $\Sh_{\bR^n\times X}(\bR^n)$ as well. Moreover, since $\Sh_{\bR^n\times X}(\bR^n)$ is dualizable and $ \nT_{a*}$ are strongly continuous (they are actually equivalences), the action functor factors through $B{\bR^n_\delta} \rightarrow \Catdual \rightarrow \PrSt$. 

Therefore, by the equivalence $\PrSt\simeq ({\operatorname{Pr_{st}^R})^{\op}}$, we have that \[{\Sh^{\bR^n_\delta}_{\bR^n\times X}(\bR^n)}  \coloneqq \varprojlim_{B{\bR^n_\delta}} \Sh_{\bR^n\times X}(\bR^n) \simeq \varinjlim_{B{\bR^n_\delta}} \Sh_{\bR^n\times X}(\bR^n) ,\] and ${\Sh^{\bR^n_\delta}_{\bR^n\times X}(\bR^n)}$ is dualizable by \cite[Proposition 1.66]{Efimov-K-theory}. 
\end{proof}

\section{Covariant Verdier duality of microsupport}\label{appendix: cosheaf microsupport}
Here, we explain a generalization of microsupport and discuss a version of comicrosupport of cosheaves.

Recall that for an $\infty$-category $\cC$ that admits small limits and a topological space $X$, we can define the category of $\cC$-valued sheaves $\Sh(X;\cC)$: Presheaves $F:\operatorname{Open}(X)^\op\rightarrow \cC$ such that $F(U)=\varprojlim_i F(U_i)$ for a cover $\{U_i\}$ of $U$, and morphisms are defined as morphisms between underlying presheaves. On the other hand, when $\cC$ admits small colimits, one can define the category of cosheaves: $\coSh(X;\cC)\coloneqq \Sh(X;\cC^\op)^\op$.

When $\cC$ admits both small limits and colimits, we can define both sheaf and cosheaf categories. If in addition $\cC$ is pointed, i.e., admits a zero object and $X$ is locally compact Hausdorff, we can define compactly supported sections \cite[Definition 5.5.5.9.]{HA}: For $K\subset X$ compact, we define $\Gamma_K(X,F)\coloneqq \fib(F(X)\rightarrow F(X\setminus K))$, and $\Gamma_c(U,F)\coloneqq \varinjlim_{K\subset U} \Gamma_K(X,F)$ where $K$ ranges over all compact subsets of $U$. Then the construction $ F_c: U\mapsto F_c(U)\coloneqq \Gamma_c(U,F)$ is a cosheaf, i.e., $F_c\in \coSh(X;\cC)$ \cite[Corollary 5.5.5.12]{HA}. The construction is functorial with respect to $F$ and we have the following powerful theorem from Lurie.
\begin{Thm}[{Covariant Verdier duality \cite[Theorem 5.5.5.1]{HA}}]\label{Thm:Covariant Verdier duality}For a stable category $\cC$ that admits both small limits and colimits, the functor
\[ \bD_{\cC}:\Sh(X;\cC) \rightarrow \coSh(X;\cC), F\mapsto F_c  \]
is an equivalence. The inverse of $\bD_{\cC}$ can be given by $\bD_{\cC^\op}^\op$.
\end{Thm}

Here, we discuss a generalization of microsupport\footnote{The idea is originally due to Beilinson \cite{Topological_epsilon}, and we learned this idea from Marco Volpe.} and discuss its interaction with the covariant Verdier duality.

Now, let $M$  be a smooth manifold. The classical story of microsupport from Kashiwara and Schapira corresponds to the case $\cC=\Mod(\bfk)$ for a discrete commutative ring. This story was extended to any compactly generated stable category by \cite{Amicrolocallemma_infinitycat}. On the other hand, it is noticed that the definition does not behave very well beyond the compactly generated case \cite[Remark 4.24]{Efimov-K-theory}.

A good replacement for the definition needs the lens characterization  \autoref{lemma: Omega-lens} (which is still true for compactly generated stable categories $\cC$) of the non-zero part of microsupport $\dot{\SS}(F )\coloneqq \SS(F)\setminus 0_M$, where the condition (2) therein can be formulated for general categories $\cC$. On the other hand, in the classical story of microsupport, we know that $\SS(F) \cap 0_M=\supp(F)$ where $M\setminus\supp(F)$ is defined as the maximal open set such that $\Gamma(U,F)=\pt$ (here $\pt$ means the terminal object in $\cC$) for all $U\subset M\setminus\supp(F)$. So, we have that $\SS(F)=\dot{\SS}(F ) \cup \supp(F) \subset T^*M$ (we identify $M$ with the zero section of $T^*M$). 

Therefore, we make the following definition. We set $\dot{T}^*M=T^*M\setminus 0_M$.
\begin{Def}For a category $\cC$ that admits small limits and a manifold $M$, and for $F\in \Sh(M;\cC)$, we define $\dot{T}^*M\setminus\dot{\SS}(F )$ as the maximal open set $\Omega \subset \dot{T}^*M $ such that, for any $\Omega$-lens given by a smooth function $g$ as in \autoref{Def: Omega-lens}, we have the restriction morphism
\[\Gamma( \{g_1<0\},F) \rightarrow \Gamma( \{g_0<0\},F)\]
is an equivalence.

Then we define $\SS(F)=\dot{\SS}(F )\cup \supp(F)$. If necessary, we will denote $\SS_\cC(F)$ to emphasize the coefficient category. 
\end{Def}

When $\cC$ admits small colimits, we consider a dual version.
\begin{Def}For a category $\cC$ that admits small colimits and a manifold $M$, and for $G\in \coSh(M;\cC)$, we define $\coSS(G)=\coSS_{\cC}(G)=\SS_{\cC^\op}(G)$, which is a conic closed subset of $T^*M$. We also set $\operatorname{co\dot{\SS}}(G)=\coSS(G)\setminus 0_M$.
\end{Def}
\begin{RMK}In the definition of $\SS_\cC(F)$, only the data of $F$ itself are used, and no morphisms from/to $F$ are needed. So, the definition of $\coSS_{\cC}(G)$ makes sense.

\end{RMK}
\begin{RMK}The main point of these definitions is that sometimes we want to run the microlocal sheaf theory for unstable coefficients. We only discuss the definition here; see \cite{Zhang_K_Remark} for more discussion as well as the existence of microsupport.
\end{RMK}

The main result of this appendix is the following, which gives us an equivalence of $\SS$ and $\coSS$ in the stable situation.
\begin{Prop}Let $M$ be a manifold, and $\cC$ be a \textbf{stable} category that admits both limits and colimits, then we have for $F\in \Sh(M)$ that
\[\dot{\SS}_{\cC}(F)=-\operatorname{co\dot{\SS}}_{\cC}(\bD_{\cC}(F)).\]
\end{Prop}
\begin{proof}Recall that $\bD_{\cC}(F)=F_c$. We first notice that under our assumption, we can use \cite[Corollary 6.13]{6functor-infinity} that enables us to run the argument of \cite[Lemma 3.3]{guillermouviterbo_gammasupport}, which shows that 
\[\operatorname{co\dot{\SS}}_{\cC}(F_c) \subset -\dot{\SS}_{\cC}(F).\]
Notice that we only use one direction of \cite[Lemma 3.3]{guillermouviterbo_gammasupport} shown in the arXiv version. 

On the other hand, by \autoref{Thm:Covariant Verdier duality}, we have 
\[\dot{\SS}_{\cC}(F)= \operatorname{co\dot{\SS}}_{\cC^\op}(\bD_{\cC^\op}^\op F_c) \subset -\dot{\SS}_{\cC^\op}(F_c) =-\operatorname{co\dot{\SS}}_{\cC}(F_c). \qedhere\]    
\end{proof}
\begin{RMK}If we take $\cC=\Mod(\bfk)$, this also gives a slightly different proof of the other direction of \cite[Lemma 3.3]{guillermouviterbo_gammasupport} (which appears only in the published version). 
\end{RMK}

For a conic closed subset $Z$, we define $\Sh_Z(M;\cC)$ and $\coSh_Z(M;\cC)$ as the full subcategories of $\Sh(M;\cC)$ and $\coSh(M;\cC)$ spanned by objects with the microsupport bound $Z$. Then the above result implies the following covariant Verdier duality with microsupport.
\begin{Thm}\label{covariant Verdier duality with microsupport}Let $M$ be a manifold and $Z\subset T^*M$ be a conic closed subset containing the zero section. Let $\cC$ be a \textbf{stable} category that admits both limits and colimits. Then the covariant Verdier duality covariant Verdier duality $\bD_{\cC}$ restricts to an equivalence
\[\Sh_Z(M;\cC) \simeq \coSh_{-Z}(M;\cC).\]    
\end{Thm}

In general, when the corresponding categories can be defined, $\Sh_Z(M;\cC)$ is closed under limits and $\coSh_Z(M;\cC)$ is closed under colimits by definition. \autoref{covariant Verdier duality with microsupport} implies that both of them are closed under limits and colimits if $\cC$ is stable.

 \clearpage
\bibliographystyle{bingyu}

\phantomsection
\bibliography{bibtex}

\noindent
\parbox[t]{28em}
{\scriptsize{
\noindent
Tatsuki Kuwagaki\\
Department of Mathematics, Kyoto University \\
Kitashirakawa Oiwakecho, Sakyo-ku, Kyoto 606-8502, Japan\\
Email: {tatsuki.kuwagaki.a.gmail.com}\\

Bingyu Zhang\\
Department of Mathematics, Kyiv School of Economics\\
3 Mykoly Shpaka St, Kyiv, 02000, Ukraine\\
Email: {bzhang@kse.org.ua}
}}

\end{document}